\documentclass[11pt,reqno]{amsart}
\pdfoutput=1
\usepackage[utf8]{inputenc}

\usepackage[dvips]{graphicx}
\usepackage{multicol}

\usepackage{amssymb, amsmath, amsthm, pstricks}
\usepackage{bbm} 
\usepackage[normalem]{ulem} 

\usepackage{stmaryrd} 

\usepackage{tikz-cd}

\makeatletter
\@namedef{subjclassname@2020}{\textup{2020} Mathematics Subject Classification}
\makeatother

\usepackage[hyphens]{url}

\usepackage{hyperref}

\usepackage{enumitem}

\newcommand{\ie}{{\it i.e.}}
\newcommand{\eg}{{\it e.g.}}

\newcommand{\term}{\emph} 

\DeclareMathOperator\Int{Int} 
\DeclareMathOperator\length{length}
\DeclareMathOperator\area{area}
\DeclareMathOperator\vol{vol}
\DeclareMathOperator\diam{diam}
\DeclareMathOperator\Prob{\mathbb P}

\DeclareMathOperator\E{\mathbb E}
\newcommand{\N}{{\mathbb N}}
\newcommand{\Z}{{\mathbb Z}}

\newcommand{\R}{{\mathbb R}}

\newcommand{\T}{{\mathbb T}}

\newcommand{\eps}{\varepsilon} 
\newcommand{\W}{\mathcal{W}} 

\newcommand{\HH}{\mathcal{P}} 

\newcommand{\MLS}{{\rm MLS}}
\newcommand{\card}{{\rm card}}
\newcommand{\Curr}{{\rm Curr}}
\newcommand{\Gr}{{\rm Gr}}
\newcommand{\sys}{{\rm sys}}
\newcommand{\mass}{{\rm mass}}

\newcommand{\mres}{\mathbin{\vrule height 1.6ex depth 0pt width
0.13ex\vrule height 0.13ex depth 0pt width 1.3ex}}

\long\def\forget#1\forgotten{}

\numberwithin{equation}{section}
\newtheorem{theorem}{Theorem}[section]
\newtheorem{proposition}[theorem]{Proposition}
\newtheorem{corollary}[theorem]{Corollary}
\newtheorem{lemma}[theorem]{Lemma}

\theoremstyle{definition}
\newtheorem{definition}[theorem]{Definition}
\newtheorem{example}[theorem]{Example}
\newtheorem{remark}[theorem]{Remark}

\newtheorem{question}[theorem]{Question}

\begin{document}
\title[Marked length spectrum and area of surfaces]{Marked length spectrum and area of surfaces with Anosov geodesic flow}

\author[M.~Cossarini]{Marcos Cossarini}
\author[S.~Sabourau]{St\'ephane Sabourau}

\thanks{Partially supported by the B\'ezout Labex (ANR-10-LABX-58) and the ANR project Min-Max (ANR-19-CE40-0014).}

\address{\parbox{\linewidth}{Univ Paris Est Creteil, CNRS, LAMA, F-94010 Creteil, France \\
Univ Gustave Eiffel, LAMA, F-77447 Marne-la-Vall\'ee, France}}

\email{marcos.cossarini@u-pec.fr}

\email{stephane.sabourau@u-pec.fr}

\subjclass[2020]
{Primary 53C23; Secondary 53C60, 53C65}

\keywords{Marked length spectrum, Anosov flow, geodesic current, intersection form, integral geometric formulas, wall system, metric discretization, discrete geometry}

\begin{abstract}
More than twenty years ago, Croke and Dairbekov showed that a negatively curved metric~$g_0$ on a closed surface has minimal area among all negatively curved Riemannian metrics~$g$ on the same surface whose marked length spectrum is greater or equal to that of~$g_0$.
We show that the curvature assumption on the second metric can be removed and extend this result to reversible Finsler metrics.
More precisely, we prove that a reversible Finsler metric~$F_0$ on a closed surface whose geodesic flow is Anosov has minimal area among all reversible Finsler metrics~$F$ on the same surface whose marked length spectrum is greater or equal to that of~$F_0$.
\end{abstract}

\maketitle


\section{Introduction}

Let $M$ be a closed manifold with a Finsler metric~$F$.
Unless specified otherwise, all Finsler metrics are assumed to be reversible.
The \emph{marked length spectrum} of~$(M,F)$ is the function $\MLS_F:\mathcal{C} \to [0,\infty)$ defined on the set~$\mathcal{C}$ of unoriented free homotopy classes of~$M$ as
\begin{equation} \label{eq:MLS}
\MLS_F(\alpha) = \inf_{c \in \alpha} \length_F(c)
\end{equation}
where the infimum is taken over all closed curves~$c$ representing the unoriented free homotopy class~$\alpha$.
For every nontrivial class~$\alpha$, this infimum is attained by a closed geodesic of least length representing~$\alpha$.

\medskip

The marked length spectrum rigidity conjecture, as stated by Burns and Katok~\cite{BK85}, asserts that a negatively curved Riemannian metric on a closed manifold is determined, up to isometry, by its marked length spectrum.
For closed surfaces, this conjecture was independently proved by Croke~\cite{Cro} and Otal~\cite{Otal}; see also~\cite{Fathi} and~\cite{CFF}.
Very recently, Guillarmou, Lefeuvre and Paternain~\cite{GLP} extended this result to Riemannian metrics on closed surfaces whose geodesic flow is Anosov.
In higher dimensions, Hamenst\"adt~\cite{Ham99} proved the conjecture for closed negatively curved locally symmetric spaces, relying on the volume entropy rigidity theorem of Besson, Courtois and Gallot~\cite{BCG}.
A local version of the conjecture was also recently proved by Guillarmou and Lefeuvre~\cite{GL}; see also~\cite{GKL19} and~\cite{CL}.
The stability of the marked length spectrum was further investigated in~\cite{butt1} and~\cite{butt2}.
Related rigidity questions arise in various contexts, including boundary rigidity, lens rigidity and spectral rigidity.
We will not discuss these problems here and refer to the monographs~\cite{PSU} and~\cite{GM} for accounts of the subject. 

\medskip

In this article, we consider metrics whose marked length spectra are not necessarily equal, but can be compared as follows.
Let~$F_1$ and~$F_2$ be two Finsler metrics on a closed surface~$M$.
The marked length spectrum of~$F_1$ is said to be less or equal to that of~$F_2$ if~$\MLS_{F_1}(\alpha) \leq \MLS_{F_2}(\alpha)$ for every unoriented free homotopy class~$\alpha$.
In this case, we write $\MLS_{F_1} \leq \MLS_{F_2}$. 

\medskip

More than twenty years ago, Croke and Dairbekov~\cite{CD} proved the following comparison theorem relating the marked length spectra and areas of negatively curved surfaces.
(See Theorem~\ref{theo:extension} for an extension of this result where the marked length spectra need only be compared on free homotopy classes represented by null-homologous closed curves.)

\begin{theorem}[\cite{CD}] \label{theo:CD}
Let $M$ be a closed surface with two negatively curved Riemannian metrics~$g_0$ and~$g$.
If $\MLS_{g_0} \leq \MLS_g$ then 
\[
\area(g_0) \leq \area(g).
\]
\end{theorem}

We show that the curvature assumption on the competitor metric~$g$ can be removed.
More generally, we extend this comparison theorem to reversible Finsler metrics, replacing the negative curvature assumption on the reference metric by the assumption that its geodesic flow is Anosov.
In the Finsler setting, area is understood to be the Holmes--Thompson area.
More precisely, we prove the following theorem.

\begin{theorem} \label{theo:main}
Let $M$ be a closed surface endowed with a Finsler metric~$F_0$ whose geodesic flow is Anosov.
Let $F$ be any Finsler metric on~$M$.
If $\MLS_{F_0} \leq \MLS_F$ then 
\[
\area(F_0) \leq \area(F).
\]
\end{theorem}

In contrast with Theorem~\ref{theo:CD}, where equality between the areas of~$g_0$ and~$g$ implies that the metrics are isometric, such a rigidity statement fails for Finsler metrics.
Indeed, given any Finsler metric~$F_0$ on a closed surface~$M$, whether or not its geodesic flow is Anosov, one can construct by deformation an infinite-dimensional family of Finsler metrics, modulo isometries, with
the same marked length spectrum and the same area as~$F_0$; see Section~\ref{sec:deformation} for further details.

\medskip

The proof proceeds by approximating the Liouville geodesic currents of the Finsler metrics by wall systems. 
This provides discrete length structures whose marked length spectra and areas are arbitrarily close to those of the Finsler metrics. 
The desired area inequality then follows from a discrete version of the theorem.
 An outline of the proof is given in Section~\ref{sec:outline}.

\medskip

For completeness, we also state the analogue of our main result for the two-torus, which follows immediately from the work of Burago and Ivanov~\cite{BI} on the stable norm of Finsler tori.

\begin{theorem} \label{theo:torus}
Let $\T^2$ be the two-torus endowed with a Finsler metric~$F_0$ without conjugate points.
Let $F$ be any Finsler metric on~$\T^2$.
If $\MLS_{F_0} \leq \MLS_F$ then 
\[
\area(F_0) \leq \area(F).
\]
\end{theorem}

The space of Finsler metrics without conjugate points on the two-torus is infinite-dimensional.
Indeed, any sufficiently small neighborhood of any point of a Finsler surface can be isometrically embedded into a Finsler two-torus without conjugate points; see~\cite{chen}.
Its homotopy type, with respect to the topology of uniform convergence, was recently determined in~\cite{NS}. 

\medskip

We plan to extend these results to compact surfaces with boundary in a forthcoming paper.

\medskip

This article is organized as follows. 
In Section~\ref{sec:outline}, we give an outline of the proof of the main theorem. 
In Section~\ref{sec:Anosov}, we recall the basic definitions and properties of Finsler metrics and Anosov geodesic flows. 
In Section~\ref{sec:current}, we review the theory of geodesic currents and the intersection form, and establish the length and area formulas for the Liouville geodesic current of Finsler metrics. 
In Section~\ref{sec:min}, we study minimal intersection numbers of closed curves on Finsler surfaces. 
In Section~\ref{sec:wall}, we introduce wall systems and the corresponding notions of discrete length, marked length spectrum, and area. 
In Section~\ref{sec:uncrossing}, we show that wall systems can be uncrossed and straightened into minimal geodesic wall systems while preserving their marked length spectrum and without increasing their area. 
This leads in Section~\ref{sec:discrete} to a discrete version of the main theorem.
In Section~\ref{sec:approximation_anosov}, we present approximation results for the Liouville geodesic current of Finsler metrics with Anosov geodesic flow, based on results of Sigmund and of Coles and Sharp.
In Section~\ref{sec:generic}, we establish a generic approximation result for arbitrary Finsler metrics using a theorem of Irie (more precisely, a refined version due to Liu and Liu).
In Section~\ref{sec:discretization}, we use these results to approximate the marked length spectrum and area of Finsler metrics by wall systems. 
We then prove the main theorem in Section~\ref{sec:proof}. 
In Section~\ref{sec:torus}, we discuss the corresponding result for the two-torus, and in Section~\ref{sec:deformation}, we construct Finsler deformations preserving both the marked length spectrum and the area. 
We conclude in Section~\ref{sec:open} with an open question. 

\medskip

\emph{AI disclosure:}
An AI tool was used during the editing stage to improve the presentation and identify several inaccuracies in the proofs. 
Interestingly, it suggested more direct alternative proofs for Propositions~\ref{prop:bonahon_length} and~\ref{prop:bonahon_area}.
We presented them for comparison.

\medskip

\emph{Acknowledgements:}
The authors would like to thank B\'ezout Labex (ANR-10-LABX-58), the ANR project Min-Max (ANR-19-CE40-0014) and Universit\'e Paris-Est Cr\'eteil for funding the first author's postdoctoral positions in 2019, 2023 and 2024--2025, during which the main theorem was proved.

\section{Outline of the proof of the main theorem} \label{sec:outline}

Without loss of generality, after passing to the orientable double cover if necessary, we assume that all surfaces considered in the remainder of the article are orientable.

\medskip

Before outlining the proof of our main theorem, let us go over Croke and Dairbekov's elegant argument for Theorem~\ref{theo:CD}.
Their proof relies on Bonahon's theory of geodesic currents, the properties of the intersection form on geodesic currents~\cite{Bon88}, and the density of positive multiples of closed geodesics in the space of geodesic currents~\cite{Sig}.
It can be summarized as follows.
By the orbit equivalence theorem for negatively curved metrics~\cite{Gro00} (or Theorem~\ref{theo:Anosov}.\eqref{Anosov3} for Finsler metrics whose geodesic flow is Anosov), the geodesic flows of~$g_0$ and~$g$ can be identified.
Using the intersection formula expressing the length of a closed geodesic~$\gamma$ in terms of its intersection with the Liouville geodesic current, the inequality
\[
\length_{g_0}(\gamma) \leq \length_g(\gamma)
\]
can be written as
\[
i(\gamma,\lambda_{g_0}) \leq i(\gamma,\lambda_{g})
\]
where $\lambda_{g_0}$ and~$\lambda_g$ are the Liouville geodesic currents of~$g_0$ and~$g$.
By Sigmund's density theorem for positive multiples of geodesic currents induced by closed geodesics~\cite{Sig}, together with the continuity of the intersection form, it follows that
\[
i(\mu,\lambda_{g_0}) \leq i(\mu,\lambda_{g})
\]
for every geodesic current~$\mu$.
Applying this inequality first with~$\mu=\lambda_{g_0}$ and then with~$\mu=\lambda_g$, and using the symmetry of the intersection form, we obtain
\[
i(\lambda_{g_0},\lambda_{g_0}) \leq i(\lambda_{g_0},\lambda_g) = i(\lambda_g,\lambda_{g_0}) \leq i(\lambda_g,\lambda_g).
\]
Finally, the self-intersection formula for the Liouville geodesic current expresses the area of a negatively curved surface in terms of~$i(\lambda_g,\lambda_g)$.
Therefore,
\[
\area(g_0)\leq\area(g),
\]
which proves Theorem~\ref{theo:CD}. 

\medskip

We can now give an outline of the proof of the main theorem.
The main obstacle to applying the previous argument to an arbitrary competitor metric~$F$ is that its geodesic flow need not be Anosov.
In particular, it is not necessarily orbit equivalent to the geodesic flow of~$F_0$, and we cannot use Sigmund's density theorem to approximate its Liouville geodesic current by closed geodesics.
Our strategy is to approximate Finsler metrics by discrete length structures induced by wall systems, following an approach initiated in~\cite{Cos} and further developed in~\cite{CS}.

Loosely speaking, a wall system on a surface is a finite collection of transverse closed curves with positive weights.
It determines a discrete length structure by counting intersections with the walls, with multiplicity and weights, and a discrete area by counting intersections and self-intersections between the walls.
A geodesic wall system can also be viewed as a geodesic current, namely as a positive linear combination of the geodesic currents induced by its walls.
The connection between the continuous and discrete settings is provided by the intersection form on geodesic currents.
For an arbitrary Finsler metric~$F$ with Liouville geodesic current~$\lambda$, we prove the formulas
\[
i(\gamma,\lambda) = \length_F(\gamma) \qquad \textrm{and} \qquad i(\lambda,\lambda)=\frac{\pi}{2} \area(F)
\]
initially established by Bonahon for hyperbolic metrics and later extended to metrics without conjugate points.
Similarly, if~$\W$ is a geodesic wall system and~$\gamma$ is a closed geodesic transverse to~$\W$, then
\[
i(\gamma,\W) = \length_\W(\gamma) \qquad \textrm{and} \qquad i(\W,\W) = 2 \area(\W).
\]
Moreover, if~$\W$ is a minimal geodesic wall system and~$\gamma$ is a closed geodesic of minimal length in its free homotopy class~$\alpha$, then
\[
\MLS_\W(\alpha) = i(\gamma,\W).
\]
Thus, the Liouville geodesic current of a Finsler metric and a geodesic wall system play analogous roles.

The first ingredient of the proof is a discrete comparison theorem.
Let~$\W_0$ and~$\W$ be two wall systems, with~$\W_0$ minimal.
If 
\[
\MLS_{\W_0} \leq \MLS_{\W}
\]
then
\[
\area(\W_0) \leq \area(\W).
\]
To prove this, we first uncross and straighten~$\W$ into a minimal geodesic wall system~$\overline{\W}$ without changing its discrete marked length spectrum or increasing its discrete area.
This step relies on a theorem of de Graaf and Schrijver in graph theory.
Both~$\W_0$ and~$\overline{\W}$ can then be viewed as geodesic currents.
The discrete marked length spectrum inequality translates into
\[
i(\gamma,\W_0)\leq i(\gamma,\overline{\W})
\]
for every closed geodesic~$\gamma$ minimizing in its free homotopy class.
Applying this inequality successively to the walls of~$\W_0$ and of~$\overline{\W}$, and using the symmetry of the intersection form as in Croke--Dairbekov's argument, gives
\[
2 \area(\W_0) = i(\W_0,\W_0) \leq i(\W_0,\overline{\W}) = i(\overline{\W},\W_0) \leq i(\overline{\W},\overline{\W}) = 2 \area(\overline{\W}).
\]
Therefore,
\begin{equation} \label{eq:aaa}
\area(\W_0)\leq\area(\overline{\W})\leq\area(\W).
\end{equation}

The second ingredient consists of approximating Finsler metrics by wall systems.
The general principle underlying these approximations is that an approximation of the Liouville geodesic current by positive linear combinations of geodesic currents induced by closed geodesics gives rise to approximations of the marked length spectrum and the area.
This follows from the length and area formulas expressing both the marked length spectrum and the area in terms of intersections with the Liouville geodesic current.
The compactness of the projectivized space of geodesic currents is used to make the approximation of the marked length spectrum uniform over all free homotopy classes.
At this point, we should stress that, for arbitrarily Finsler metrics, the intersection form need not be continuous everywhere, but it is still continuous if one of the geodesic currents is the Liouville geodesic current.
To apply this principle, we make use of a generic approximation theorem of Irie for Reeb flows (more precisely a refined version due to Liu and Liu).
It implies that every Finsler metric can be approximated arbitrarily closely by a reversible Finsler metric whose Liouville geodesic current can itself be approximated by positive linear combinations of geodesic currents induced by closed geodesics.
We apply this result to both~$F_0$ and~$F$.
Thus, after arbitrarily small perturbations to reversible Finsler metrics~$F'_0$ and~$F'$, we obtain $F'_0$- and $F'$-geodesic wall systems~$\W_0$ and~$\W$ such that
\[
\MLS_{\W_0} \simeq \MLS_{F_0} \qquad \textrm{and} \qquad \area(\W_0)  \simeq \frac{\pi}{4} \area(F_0)
\]
as well as
\[
\MLS_{\W} \simeq \MLS_{F} \qquad \textrm{and} \qquad \area(\W)  \simeq \frac{\pi}{4} \area(F).
\]
Now, since the geodesic flow of~$F_0$ is Anosov and the Anosov property is structurally stable, the perturbation~$F'_0$ may be chosen sufficiently small so that its geodesic flow is still Anosov.
In particular, every closed~$F'_0$-geodesic is length-minimizing in its free homotopy class, and hence the geodesic wall system~$\W_0$ is minimal.

Using these marked length spectrum and area approximations, together with the previous discrete comparison theorem, see~\eqref{eq:aaa}, we conclude the proof of the main theorem.

\forget
The second ingredient consists of approximating Finsler metrics by wall systems.
The general principle underlying these approximations is that an approximation of the Liouville geodesic current by positive linear combinations of geodesic currents induced by closed geodesics gives rise to approximations of the marked length spectrum and the area.
This follows from the length and area formulas expressing both the marked length spectrum and the area in terms of intersections with the Liouville geodesic current.
The compactness of the projectivized space of geodesic currents is used to make the approximation of the marked length spectrum uniform over all free homotopy classes.
At this point, we should stress that, for arbitrarily Finsler metrics, the intersection form need not be continuous everywhere, but it is still continuous if one of the geodesic currents is the Liouville geodesic current.
We apply this approximation principle to both Finsler metrics. 
If the geodesic flow of~$F_0$ is Anosov, Sigmund's density theorem allows us to approximate its Liouville geodesic current by positive multiples of geodesic currents induced by closed geodesics.
This yields a minimal geodesic wall system~$\W_0$ such that
\[
\MLS_{\W_0} \simeq \MLS_{F_0} \qquad \textrm{and} \qquad \area(\W_0)  \simeq \frac{\pi}{4} \area(F_0).
\]
For the arbitrary Finsler metric~$F$, we instead make use of a generic approximation theorem of Irie for Reeb flows (more precisely, a refined version due to Liu and Liu).
The resulting approximation theorem shows that, after an arbitrarily small perturbation of~$F$, its Liouville geodesic current can be approximated by positive linear combinations of geodesic currents induced by closed geodesics.
Applying the same principle yields an almost geodesic wall system~$\W$ such that
\[
\MLS_{\W} \simeq \MLS_F \qquad \textrm{and} \qquad \area(\W) \simeq \frac{\pi}{4} \area(F).
\]

Using these marked length spectrum and area approximations, together with the previous discrete comparison theorem, see~\eqref{eq:aaa}, we conclude the proof of the main theorem.

\forgotten

\section{Finsler metrics and Anosov geodesic flow} \label{sec:Anosov}

We recall basic definitions of Finsler geometry and dynamics; see~\cite{Pat}.

\begin{definition} \label{def:finsler}
A (reversible) \emph{Finsler metric} on a manifold~$M$ is a continuous function $F:TM \to [0,\infty)$ on the tangent bundle~$TM$ of~$M$ satisfying the following properties (here, $F_x:=F|_{T_x M}$ for short):
\begin{enumerate}
\item Smoothness: $F$ is smooth outside the zero section; \label{item:smooth}
\item Homogeneity: $F_x(tv) = |t| \, F_x(v)$ for every $v \in T_x M$ and $t \in \R$; \label{item:homogene}
\item Quadratic convexity: for every $x \in M$, the function~$F_x^2$ has positive definite second derivatives on $T_x M \setminus \{ 0\}$, that is, for every~\mbox{$p\in T_{x}M\setminus\{0\}$}, and $u,v \in T_x M$, the symmetric bilinear form
\[
g_p(u,v) = \frac{1}{2} \, \frac{\partial^2}{\partial s \partial t} F_x^2(p+tu+sv)|_{(s,t)=(0,0)}
\]
is positive definite. \label{item:convex}
\end{enumerate}
Unless specified otherwise, all Finsler metrics are assumed reversible.

The \emph{$F$-length} of a smooth curve~$\gamma:[a,b] \to M$ is defined by
\[
\length_F(\gamma) = \int_a^b F(\gamma'(t)) \, dt
\]
and the distance $d_F(p,q)$ between two points $p,q \in M$ is the infimum $F$-length of the curves joining~$p$ and~$q$.
\end{definition}

\forget
Consider the \emph{Legendre transform}
\begin{equation} \label{eq:Legendre}
\mathcal{L}:TM\rightarrow T^{*}M
\end{equation}
of the Lagrangian~$\frac{1}{2}F^{2}$ between the tangent and cotangent bundles of~$M$; see~\cite{Be78}. 
Since $F$ is quadratically convex, the Legendre transform~$\mathcal{L}$ is a diffeomorphism between $TM \setminus\{0\}$ and $T^{*}M\setminus\{0\}$. 
By the homogeneity of~$F$, it preserves the norm on each fiber of the bundle vectors~$TM$ and~$T^{*}M$. 
In addition, it induces a diffeomorphism between the unit tangent and cotangent bundles~$UM$ and~$U^{*}M$ of~$M$. 
Geometrically, this diffeomorphism is defined as follows: For every vector $v\in U_{x}M$, the image~$\mathcal{L}(v)$ is the unique co-vector of~$U_{x}^{*}M$ such that $\mathcal{L}(v)(v)=1$.

The quadratically convex condition (as opposed to a mere convex condition) also allows us to define the \emph{geodesic and cogeodesic flows} on~$UM$ and~$U^*M$; see~\cite{Be78}.
Note that both flows are conjugate by the Legendre transform.
They will be denoted by $\phi_t:UM \to UM$ and $\phi_t:U^*M \to U^*M$.
\forgotten


We now recall the Hamiltonian description of the geodesic flow associated with a Finsler metric; see~\cite{Pat}.

\begin{definition} \label{def:cogeodesic}
Let $M$ be a manifold with a Finsler metric.
The \emph{tautological one-form}~$\alpha=\alpha_M$ on the cotangent bundle~$T^*M$ of~$M$ is defined as
\[  
(\alpha_M)_\xi(V) = \xi (d\pi_\xi(V))
\]
for every $\xi \in T^*M$ and $V \in T_\xi T^*M$, where $\pi:T^*M \to M$ is the canonical projection.
Similarly, the \emph{standard symplectic form}~$\omega = \omega_M$ on~$T^*M$ is given by
\[  
\omega_M = d \alpha_M.
\]
Note that neither~$\alpha_M$ nor~$\omega_M$ depends on the Finsler metric on~$M$.

The restriction~$\alpha_F$ of~$\alpha_M$ to the unit cotangent bundle~$U^*M$ is a contact form.
Its Reeb vector field~$Z$, referred to as the \term{geodesic vector field}, is characterized by
\begin{align*}
\iota_Z(d\alpha_F) &= 0\\
\iota_Z(\alpha_F) &= 1
\end{align*}
where $\iota_Z$ denotes the contraction with the vector field~$Z$.
The \emph{geodesic flow} of~$M$, denoted by $\phi_t:U^*M \to U^*M$, is the flow generated by~$Z$.
Its orbits project to the geodesics of~$M$.
It follows from the Cartan formula that the forms~$\alpha_F$ and~$d\alpha_F$ are invariant under the geodesic flow on~$U^*M$. 

As usual, the unit tangent bundle~$UM$ identifies with the unit cotangent bundle~$U^*M$ under the Legendre transform $\mathcal{L}:UM \to U^*M$ defined as the diffeomorphism sending every vector $v \in U_xM$ to the unique co-vector~$\mathcal{L}(v) \in U_{x}^{*}M$ such that $\mathcal{L}(v)(v)=1$.
\end{definition}

We will consider the following notion of volume.

\begin{definition}
The \emph{Holmes--Thompson volume} of a Finsler $n$-manifold~$M$ is defined as the symplectic volume of its unit co-ball bundle~$B^*M\subseteq T^*M$, divided by the volume~$\epsilon_n$ of the Euclidean unit ball in~$\R^n$. 
That is,
\[
\vol(M)=\frac 1{\epsilon_n}\int_{B^*M}\textstyle{\frac 1{n!}} \, \omega_M^n 
\]
where $\omega_M$ is the standard symplectic form on~$T^*M$ and $\frac 1{n!} \, \omega_M^n=\frac 1{n!} \, \omega_M \wedge \dots \wedge \omega_M$ is the corresponding volume form.
The factor~$\frac 1{\epsilon_n}$ ensures that for Riemannian metrics, the Holmes--Thompson definition of volume agrees with the conventional Riemannian definition.
\end{definition}

Let us recall the notion of Anosov geodesic flow.

\begin{definition} \label{def:Anosov}
Let $M$ be a closed manifold with a Finsler metric.
The geodesic flow $\phi_t:U^*M \to U^*M$ is \emph{Anosov} if there is a splitting of the tangent bundle~$TU^*M$ of~$U^*M$ into three flow-invariant subbundles
\[
TU^*M = E_s \oplus E_u \oplus E_0
\]
such that $d{\phi_t}_{\mid E_s}$ is exponentially contracting, $d{\phi_t}_{\mid E_u}$ is exponentially expanding and $E_0$ is generated by the vector field of the geodesic flow; see~\cite[\S2.4.2]{Pat} for further details.
\end{definition}

\begin{example}
\mbox{ }
\begin{enumerate}
\item The geodesic flow of every closed Riemannian manifold with negative curvature is Anosov; see~\cite{anosov}.
\item There exist embedded closed surfaces in~$\R^3$ with an Anosov geodesic flow; see~\cite{DP}.
\end{enumerate}
\end{example}

We recall the following classical properties of Finsler metrics with Anosov geodesic flow; see~\cite{PP}, \cite[Theorem~29]{CIPP}, \cite{ghys} and~\cite{anosov}.

\begin{theorem} \label{theo:Anosov}
Let $M$ be an orientable closed surface with a Finsler metric~$F$ whose geodesic flow is Anosov.
Then,
\begin{enumerate}
\item the metric~$F$ has no conjugate points; \label{Anosov1}
\item every nontrivial unoriented free homotopy class of~$M$ contains a unique closed geodesic, up to reparametrization; \label{Anosov2}
\item the geodesic flow of~$F$ is topologically equivalent to the geodesic flow of any hyperbolic metric on~$M$; \label{Anosov3}
\item every sufficiently small $C^1$-perturbation of the geodesic flow of~$F$ is Anosov. \label{Anosov4}
\end{enumerate}
\end{theorem}

\begin{proof}
The first two assertions and the last one hold in all dimensions and follow from~\cite{PP}, \cite[Theorem~29]{CIPP} and~\cite{anosov}.

The third assertion is stated in every dimension for Riemannian metrics in~\cite{ghys} and it seems that the proof also works with Finsler metrics.
The assertion can also be derived from the classification~\cite{ghys} of general Anosov flows on circle bundles over surfaces, up to finite covers.
Indeed, let $F_0$ be a hyperbolic metric on~$M$.
Denote by~$N=U^*_F M$ and $N_0=U^*_{F_0} M$ the corresponding unit cotangent bundles.
By~\cite{ghys}, there exists a finite cover $p: N_0' \to N_0$ such that $N'_0$ is diffeomorphic to~$N$ and the $F$-geodesic flow on~$N$ is topologically equivalent to the lift to~$N_0'$ of the $F_0$-geodesic flow on~$N_0$.
We claim that $p$ has degree one.

The finite cover~$p$ induces a commutative diagram
\[
\begin{tikzcd}
N_0' \ar[d] \ar[r,"p"] & N_0 \ar[d] \\
M' \ar[r] & M
\end{tikzcd}
\]
where the vertical maps are circle bundles and the bottom horizontal map is a finite cover.
By~\cite[\S10.3.7]{martelli}, the degree~$d$ of~$p$ decomposes as~$d = d_h d_v$, where $d_h$ is the degree of the cover~$M' \to M$ and $d_v$ is the degree with which a fiber of~$N_0'$ covers a fiber of~$N_0$.
Moreover, by~\cite[Proposition~10.3.22]{martelli}, we have the relations $\chi(M') = d_h \, \chi(M)$ and $e(N'_0) = \frac{d_h}{d_v} \, e(N_0)$ between Euler characteristics and Euler numbers.
Observe that $N_0'$ and $N_0$ are diffeomorphic since they are both diffeomorphic to~$N=U^*_F M$.
By~\cite[Corollary~10.2.10]{martelli}, their base surfaces have the same genus and $|e(N'_0)| = |e(N_0)|$.
Since $e(N_0) = \chi(M) \neq 0$, it follows that $d_h=1$ and $d_v=1$.
Thus, the cover $p:N_0' \to N_0$ is of degree~$1$, and hence is a diffeomorphism.
Therefore, the geodesic flows of~$F$ and~$F_0$ are topologically equivalent.
\end{proof}

We conclude this section by recalling the following classical result for reference; see~\cite[Lemma~4.2]{gro99}.

\begin{proposition} \label{prop:classical}
Let $M$ be an orientable closed surface endowed with an intrinsic (pseudo)-distance~$d$.
Let $\gamma$ be a closed curve of~$M$ of minimal length in its free homotopy class.
Then every iterate~$\gamma^k$ of~$\gamma$ is also length-minimizing in its free homotopy class.

In particular,
\[
\MLS_d(\alpha^k) = k \, \MLS_d(\alpha)
\]
for every unoriented free homotopy class~$\alpha$ of~$M$.
\end{proposition}

\section{Geodesic currents and intersection form} \label{sec:current}

In this section, we recall the definition of geodesic currents as holonomy-invariant transverse measures for the geodesic foliation and introduce their intersection form. 
Working directly with transverse measures allows us to deal with arbitrary Finsler metrics, without any curvature or conjugate-points assumption. 
We then establish the basic properties of the Liouville geodesic current, including its intersection with closed geodesics and with itself. 
Finally, we introduce the notion of filling geodesic currents and prove the compactness of a normalized space of currents, from which the compactness of the projectivized space of geodesic currents follows.
All these classical results are due to Bonahon for hyperbolic metrics; see~\cite{Bon86} and~\cite{Bon88}.
The proofs we present apply to general Finsler metrics.

\medskip

Let us recall the notion of geodesic currents introduced by Bonahon~\cite{Bon86}.
See also~\cite{Bon88}, \cite{AL} and~\cite{EU} for further details.
In the following definition, we work on the projective cotangent bundle~$PT^*M$ rather than on the projective tangent bundle~$PTM$ as in the original definition.
The two formulations are equivalent via the Legendre transform.

\begin{definition}
Let $M$ be a closed surface endowed with a Finsler metric.
Let $\mathcal{F}$ be the geodesic foliation on the projective cotangent bundle~$PT^*M$.
A \emph{geodesic current} of~$M$ is a holonomy-invariant transverse measure for the geodesic foliation~$\mathcal{F}$ of~$PT^*M$.
More precisely, a geodesic current~$\mu$ assigns to every immersed surface $q:\Sigma \to PT^*M$ transverse to~$\mathcal{F}$ a Radon measure~$q^*\mu$ on~$\Sigma$ (also denoted by~$\mu_\Sigma$ when $q$ is the inclusion map), and this assignment is invariant by holonomy, that is, $q'^* \mu=q^* \mu$ for any other immersion $q':\Sigma \to PT^*M$ homotopic to~$q$ by a homotopy along the leaves of~$\mathcal{F}$.


The space of geodesic currents of~$M$, denoted by~$\Curr(M)$, is endowed with the vague topology.
It is the coarsest topology for which the map
\[
\mu \mapsto \int_\Sigma f \, d\mu_\Sigma
\]
is continuous for every compactly supported continuous function~$f$ on every transverse surface~$\Sigma \subseteq PT^*M$.
Furthermore, the space of geodesic currents is metrizable; see Proposition~\ref{prop:metrizable}.
\end{definition}

\begin{remark}
Alternatively, a geodesic current can be viewed as an invariant transverse measure to the geodesic foliation of~$U^*M$ which is invariant under the antipodal map $\xi \mapsto -\xi$ of~$U^*M$.
This point of view will be useful in the definition of the Liouville geodesic current; see Definition~\ref{def:Liouville_current}.
\end{remark}

We have the following fundamental example of geodesic currents.

\begin{example}
Every closed geodesic~$\gamma$ induces a geodesic current~$\mu_\gamma$ such that for every surface~$\Sigma \subseteq PT^*M$ transverse to~$\mathcal{F}$, the measure~$\mu_{\gamma,\Sigma}$ induced by~$\mu_\gamma$ on~$\Sigma$ is the sum of the Dirac measures supported at the cotangent lines associated with~$\gamma$ under the Legendre transform, counted with multiplicity.
More precisely, the current~$\mu_\gamma$ is defined by
\[
q^*\mu_\gamma(A) = \card \{ (\sigma, t) \in A \times S^1 \mid q(\sigma) = \hat{\gamma}(t) \},
\]
for any immersed surface $q:\Sigma \to PT^*M$ tranverse to the geodesic flow and any Borel set $A \subseteq \Sigma$.
Here, the curve~$\hat{\gamma}$ denotes the lift of~$\gamma$ to~$U^*M$.
In practice, we will also denote~$\mu_\gamma$ by~$\gamma$.
\end{example}

We will also need the notion of Liouville geodesic current.
We present here its definition in terms of invariant transverse measures.
Alternative definitions have been introduced in~\cite{Bon88} in terms of measures on the space of geodesic lines in the universal cover (often identified through their endpoints at infinity).
However, these descriptions are not suitable in the presence of conjugate points when the space of geodesic lines is no longer independent of the metric.

\begin{definition} \label{def:Liouville_current}
Let $M$ be a closed surface endowed with a Finsler metric~$F$.
The \emph{Liouville geodesic current}~$\lambda$ of~$M$ is defined so that on every surface~$\Sigma \subseteq PT^*M$ transverse to the geodesic foliation
\[
\lambda_{\Sigma} = \tfrac{1}{2} \left| \omega|_{\bar{\Sigma}} \right|
\]
where $\omega$ is the restriction to~$U^*M$ of the canonical symplectic form of~$T^*M$ and $\bar{\Sigma}$ is either lift of~$\Sigma$ to~$U^*M$.
Since $\omega$ is invariant under the geodesic flow, the transverse measure~$\lambda$ is invariant by holonomy.
Moreover, since the antipodal map of~$U^*M$ sends~$\omega$ to~$-\omega$, it is also invariant under the antipodal map.
\end{definition}

Let us mention the following topological result. See~\cite[Proposition~4.1]{Bon86}.

\begin{proposition} \label{prop:metrizable}
The space~$\Curr(M)$ of geodesic currents is a Polish space (\ie, a separable, completely metrizable topological space).
\end{proposition}

\begin{proof}
Since $PT^*M$ is compact, there exist finitely many flow boxes $B_i \simeq \Sigma_i \times [-1,1]$ whose interiors cover~$PT^*M$, where each~$\Sigma_i$ is a compact surface transverse to~$\mathcal{F}$.
A geodesic current~$\mu$ is determined by the Radon measures ~$\mu_i = \mu_{\Sigma_i}$ it induces on the transverse surfaces~$\Sigma_i$, which satisfy the holonomy compatibility conditions.
These compatibility conditions are closed.
Thus, $\Curr(M)$ can be identified homeomorphically with a closed subset of the product~$\prod_i \mathcal{M}^+(\Sigma_i)$, where $\mathcal{M}^+(\Sigma_i)$ denotes the space of Radon measures on~$\Sigma_i$ (which is a closed cone in the vector space $\mathcal{M}(\Sigma_i)$ of signed Radon measures endowed with the vague topology).
By~\cite[15.7.7]{kallenberg}, the space~$\mathcal{M}^+(\Sigma_i)$ of Radon measures on~$\Sigma_i$ is Polish for the vague topology.
The same holds for every closed subset of the product space~$\prod_i \mathcal{M}^+(\Sigma_i)$, and in particular for the space of geodesic currents~$\Curr(M)$.
\end{proof}

Let us now define Bonahon's intersection form between geodesic currents; see~\cite{Bon86}.

\begin{definition}
Let $M$ be a closed surface endowed with a Finsler metric.
Let $PT^*M \times_M PT^*M$ be the Whitney sum of the projective tangent bundle $PT^*M \to M$, which is the closed $4$-manifold consisting of triples~$(x,\xi_1,\xi_2)$ where $x \in M$ and $\xi_1, \xi_2$ are cotangent lines to~$M$ at~$x$.
Denote by $p_i:PT^*M \times_M PT^*M \to PT^*M$ the projection to the $i$-th factor defined by
\[
p_i(x,\xi_1,\xi_2) = (x,\xi_i).
\]
Let $\mathcal{F}_i=p_i^*(\mathcal{F})$ be the codimension-two foliation of~$PT^*M \times_M PT^*M$ whose leaves are the preimages under~$p_i$ of the leaves of~$\mathcal{F}$.
The foliations~$\mathcal{F}_1$ and~$\mathcal{F}_2$ are transverse outside the diagonal $\Delta = \{ (x,\xi,\xi) \mid (x,\xi) \in PT^*M \}$ of $PT^*M \times_M PT^*M$.

For every $\mu \in \Curr(M)$, define the measure~$p_i^*(\mu)$ transverse to~$\mathcal{F}_i$ by the formula
\begin{equation} \label{eq:qhat}
\check{q}^* \left( p_i^*(\mu) \right) = (p_i \circ \check{q})^* \mu
\end{equation}
for every immersed surface $\check{q}:\Sigma \to PT^*M \times_M PT^*M$ transverse to~$\mathcal{F}_i$.

Let $\mu_1, \mu_2 \in \Curr(M)$ be two geodesic currents.
Since the foliations~$\mathcal{F}_1$ and~$\mathcal{F}_2$ are transverse away from~$\Delta$, the transverse measures~$p_1^*(\mu_1)$ and~$p_2^*(\mu_2)$ combine to define a product measure~$p_1^*(\mu_1) \times p_2^*(\mu_2)$ on $PT^*M \times_M PT^*M \setminus \Delta$.
More precisely, by transversality, every sufficiently small neighborhood~$U$ in~$PT^*M \times_M PT^*M \setminus \Delta$ admits a bifoliated chart $\Sigma_1 \times \Sigma_2 \to U$, where $\Sigma_1$ lies in a leaf of~$\mathcal{F}_2$ and $\Sigma_2$ lies in a leaf of~$\mathcal{F}_1$, so that the horizontal surfaces~$\Sigma_1 \times \{ \sigma_2 \}$ are sent into leaves of~$\mathcal{F}_2$ and the vertical surfaces~$\{ \sigma_1 \} \times \Sigma_2$ are sent into leaves of~$\mathcal{F}_1$.
In such a chart, the measure~$p_1^*(\mu_1) \times p_2^*(\mu_2)$ is defined as the pushforward of the ordinary product measure
\begin{equation} \label{eq:product}
p_1^*(\mu_1)_{\Sigma_1} \times p_2^*(\mu_2)_{\Sigma_2}.
\end{equation}
This definition does not depend on the choice of the bifoliated chart.
Indeed, on the overlap of two such charts, the corresponding horizontal surfaces are related by holonomy along the leaves of~$\mathcal{F}_1$, while the corresponding vertical surfaces are related by holonomy along the leaves of~$\mathcal{F}_2$.
Since $p_1^*(\mu_1)$ and~$p_2^*(\mu_2)$ are invariant under the corresponding holonomy maps, the product measures agree on the overlap.
The measure~$p_1^*(\mu_1) \times p_2^*(\mu_2)$ thus-defined on $PT^*M \times_M PT^*M \setminus \Delta$ has finite mass; see Proposition~\ref{prop:continuity}.

The intersection form
\[
i:\Curr(M) \times \Curr(M) \to \R_+
\]
is defined as
\[
i(\mu_1,\mu_2) = \mass(p_1^*(\mu_1) \times p_2^*(\mu_2)).
\]
It is a symmetric bilinear form.
(Technically, it is bilinear with respect to nonnegative linear combinations since the space of geodesic currents is only a cone, not a vector space.)
\end{definition}

\begin{example} \label{ex:intersection}
For every pair of closed geodesics~$\gamma_1$ and~$\gamma_2$, the intersection~$i(\gamma_1,\gamma_2)$ between the geodesic currents they induce is equal to the number of transverse intersections between~$\gamma_1$ and~$\gamma_2$, counted with multiplicity.
Indeed, by the local definition of the product measure, see~\eqref{eq:product}, the total mass of~$p_1^*(\gamma_1) \times p_2^*(\gamma_2)$ is equal to the number of triples~$(x,\xi_1,\xi_2) \in PT^*M \times_M PT^*M$, where $x$ is an intersection point of~$\gamma_1$ and~$\gamma_2$, and $\xi_1$ and~$ \xi_2$ are distinct lines tangent to~$\gamma_1$ and~$\gamma_2$ at~$x$.
This is exactly the number of intersections between~$\gamma_1$ and~$\gamma_2$ counted with multiplicity.
See~\cite[Proposition~4.5]{Bon86}.
\end{example}

The following fundamental properties of the intersection form were first established in~\cite{Bon86} and~\cite{Bon88} for hyperbolic metrics; see also~\cite{Otal} for negatively curved Riemannian metrics and~\cite{Fathi} for Riemannian metrics without conjugate points.
We extend these results to arbitrarily Finsler metrics, without any assumption on conjugate points.
To this end, we work directly with the original definition of geodesic currents as invariant transverse measures, rather than with their alternative description as measures on the space of geodesic lines in the universal cover.

\medskip

The first property concerns the continuity of the intersection form, which was first proved in~\cite[Proposition~4.5]{Bon86} for hyperbolic metrics.
For arbitrary Finsler or Riemannian metrics, however, the intersection form need not be continuous everywhere.
For example, on the round sphere, if $\gamma_n$ is a sequence of great circles converging to a great circle~$\gamma$ with $\gamma_n \neq \gamma$, then $i(\gamma_n,\gamma)=2$ while $i(\gamma,\gamma)=0$.
We show, however, that the intersection form is continuous at every pair of geodesic currents with no atoms in common, which is enough for our purposes.

\begin{proposition} \label{prop:continuity}
Let $M$ be a closed surface endowed with a Finsler metric~$F$.
Then the intersection form 
$
i:\Curr(M) \times \Curr(M) \to \R_+
$
is finite-valued, and continuous at each pair of geodesic currents~$(\mu_1,\mu_2)$ that have no atom in common.
\end{proposition}

\begin{proof}
We follow the argument of~\cite[Proposition~4.5]{Bon86}.
Let $\nu_1, \nu_2 \in \Curr(M)$.
By definition,
\[
i(\nu_1,\nu_2) = p_1^*(\nu_1) \times p_2^*(\nu_2) (PT^*M \times_M PT^*M \setminus \Delta)
\]
where $\Delta$ is the diagonal of~$PT^*M \times_M PT^*M$.

We first show that the intersection form is finite.
Choose finitely many flow boxes $B_j \simeq \Sigma_j \times [-1,1]$ whose interiors cover~$PT^*M$, where $\Sigma_j$ is a compact surface transverse to~$\mathcal{F}$.
Taking the boxes sufficiently small, we may assume that any pair of distinct portions of leaves lying in flow boxes project to simple geodesic arcs in~$M$ intersecting at most once.
For $B_j \times_M B_k := p_1^{-1}(B_j) \cap p_2^{-1}(B_k)$, it follows from the definition of the product measure that
\begin{equation} \label{eq:mass}
p_1^*(\nu_1) \times p_2^*(\nu_2)(B_j \times_M B_k \setminus \Delta) \leq \mass({\nu_1}|_{\Sigma_j}) \, \mass({\nu_2}|_{\Sigma_k}).
\end{equation}
Since finitely many~$B_j \times_M B_k$ cover~$PT^*M \times_M PT^*M$, this proves that $i(\nu_1,\nu_2)$ is finite.

We now prove continuity at $(\mu_1,\mu_2)$ with no atom in common.
Fix~$\varepsilon >0$.
We first claim that there exist an open neighborhood~$U$ of~$\Delta$ in~$PT^*M \times_M PT^*M$ 
and a neighborhood~$\mathcal{U}$ of~$(\mu_1,\mu_2)$ in~$\Curr(M) \times \Curr(M)$ such that
\begin{equation} \label{eq:U-D}
p_1^*(\nu_1) \times p_2^*(\nu_2)(U \setminus \Delta) < \varepsilon
\end{equation}
for every $(\nu_1,\nu_2) \in \mathcal{U}$.

To prove the claim, fix~$\delta>0$.
We may choose the above flow boxes~$B_j$ such that 
\[
\mass({\mu_1}|_{\Sigma_j}) < \delta \qquad \mbox{ or } \qquad \mass({\mu_2}|_{\Sigma_j}) < \delta.
\]
This is possible because $\mu_1$ and~$\mu_2$ have no atom in common, and every massless point has neighborhoods of arbitrarily small mass.

Recall that the space of geodesic currents in metrizable; see Proposition~\ref{prop:metrizable}.
By the portmanteau theorem~\cite[Theorem~2.1.(iii)]{billingsley}, there exists a neighborhood~$\mathcal{U}$ of~$(\mu_1,\mu_2)$ such that for every $(\nu_1,\nu_2) \in \mathcal{U}$, 
\begin{equation} \label{eq:or}
\mass({\nu_1}|_{\Sigma_j}) < \delta \qquad \mbox{ or } \qquad \mass({\nu_2}|_{\Sigma_j}) < \delta.
\end{equation}
Shrinking $\mathcal{U}$ if necessary, we may also assume that
\[
\mass({\nu_i}|_{\Sigma_j}) \leq C
\]
for some~$C>0$ depending on~$(\mu_1,\mu_2)$, but not on~$(\nu_1,\nu_2)$.

The set $U=\Int(\cup_j B_j \times_M B_j)$ is an open neighborhood of~$\Delta$.
By~\eqref{eq:mass} and~\eqref{eq:or}, we have
\[
p_1^*(\nu_1) \times p_2^*(\nu_2)(B_j \times_M B_j \setminus \Delta) < C \delta.
\]
Since there are finitely many flow boxes~$B_j$, it follows that
\[
p_1^*(\nu_1) \times p_2^*(\nu_2)(U \setminus \Delta) < C'  \delta
\]
for some~$C'>0$ depending on~$(\mu_1,\mu_2)$, but not on~$(\nu_1,\nu_2)$.
Choosing $\delta$ sufficiently small gives~\eqref{eq:U-D}.

Define the compact set $K=PT^*M \times_M PT^*M \setminus U$.
Slightly deforming~$K$ (by enlarging it away from~$\Delta$), we may assume
\[
p_1^*(\mu_1) \times p_2^*(\mu_2)(\partial K) = 0.
\]
Now, the map $(\nu_1,\nu_2) \mapsto p_1^*(\nu_1) \times p_2^*(\nu_2)$ to the space of Radon measures on~$PT^*M \times_M PT^*M \setminus \Delta$  is continuous for the vague topology.
Hence, by the portmanteau theorem, there exists a neighborhood~$\mathcal{U}'$ of~$(\mu_1,\mu_2)$ such that
\begin{equation} \label{eq:K}
\left| p_1^*(\nu_1) \times p_2^*(\nu_2)(K) - p_1^*(\mu_1) \times p_2^*(\mu_2)(K) \right| < \varepsilon.
\end{equation}
for every $(\nu_1,\nu_2) \in \mathcal{U}'$.

Since $U$ and~$K$ cover~$PT^*M \times_M PT^*M$, the continuity of the intersection form at~$(\mu_1,\mu_2)$ follows from~\eqref{eq:U-D} and~\eqref{eq:K}.
\end{proof}

The second property is a geometric interpretation of the intersection form between a geodesic current induced by a closed geodesic and another geodesic current.
See~\cite[Proposition~14]{Bon88}.

\begin{proposition} \label{prop:transverse}
Let $M$ be a closed surface endowed with a Finsler metric~$F$.
Let $\gamma$ be a closed geodesic whose lift to~$U^*M$ is denoted by~$\hat{\gamma}$.
Let~$\xi_\gamma(t)$ be the cotangent line generated by~$\hat{\gamma}(t)$.
Define
\[
\Sigma_\gamma = \{ (t,\xi) \mid t \in S^1, \xi \in PT^*_{\gamma(t)}M, \xi \neq \xi_\gamma(t) \}.
\]
The map 
$
q_\gamma: \Sigma_\gamma \to PT^*M
$
defined by $q_\gamma(t,\xi)=(\gamma(t),\xi)$ is an immersion transverse to the geodesic foliation~$\mathcal{F}$.
Thus, every geodesic current~$\mu$ induces a measure~$q_\gamma^* \mu$ on~$\Sigma_\gamma$.
Then,
\[
i(\gamma,\mu) = \mass(q_\gamma^* \mu).
\]
\end{proposition}

\begin{proof}
Let us compute the product measure~$p_1^*(\gamma) \times p_2^*(\mu)$.
The geodesic~$\gamma$ determines a $k$-fold cover of a closed leaf of~$\mathcal{F}$, where $k$ is the multiplicity of~$\gamma$.
The preimage of this leaf under~$p_1$ is a leaf of~$\mathcal{F}_1$ whose intersection with the complement of~$\Delta$ is covered $k$ times by
\[
\check{q}_\gamma: \Sigma_\gamma \to PT^*M \times_M PT^*M \setminus \Delta
\]
where $\check{q}_\gamma(t,\xi) = (t,\xi_\gamma(t),\xi)$.
Note that $p_2 \circ \check{q}_\gamma = q_\gamma$.

Since $\mathcal{F}_1$ and~$\mathcal{F}_2$ are transverse away from~$\Delta$, the measure~$p_2^*(\mu)$ transverse to~$\mathcal{F}_2$ induces a measure~$\check{q}_\gamma^* \left( p_2^*(\mu) \right) = (p_2 \circ \check{q}_\gamma)^*(\mu) = q_\gamma^* \mu$ on this leaf, more precisely on~$\Sigma_\gamma$; see~\eqref{eq:qhat}.

On the other hand, 
the measure~$p_1^*(\gamma)$ transverse to~$\mathcal{F}_1$ is supported on the leaf parametrized by~$\check{q}_\gamma$.
More precisely, on every surface~$\Sigma$ transverse to~$\mathcal{F}_1$, the induced measure~$p_1^*(\gamma)_\Sigma$ is $k$ times the sum of the Dirac measures supported at the intersection points of~$\Sigma$ with this leaf.

By the local definition of the product measure, see~\eqref{eq:product}, we have
\[
p_1^*(\gamma) \times p_2^*(\mu) = (\check{q}_\gamma)_*(q_\gamma^* \mu).
\]
Hence,
\begin{equation*} \label{eq:gamma-mu}
i(\gamma,\mu) = \mass((\check{q}_\gamma)_*(q_\gamma^* \mu)) = \mass(q_\gamma^* \mu). \qedhere
\end{equation*}
%
%
%
\end{proof}

\forget
The second property is a geometric interpretation of the intersection form between geodesic currents induced by closed geodesics.
See~\cite[Proposition~4.5]{Bon86}.

\begin{proposition} \label{prop:transverse}
Let $M$ be a closed surface endowed with a Finsler metric~$F$.
Then, for every pair of closed geodesics~$\gamma_1$ and~$\gamma_2$, the intersection~$i(\gamma_1,\gamma_2)$ between the geodesic currents they induce is equal to the number of transverse intersections between~$\gamma_1$ and~$\gamma_2$, counted with multiplicity.
\end{proposition}

\begin{proof}
By bilinearity of the intersection form, we may assume that $\gamma_1$ and~$\gamma_2$ are primitive.

Let $\gamma$ be a primitive geodesic and $\mu$ be a geodesic current.
In order to determine~$i(\gamma,\mu)$, let us compute $p_1^*(\gamma) \times p_2^*(\mu)$.
The geodesic~$\gamma$ determines a leaf of~$\mathcal{F}$.
The preimage of this leaf under~$p_1$ is a leaf of~$\mathcal{F}_1$ whose restriction outside~$\Delta$ is given by
\[
\Sigma_\gamma = \{ (\gamma(t),\xi_\gamma(t),\xi) \in PT^*M \times_M PT^*M \mid t \in S^1, \xi \neq \xi_\gamma(t)  \}
\]
where $\xi_\gamma(t)$ is the tangent line generated by~$\gamma'(t)$.
Denote by
\[
\iota_\gamma : \Sigma_\gamma \hookrightarrow PT^*M \times_M PT^*M \setminus \Delta
\]
the inclusion map.
The geodesic current~$\mu$ induces a transverse measure~$p_2^*(\mu)$ to the foliation~$\mathcal{F}_2$.
Since $\mathcal{F}_1$ is transverse to~$\mathcal{F}_2$ away from~$\Delta$ and $\Sigma_\gamma$ lies in a leaf of~$\mathcal{F}_1$, the transverse measure~$p_2^*(\mu)$ determines a measure
\[
\mu_\gamma = p_2^*(\mu)_{\Sigma_\gamma}
\]
on~$\Sigma_\gamma$.

Now, since $\gamma$ is primitive, the transverse measure~$p_1^*(\gamma)$ induces  a measure~$p_1^*(\gamma)_\Sigma$ on every surface~$\Sigma$ transverse to~$\mathcal{F}_1$ defined as the sum of the Dirac measures supported at the intersection points of the transverse surfaces~$\Sigma_\gamma$ and~$\Sigma$.
The definition~\eqref{eq:product} of~$p_1^*(\gamma) \times p_2^*(\mu)$ in a bifoliated chart~$\Sigma \times \Sigma_\gamma$ yields the following formula
\[
p_1^*(\gamma) \times p_2^*(\mu) = (\iota_\gamma)_*(\mu_\gamma).
\]
Hence,
\begin{equation} \label{eq:gamma-mu}
i(\gamma,\mu) = \mass(\mu_\gamma) = \mass(\mu_{p_2(\Sigma_\gamma)})
\end{equation}
where by construction,
\[
p_2(\Sigma_\gamma) = \{ (\gamma(t),\xi) \mid t \in S^1, \xi \neq \xi_\gamma(t) \}.
\]

By the formula~\eqref{eq:gamma-mu} applied to~$\gamma=\gamma_1$ and~$\mu=\gamma_2$, the intersection~$i(\gamma_1,\gamma_2)$ is equal to the number of triples~$(x,\xi_1,\xi_2) \in PT^*M \times_M PT^*M$, where $x$ is an intersection point of~$\gamma_1$ and~$\gamma_2$, and $\xi_1$ and~$ \xi_2$ are distinct lines tangent to~$\gamma_1$ and~$\gamma_2$ at~$x$.
This is exactly the number of transverse intersections between~$\gamma_1$ and~$\gamma_2$ counted with multiplicity. 
\end{proof}
\forgotten

The third property expresses the length of a closed geodesic in terms of the intersection between the geodesic current it induces and the Liouville geodesic current.
See~\cite[Proposition~14]{Bon88}.

\begin{proposition} \label{prop:bonahon_length}
Let $M$ be a closed surface endowed with a Finsler metric~$F$.
Let $\lambda$ be the Liouville geodesic current of~$M$.
Then, for every closed geodesic~$\gamma$,
\[
i(\gamma,\lambda) = \length(\gamma).
\]
\end{proposition}

\begin{proof}
By bilinearity of the intersection form, we may assume that $\gamma$ is primitive.
By Proposition~\ref{prop:transverse}, 
\[
i(\gamma,\lambda) = \mass(\lambda_{\Sigma_\gamma})
\]
where 
\[
\Sigma_\gamma = \{ (t,\xi) \mid t \in S^1, \xi \in PT^*_{\gamma(t)}M, \xi \neq \xi_\gamma(t) \}.
\]
Fix a co-orientation of~$\gamma$.
Define
\[
C^*\gamma = \{ (t,\eta) \mid t \in S^1, \eta \in U^*_{\gamma(t)} M \textrm{ crosses } \gamma \textrm{ positively at } t \}
\]
where a covector~$\eta$ crosses~$\gamma$ positively if the unit vector corresponding to~$\eta$ under the Legendre transform is transverse to~$\gamma$ and points in the positive direction with respect to the chosen co-orientation; see~\cite[Definition~3.8]{CS}.
The natural projection $U^*M \to PT^*M$ induces a diffeomorphism $C^*\gamma \to \Sigma_\gamma$.
By definition of the Liouville geodesic current, under this identification, we have
\[
\lambda_{\Sigma_\gamma} = \tfrac{1}{2} \, | \omega|_{C^*\gamma} |.
\]
Therefore, 
\[
i(\gamma,\lambda) = \frac{1}{2} \int_{C^*\gamma} |\omega|.
\]

Consider the restriction map $\rho:C^*\gamma \to \textrm{Int}(B^* \gamma)$ defined by $\rho(t,\eta) = (t,\eta|_{T_{\gamma(t)} \gamma})$, where $B^*\gamma \subseteq T^*\gamma$ is the unit cotangent ball bundle of the Finsler metric induced on~$\gamma$.
By~\cite[Lemma~5.4]{APB} (see also~\cite[Lemma~3.9]{CS} for details), the map~$\rho$ is a diffeomorphism and satisfies
\[
\rho^* \omega_\gamma = \omega|_{C^*\gamma}
\]
where $\omega_\gamma$ is the canonical symplectic form on~$T^*\gamma$.
Hence
\[
i(\gamma,\lambda) = \frac{1}{2} \int_{B^*\gamma} \omega_\gamma = \length(\gamma). \qedhere
\]
\end{proof}

Although less conceptual, the following alternative proof, using computations in local coordinates, is more direct.
It was suggested by an AI tool.

\begin{proof}[Alternative proof]
As above, we may assume that $\gamma$ is primitive.
By Proposition~\ref{prop:transverse}, 
\[
i(\gamma,\lambda) = \mass(\lambda_{\Sigma_\gamma}).
\]

Parametrize~$\gamma$ by arclength with $s \in [0,\length(\gamma)]$.
Consider the double cover of the surface~$\Sigma_\gamma$ obtained by lifting the projective cotangent directions to~$U^*M$.
Up to the two sections corresponding to the cotangent directions, which do not affect the measure, this double cover can be parametrized by $(s,\xi(s,\theta))$, with $\xi(s,\theta) \in U^*_{\gamma(s)} M$, where $(s,\theta) \in [0,\length(\gamma)] \times S^1$.
In these coordinates, the tautological one-form may be written
\[
\alpha|_{(s,\theta)} = \xi(s,\theta)(\gamma'(s)) \, ds.
\]
Set $f(s,\theta) = \xi(s,\theta)(\gamma'(s))$.
Then
\[
\omega = \frac{\partial f}{\partial \theta} \, d\theta \wedge ds.
\]

For every fixed~$s$, the function~$f_s=f(s,\cdot)$ has a unique maximum and a unique minimum on~$S^1$.
Moreover, $f_s$ is monotone on each of the two arcs joining its minimum and maximum.
By duality between the Finsler norm and its dual norm,
\[
\max_{\theta \in S^1} f(s,\theta) = \max_{\xi \in U^*_{\gamma(s)}M} \xi(\gamma'(s)) = F(\gamma'(s)) = 1.
\]
Since the Finsler metric is reversible, $\min f_s = - \max f_s = -1$.
It follows that
\[
\int_{S^1} \left| \frac{\partial f}{\partial \theta}  \right| \, d\theta = 2 \, (\max f_s - \min f_s) = 4 \max f_s = 4.
\]
Hence, the total $|\omega|$-mass of the double cover is
\[
\int_0^{\length(\gamma)} \int_{S^1}\left| \frac{\partial f}{\partial \theta}  \right| \, d\theta \, ds = 4 \length(\gamma).
\]

The two lifts of~$\Sigma_\gamma$ to~$U^*M$ are exchanged by the antipodal map, which preserves~$|\omega|$.
Thus, the $|\omega|$-mass of each lift is equal to $2 \length(\gamma)$.

By definition of the Liouville geodesic current, 
\[
i(\gamma,\lambda) = \mass(\lambda_{\Sigma_\gamma}) = \length(\gamma). \qedhere
\]
\end{proof}

The fourth property gives a formula for the area of a Finsler surface in terms of the intersection between the Liouville geodesic current and itself.
See~\cite[Proposition~15]{Bon88}.

\begin{proposition} \label{prop:bonahon_area}
Let $M$ be a closed surface endowed with a Finsler metric~$F$.
Let $\lambda$ be the Liouville geodesic current of~$M$.
Then,
\[
i(\lambda,\lambda) = \frac{\pi}{2}  \area(M).
\]
\end{proposition}

\begin{proof}
Recall that
\[
i(\lambda,\lambda) = \mass(p_1^*(\lambda) \times p_2^*(\lambda)).
\]
By definition, the transverse measure~$p_i^*(\lambda)$ induces a measure on every surface~$\Sigma$ transverse to~$\mathcal{F}_i$ given by
\[
p_i^*(\lambda) = \tfrac{1}{2} \left| p_i^* \omega|_{\bar{\Sigma}} \right|
\]
where $\bar{\Sigma}$ is a lift of~$\Sigma$ to~$U^*M$.
Thus,
\[
p_1^*(\lambda) \times p_2^*(\lambda) = \tfrac{1}{4} \, |p_1^* \omega \wedge p_2^* \omega|.
\]

Decompose $M$ into finitely many sufficiently small geodesically convex polygonal disks $D_1,\dots,D_n$ with minimizing interior geodesics and pairwise disjoint interiors.
We have 
\[
U^*M \times_M U^*M = \cup_{j=1}^m U^*D_j \times_{D_j} U^*D_j.
\]

Fix one of these disks~$D$.
Let~$\Gamma_D$ be the space of oriented geodesics of~$D$, constructed as the quotient of~$U^*D$ by the geodesic flow.
Let \mbox{$\pi:U^*D \to \Gamma_D$} be the corresponding projection map, which sends each unit covector $(x,\xi) \in U^*D$ to the geodesic arc~$\gamma_\xi$ is generates.
Note that~$\Gamma_D$ is a cylinder.
Since the $2$-form~$\omega$ on~$U^*D$ is invariant by the geodesic flow and vanishes in the direction of the flow, see Definition~\ref{def:cogeodesic}, it passes to the quotient.
That is, there is a unique symplectic form~$\Omega$ on~$\Gamma_D$ such that
\[
\pi^* \Omega = \omega.
\]
See~\cite[\S4.3]{AM78} for details.

We have the commutative diagram
\[
\begin{tikzcd}
U^*D \ar[d, "\pi"'] & \ar[l, "p_1"'] U^*D \times_D U^*D \ar[r, "p_2"] \ar[d, "\Pi"'] & U^*D \ar[d, "\pi"] \\
\Gamma_D & \Gamma_D \times \Gamma_D \ar[l, "q_1"] \ar[r, "q_2"'] & \Gamma_D 
\end{tikzcd}
\]
where $p_i(x,\xi_1,\xi_2)=(x,\xi_i)$, $q_i(\gamma_1,\gamma_2)=\gamma_i$, $\pi(x,\xi)=\gamma_\xi$ and $\Pi(x,\xi_1,\xi_2) = (\gamma_{\xi_1},\gamma_{\xi_2})$.
Thus, $p_i^* \omega = p_i^* \pi^* \Omega = \Pi^* q_i^* \Omega$ and
\[
p_1^* \omega \wedge p_2^* \omega = \Pi^*(q_1^* \Omega \wedge q_2^* \Omega).
\]

Let $\mu=|\Omega|$ be the measure induced by~$\Omega$ on~$\Gamma_D$.
By the coarea formula~\cite[Theorem~3.2.3]{Fed} and the Santal\'o+Blaschke formula~\cite[Corollary~3.5]{CS},
\begin{align*}
\int_{U^*D \times_D U^*D} \tfrac{1}{4} \, |p_1^* \omega \wedge p_2^* \omega | & = \frac{1}{4} \, \int_{U^*D \times_D U^*D} |\Pi^*(q_1^* \Omega \wedge q_2^* \Omega)| \\
 & = \frac{1}{4} \, \int_{(\gamma_1,\gamma_2) \in \Gamma_D \times \Gamma_D} \# \Pi^{-1}(\gamma_1,\gamma_2) \, | q_1^* \Omega \wedge q_2^* \Omega | \\
 & = \frac{1}{4} \, \int_{(\gamma_1,\gamma_2) \in \Gamma_D \times \Gamma_D} \#(\gamma_1 \cap \gamma_2 \cap D) \, d\mu(\gamma_1) d\mu(\gamma_2) \\
 & = 2\pi \, \area(D).
\end{align*}

Taking the sum over all disks~$D_j$ and dividing the result by~$4$ because the computations were performed in the four-fold cover~$U^*M \times_M U^*M$, we obtain
\[
i(\lambda,\lambda) = \mass \left( \tfrac{1}{4} \,  | p_1^* \omega \wedge p_2^* \omega | \right) = \frac{\pi}{2} \area(M). \qedhere
\]
\end{proof}

Although less conceptual, the following alternative proof, using computations in local coordinates, is more direct.
It was suggested by an AI tool.

\begin{proof}[Alternative proof]
As in the previous proof, we perform computation on the four-fold cover~$U^*M \times_M U^*M$ of~$PT^*M \times_M PT^*M$.
Take local coordinates $x=(x_1,x_2)$ on~$M$.
Choose an orientation-preserving parametrization~$\zeta(x,\theta)$ of~$U_x^*M$ with~$\theta \in S^1$, equivariant under the antipodal map.
This gives a parametrization $(\xi,\eta)= (\zeta(x,\theta),\zeta(x,\varphi))$ of the Whitney sum $U^*M \times_M U^*M$ with $x \in M$ and $\theta,\varphi \in S^1$.
Note that $\xi(x,\theta)=\eta(x,\varphi)$ if $\theta=\varphi$.
In this coordinate system, the pullbacks of~$\omega$ under~$p_1$ and~$p_2$ are given by
\begin{align*}
p_1^* \omega & = \left( \frac{\partial \xi_2}{\partial x_1} - \frac{\partial \xi_1}{\partial x_2} \right) \, dx_1 \wedge dx_2 + \frac{\partial \xi_1}{\partial \theta} \, d\theta \wedge dx_1 + \frac{\partial \xi_2}{\partial \theta} \, d\theta \wedge dx_2 \\
p_2^* \omega & = \left( \frac{\partial \eta_2}{\partial x_1} - \frac{\partial \eta_1}{\partial x_2} \right) \, dx_1 \wedge dx_2 + \frac{\partial \eta_1}{\partial \varphi} \, d\varphi \wedge dx_1 + \frac{\partial \eta_2}{\partial \varphi} \, d\varphi \wedge dx_2.
\end{align*}
Thus,
\[
p_1^* \omega \wedge p_2^* \omega = - \left( \frac{\partial \xi_1}{\partial \theta} \frac{\partial \eta_2}{\partial \varphi} - \frac{\partial \xi_2}{\partial \theta} \frac{\partial \eta_1}{\partial \varphi} \right) \, dx_1 \wedge dx_2 \wedge d\theta \wedge d\varphi.
\]
That is,
\[
| p_1^* \omega \wedge p_2^* \omega | = | \det(\partial_\theta \xi, \partial_\varphi \eta) | \,  dx_1 dx_2 d\theta d\varphi.
\]

Let us compute 
\[
J(x) = \int_{S^1} \int_{S^1} | \det(\partial_\theta \xi, \partial_\varphi \eta) | \, d\theta d\varphi.
\]
Fix $x \in M$ and $\varphi \in S^1$.
Define the odd function $f(\theta) = \det(\xi(x,\theta),\partial_\varphi \eta(x,\varphi))$ on~$S^1$.
Its derivative $f' = \det(\partial_\theta \xi,\partial_\varphi \eta)$ is positive on one of the two arcs of~$S^1$ with endpoints~$\pm \varphi$ where $\partial_\theta \xi$ is parallel to~$\partial_\varphi \eta$, and negative on the other.
Thus,
\[
\int_{S^1} | f' | \, d\theta = 2 (\max f - \min f) = 4 \max f.
\]
Since $f'$ changes sign at~$\pm \varphi$, we have $\max f = \max \{ f(\varphi),f(-\varphi) \}$.
Actually, since $\eta$ is orientation preserving, $\max f = f(\varphi) = \det(\eta,\partial_\varphi \eta)$.
Therefore,
\[
\int_{S^1} | f' | \, d\theta = 4 \det(\eta,\partial_\varphi \eta).
\]
Consequently,
\[
J(x) = 4\int_{S^1} \det(\eta,\partial_\varphi \eta) \, d\varphi = 8 \area(B_x^*M)
\]
where the last equality follows from Green's formula; see~\cite[pp.150-151]{APT}.
Hence, in $U^*M \times_M U^*M$,
\begin{align*}
\int_{(\theta,\varphi) \in S^1 \times S^1} \tfrac{1}{4} \, | p_1^* \omega \wedge p_2^* \omega | & = \tfrac{1}{4} J(x) \, dx_1 dx_2 \\
 & = 2 \area(B_x^*M) \, dx_1 dx_2 \\
 & = 2 \pi \, dA.
\end{align*}
Since the computations were performed in the four-fold cover~$U^*M \times_M U^*M$, we need to divide this result by~$4$ to have its expression in~$PT^*M \times_M PT^*M$.

Finally, integrating over~$M$, we obtain
\[
i(\lambda,\lambda) = \mass \left( \tfrac{1}{4} \,  | p_1^* \omega \wedge p_2^* \omega | \right) = \frac{\pi}{2} \area(M). \qedhere
\]
\end{proof}

We finish this section with some results regarding filling currents and compactness.

\begin{definition}
A geodesic current~$\mu_0$ is \emph{filling} if $i(\mu_0,\mu) >0$ for every nonzero geodesic current~$\mu$.
\end{definition}

The Liouville geodesic current provides a fundamental example.

\begin{lemma}
Let $M$ be a closed surface endowed with a Finsler metric.
The Liouville geodesic current~$\lambda$ on~$M$ is filling.
\end{lemma}

\begin{proof}
Let $\mu$ be a nonzero geodesic current.
Then there exist a transverse surface~$\Sigma \subseteq PT^*M$ and an arbitrarily small open set~$U \subseteq \Sigma$ such that $\mu_\Sigma(U) >0$.
 Choose another open set $V \subseteq \Sigma$ disjoint from~$U$.
 Since the Liouville transverse measure is given locally by a smooth positive density, it satisfies $\lambda_\Sigma(V) >0$.
Therefore, the mass of the product measure $p_1^*(\lambda) \times p_2^*(\mu)$ is positive on the open set of $PT^*M \times_M PT^*M \setminus \Delta$ corresponding to the product~$U \times V$.
Hence,
 \[
 i(\lambda,\mu) = \mass \left( p_1^*(\lambda) \times p_2^*(\mu) \right) >0.
 \]
 Thus, $\lambda$ is filling.
\end{proof}


Our next proposition is a compactness result for normalized geodesic currents.
See~\cite[Proposition~4]{Bon88}.

\begin{proposition} \label{prop:compact}
Let $M$ be a closed surface endowed with a Finsler metric.
Let $\mu_0$ be a filling geodesic current without atoms (\eg, the Liouville geodesic current~$\lambda$).
Then the space
\[
\mathcal{K} = \{ \mu \in \Curr(M) \mid i(\mu_0,\mu) = 1\}
\]
is compact and metrizable.
\end{proposition}

\begin{proof}
%
We begin as in the proof of Proposition~\ref{prop:metrizable}.
Since $PT^*M$ is compact, there exist finitely many flow boxes $B_i \simeq \Sigma_i \times [-1,1]$ whose interiors cover~$PT^*M$, where each~$\Sigma_i$ is a compact surface transverse to~$\mathcal{F}$.
Every geodesic current~$\mu$ is determined by the transverse measures~$\mu_{\Sigma_i}$, which satisfy closed holonomy compatibility conditions.
Thus, $\Curr(M)$ can be identified homeomorphically with a closed subset of the product~$\prod_i \mathcal{M}^+(\Sigma_i)$.

For every $C\geq 0$, the space~$\mathcal{M}_{\leq C}^+(\Sigma_i)$ of Radon measures on~$\Sigma_i$ with mass bounded by~$C$ is compact and metrizable for the vague topology.
Indeed, by the Riesz representation theorem~\cite[Corollary~7.18]{folland}, the space $\mathcal{M}(\Sigma_i)$ of signed Radon measures on~$\Sigma_i$, endowed with the total variation norm, is isometric to the dual of the Banach space~$C(\Sigma_i)$ of continuous functions on~$\Sigma_i$, where $C(\Sigma_i)$ is endowed with the uniform norm.
Since $\Sigma_i$ is compact and metrizable, the Banach space~$C(\Sigma_i)$ is separable.
Hence, by the Banach--Alaoglu theorem and~\cite[Theorem~3.28]{brezis}, the dual ball~$\mathcal{M}_{\leq C}(\Sigma_i)$ is compact and metrizable for the weak$^*$ topology, which coincides with the vague topology since $\Sigma_i$ is compact.
Since the space~$\mathcal{M}_{\leq C}^+(\Sigma_i)$ of nonnegative Radon measures is a closed subset of this ball, it is therefore compact and metrizable.

Define an auxiliary continuous functional $h:\Curr(M) \to \R_+$ by
\[
h(\mu) = \sum_i \int_{\Sigma_i^+} \psi_i \, d\mu_{\Sigma_i^+},
\]
where~$\Sigma_i^+$ is an open transverse surface extending the compact surface~$\Sigma_i$, and $\psi_i \in C_c(\Sigma_i^+)$ is a nonnegative function such that $\psi_i \equiv 1$ on~$\Sigma_i$.

We claim that $h(\mu)$ is positive for every nonzero geodesic current~$\mu$.
Indeed, for every $\mu \neq 0$, there exist a transverse surface~$\Sigma \subseteq PT^*M$ and a point~$\xi$ in the support of~$\mu_\Sigma$.
Thus, $\mu_\Sigma(U) >0$ for every neighborhood~$U \subseteq \Sigma$ of~$\xi$.
Since the interiors of the flow boxes~$B_i$ cover~$PT^*M$, the point~$\xi$ lies in the interior of some~$B_i$.
Shrinking~$U$ if necessary and following the leaves of~$\mathcal{F}$, we construct a holonomy diffeomorphism $h:U \to V \subseteq \Sigma_i$.
By holonomy invariance, we have $\mu_{\Sigma_i}(V) = \mu_\Sigma(U) >0$.
Since $\psi_i \geq 0$ and $\psi_i \equiv 1$ on~$\Sigma_i$, it follows that
\[
h(\mu) \geq \int_{\Sigma_i^+} \psi_i \, d\mu_{\Sigma_i^+} \geq \mu_{\Sigma_i}(V) >0.
\]

We first show that $\mathcal{K}_h(C) = \{ \mu \in \Curr(M) \mid h(\mu) \leq C\}$ is compact for every $C \geq 0$.
Indeed, as previously, since $\psi_i \geq 0$ and $\psi_i \equiv 1$ on~$\Sigma_i$, we have
\[
\mu_{\Sigma_i}(\Sigma_i) \leq \int_{\Sigma_i^+} \psi_i \, d\mu_{\Sigma_i^+} \leq h(\mu) \leq  C
\]
for every $\mu \in \mathcal{K}_h(C)$.
Thus, the mass of each measure~$\mu_{\Sigma_i}$ is bounded by~$C$.
Therefore, under the above embedding of~$\Curr(M)$, the set~$\mathcal{K}_h(C)$ is a closed subset of~$\prod_i \mathcal{M}_{\leq C}^+(\Sigma_i)$.
Hence it is compact and metrizable.

We claim that $\mathcal{K} \subseteq \mathcal{K}_h(C_0)$ for some $C_0\geq 0$.
Otherwise, there exists a sequence of geodesic currents~$\mu_n \in \mathcal{K}$ such that $h(\mu_n) \to \infty$.
By positive homogeneity of both~$h$ and the intersection form, the geodesic currents $\nu_n=\frac{\mu_n}{h(\mu_n)}$ satisfy \mbox{$h(\nu_n)=1$} and \mbox{$i(\nu_n,\mu_0) \to 0$}.
Since $\mathcal{K}_h(1)$ is compact, a subsequence of~$(\nu_n)$ converges to a geodesic current~$\nu$.
By continuity of~$h$ and the intersection form at~$(\nu,\mu_0)$, see Proposition~\ref{prop:continuity}, we obtain $h(\nu)=1$ and $i(\nu,\mu_0)=0$.
The first equality implies that $\nu$ is nonzero, while the second contradicts the assumption that $\mu_0$ is filling.

Finally, it follows from Proposition~\ref{prop:continuity} that $\mathcal{K}$ is a closed subset of the compact metrizable space~$\mathcal{K}_h(C_0)$.
Therefore, $\mathcal{K}$ is compact and metrizable.
\end{proof}

\begin{remark}
The fact that $\mathcal{K}$ is metrizable also follows from Proposition~\ref{prop:metrizable}.
The proof above provides a slightly different argument.
\end{remark}

\forget

\begin{remark} \label{rem:metrizable}
The space~$\Curr(M)$ of geodesic currents is a Polish space (\ie, a separable, completely metrizable topological space).
Indeed, as observed in the proof of Proposition~\ref{prop:compact}, the space~$\Curr(M)$ can be identified homeomorphically with a closed subset of~$\prod_i \mathcal{M}^+(\Sigma_i)$, where the $\Sigma_i$ are compact surfaces transverse to~$\mathcal{F}$.
By~\cite[15.7.7]{kallenberg}, the space~$\mathcal{M}^+(\Sigma_i)$ of Radon measures on~$\Sigma_i$ is Polish for the vague topology.
The same holds for every closed subset of the product space~$\prod_i \mathcal{M}^+(\Sigma_i)$.
\end{remark}

\forgotten

\forget

\begin{proposition} \label{prop:compact}
Let $M$ be a closed surface endowed with a Finsler metric.
Let $\mu_0$ be a filling geodesic current (\ie, $\mu_0$ is the Liouville geodesic current~$\lambda$).
Then the (metrizable) space
\[
\mathcal{K} = \{ \mu \in \Curr(M) \mid i(\mu_0,\mu) = 1\}
\]
is both compact and sequentially compact.
\end{proposition}

\begin{proof}
Since $PT^*M$ is compact, there exist finitely many flow boxes $B_i \simeq \Sigma_i \times [-1,1]$ whose interiors cover~$PT^*M$, where each~$\Sigma_i$ is a compact surface transverse to~$\mathcal{F}$.
Extend~$\Sigma_i$ to an open transverse surface~$\Sigma_i^+$ and choose a nonnegative function $\psi_i \in C_c(\Sigma_i^+)$ such that $\psi_i \equiv 1$ on~$\Sigma_i$.
The functional $m:\Curr(M) \to \R_+$ defined by
\[
m(\mu) = \sum_i \int_{\Sigma_i^+} \psi_i \, d\mu_{\Sigma_i^+}
\]
is continuous and positively homogeneous.

Let $\mu$ be a nonzero geodesic current.
We claim that $m(\mu) >0$.
Indeed, there exist a transverse surface~$\Sigma \subseteq PT^*M$ and a point~$\xi$ in the support of~$\mu_\Sigma$.
In particular, $\mu_\Sigma(U) >0$ for every neighborhood~$U \subseteq \Sigma$ of~$\xi$.
Since the flow boxes cover~$PT^*M$, the point~$\xi$ lies in some~$B_i$.
Shrinking~$U$ if necessary and following the leaves of~$\mathcal{F}$, we construct a holonomy diffeomorphism $h:U \to V \subseteq \Sigma_i^+$ with $h(\xi) \in \Sigma_i$.
Shrinking~$U$ further if necessary, we may assume that $\psi_i \geq \frac{1}{2}$ on~$V$.
By holonomy invariance, $\mu_{\Sigma_i^+}(V) = \mu_\Sigma(U) >0$.
Thus,
\[
m(\mu) \geq \int_{\Sigma_i^+} \psi_i \, d\mu_{\Sigma_i^+} \geq \tfrac{1}{2} \mu_{\Sigma_i^+}(V) >0.
\]

Now, a geodesic current~$\mu$ is uniquely determined by the family of (nonnegative) Radon measures~$\mu_i = \mu_{\Sigma_i}$ it induces on the transverse surfaces~$\Sigma_i$ along with the holonomy compatibility conditions.
These compatibility conditions are closed.
Thus, $\Curr(M)$ can be identified homeomorphically with a closed subset of the product~$\prod_i \mathcal{M}(\Sigma_i)$, where $\mathcal{M}(\Sigma_i)$ denotes the space of Radon measures on~$\Sigma_i$.

Let us show that $\mathcal{K}_m(C) = \{ \mu \in \Curr(M) \mid m(\mu) \leq C\}$ is compact for every $C>0$.
Since $\psi_i \equiv 1$ on~$\Sigma_i$, we have
\[
\mu_{\Sigma_i}(\Sigma_i) \leq \int_{\Sigma_i^+} \psi_i \, d\mu_{\Sigma_i^+} \leq m(\mu) \leq  C
\]
for every geodesic current~$\mu$ with $m(\mu) \leq C$.
Thus, the mass of each measure~$\mu_{\Sigma_i}$ is bounded by~$C$.
Since $\Sigma_i$ is compact, the space~$\mathcal{M}_{\leq C}(\Sigma_i)$ of Radon measures on~$\Sigma_i$ with mass bounded by~$C$ is compact for the weak$^*$ topology.
Indeed, by the Riesz representation theorem~\cite[Corollary~7.18]{folland}, the space of Radon measures on~$\Sigma_i$ isometrically identifies with the positive cone of the dual space~$C(\Sigma_i)^*$.
The compactness of~$\mathcal{M}_{\leq C}(\Sigma_i)$ for the weak$^*$ topology follows from the Banach--Alaoglu theorem.
Since $\Sigma_i$ is compact, the weak$^*$ and vague topologies agree on~$\mathcal{M}_{\leq C}(\Sigma_i)$.
It follows that $\mathcal{K}_m(C)$ is a closed subset of the compact set~$\prod_i \mathcal{M}_{\leq C}(\Sigma_i)$.
It is therefore compact.

Finally, we claim that there exists $C_0>0$ such that $m(\mu) \leq C_0$ for every $\mu \in \mathcal{K}$.
Otherwise, there exists a sequence of geodesic currents~$\mu_n \in \mathcal{K}$ such that $m(\mu_n) \to \infty$.
By positive homogeneity of both~$m$ and the intersection form, the geodesic currents $\nu_n=\frac{\mu_n}{m(\mu_n)}$ satisfy $m(\nu_n)=1$ and $i(\nu_n,\mu_0) \to 0$.
By compactness of~$\mathcal{K}_m(1)$, a subsequence of~$(\nu_n)$ converges to a geodesic current~$\nu$.
By continuity of~$m$ and the intersection form, $m(\nu)=1$ and $i(\nu,\mu_0)=0$.
in particular, $\nu$ is nonzero, contradicting the assumption that $\mu_0$ is filling.

Consequently, $\mathcal{K}$ is a closed subset of the compact set $\mathcal{K}_m(C_0)$.
It is therefore compact.
Since it is metrizable, it is also sequentially compact.
\end{proof}

\forgotten

We immediately obtain the following compactness result, 
although it will not be needed in this article.
See~\cite[Corollary~5]{Bon88}.

\begin{corollary}
Let $M$ be a closed surface endowed with a Finsler metric. 
Then the projectivized space of geodesic currents
\[
\mathbb{P} \Curr(M) = \left( \Curr(M) \setminus \{0\} \right)  / \, \R_{>0}
\]
is compact.
\end{corollary}

\begin{proof}
Let $\mu_0$ be a filling geodesic current without atoms (\eg, the Liouville geodesic current).
Consider the compact space~$\mathcal{K}$ of Proposition~\ref{prop:compact}.
Since $\mu_0$ is filling, every nonzero geodesic current~$\mu$ satisfies $i(\mu,\mu_0) > 0$.
Consequently, every projective class $[\mu] \in \mathbb{P} \Curr(M)$ contains a unique representative in~$\mathcal{K}$, namely 
$\frac{\mu}{i(\mu,\mu_0)}$.
By Proposition~\ref{prop:continuity}, the map $\mathcal{K} \rightarrow \mathbb{P} \Curr(M)$ sending~$\mu$ to~$[\mu]$ is continuous and surjective.
Since $\mathcal{K}$ is compact, its image~$\mathbb{P} \Curr(M)$ is also compact.
%
\forget
\\
\[
\begin{aligned}
\mathcal{K} & \rightarrow \mathbb{P} \Curr(M) \\
\mu  & \mapsto [\mu]
\end{aligned}
\qquad 
\textrm{ and }
\qquad 
\begin{aligned}
\mathbb{P} \Curr(M) & \rightarrow \mathcal{K} \\
[\mu]  & \mapsto \hat{\mu} 
\end{aligned}
\]
\mbox{ }

\noindent 
\forgotten
\end{proof}

\section{Minimal intersection number on surfaces} \label{sec:min}

In this section, we introduce the intersection and self-intersection numbers of closed curves on a surface.
We then show that, on a closed surface endowed with a Finsler metric, the closed geodesics of minimal length in their free homotopy classes have the least possible intersection and self-intersection numbers.
This is a simple adaptation of results of Freedman, Hass and Scott~\cite[\S3]{FHS} to our setting.

\begin{definition} \label{def:intersection}
Let $c_1$ and~$c_2$ be two closed curves on a surface~$M$.
The \emph{number of intersections} between $c_1$ and~$c_2$ (always counted with multiplicity) is defined as
\[
\#(c_1 \cap c_2) = \card \{ (s_1,s_2) \in S^1 \times S^1 \mid c_1(s_1) = c_2(s_2) \}
\]
where $\card(S)$ denotes the cardinality of a set~$S$.
By Example~\ref{ex:intersection}, if~$\gamma_1$ and~$\gamma_2$ are transverse closed geodesics for some Finsler metric, then
\begin{equation} \label{eq:itt}
\#(\gamma_1 \cap \gamma_2) = i(\gamma_1,\gamma_2).
\end{equation}

Similarly, let $c$ be a closed curve of~$M$.
The \emph{number of self-intersections} of~$c$ (again always counted with multiplicity) is defined as
\[
\#(c) = \frac{1}{2} \card \{ (s_1,s_2) \in S^1 \times S^1 \mid s_1 \neq s_2, c(s_1) = c(s_2) \}.
\]
By Example~\ref{ex:intersection}, for every primitive closed geodesic~$\gamma$,
\begin{equation} \label{eq:it}
\#(\gamma) = \frac{1}{2} \, i(\gamma,\gamma).
\end{equation}
Indeed, each self-intersection of~$\gamma$ corresponds to two ordered pairs of distinct tangent lines (or cotangent lines under the Legendre transform), which accounts for the factor~$\frac{1}{2}$.

Let $\alpha_1$ and~$\alpha_2$ be two unoriented free homotopy classes of closed curves on~$M$.
Their \emph{intersection number} is 
\[
I(\alpha_1,\alpha_2) = \min \{ \#(c_1 \cap c_2) \mid c_1 \in \alpha_1 \text{ and } c_2 \in \alpha_2 \}
\]
where the minimum is taken over closed curves representing~$\alpha_1$ and~$\alpha_2$.
\end{definition}

It will be convenient to introduce the following notion of a nearby push-off curve.

\begin{definition} \label{def:nearby}
Let $M$ be an orientable closed surface endowed with a Finsler metric.
Let $\gamma$ be a closed geodesic.
Fix a smooth unit vector field~$\nu$ along~$\gamma$ transverse to~$\gamma$.
For $\varepsilon > 0$ small enough, define the closed curve
\[
\gamma^\varepsilon(s) = \pi \circ \phi_\varepsilon(\nu(s))
\]
where $\phi_t:UM \to UM$ is the geodesic flow and $\pi:UM \to M$ is the canonical projection.
We will refer to~$\gamma^\varepsilon$ as an \emph{$\varepsilon$-push-off} of~$\gamma$.
\end{definition}

Observe that the number of transverse intersections between the closed geodesics~$\gamma_1$ and~$\gamma_2$ is equal to the number of intersections between a small perturbation~$\gamma_1^\varepsilon$ of~$\gamma_1$ and~$\gamma_2$.
More precisely, for every sufficiently small push-off~$\gamma_1^\varepsilon$ of~$\gamma_1$,
\begin{equation} \label{eq:i-nearby}
i(\gamma_1,\gamma_2) = \#(\gamma_1^\varepsilon \cap \gamma_2).
\end{equation}

The next two propositions are analogues, in our setting, of results of Freedman, Hass and Scott; see~\cite[\S3]{FHS}.

\medskip

The first result asserts that two shortest closed curves in their free homotopy classes have the least possible intersection number; see~\cite[Theorem~3.3]{FHS}.

\begin{proposition} \label{prop:i}
Let $M$ be an orientable closed surface endowed with a Finsler metric.
Let $\gamma_1$ and~$\gamma_2$ be two noncontractible closed geodesics of minimal length in their unoriented free homotopy classes~$\alpha_1$ and~$\alpha_2$.
Then,
\begin{equation} \label{eq:i}
I(\alpha_1,\alpha_2) =  i(\gamma_1,\gamma_2).
\end{equation}
\forget
Moreover, 
\begin{enumerate}
\item if $\gamma_1$ and~$\gamma_2$ are not iterates of the same primitive closed geodesic, then
\begin{equation} \label{eq:i-sharp}
i(\gamma_1,\gamma_2) = \#(\gamma_1 \cap \gamma_2)
\end{equation}
\item if $\gamma_1=\gamma^{p_1}$ and~$\gamma_2=\gamma^{p_2}$ are iterates of the same primitive closed geodesic~$\gamma$ with $p_1,p_2 \geq 1$, then
\begin{equation}
i(\gamma_1,\gamma_2) = 2 p_1 p_2 \, \#_t(\gamma).
\end{equation}
\end{enumerate}


Moreover, for every $\varepsilon$-push-off~$\gamma_1^\varepsilon$ of~$\gamma_1$ with $\varepsilon>0$ small enough,
\begin{equation} \label{eq:i-nearby}
i(\gamma_1,\gamma_2) = \#(\gamma_1^\varepsilon \cap \gamma_2).
\end{equation}
\forgotten
\end{proposition}

\begin{proof}
We may assume that both $\gamma_1$ and~$\gamma_2$ are noncontractible.
Let $l_i$ be the geodesic line in the universal cover~$\bar{M}$ of~$M$ projecting to the image of~$\gamma_i$.
Since $\gamma_i$ is length-minimizing in its homotopy class, the geodesic line~$l_i$ is also length-minimizing (between any pair of its points); see Proposition~\ref{prop:classical}. 
Thus, the geodesic lines~$g \cdot l_2$ and~$l_1$, where $g \in \pi_1(M)$, have at most one transverse intersection point.
Denote by~$g_i \in \pi_1(M)$ the homotopy class representing~$\gamma_i$ acting on~$l_i$.
Fix a semi-open arc~$\delta_i$ of~$l_i$ which is a fundamental domain for the action of~$\langle g_i \rangle$ on~$l_i$.
The intersection form~$i(\gamma_i,\gamma_2)$ between~$\gamma_1$ and~$\gamma_2$ is equal to the number of elements~$g \in \pi_1(M)$ such that the arc~$g \cdot \delta_2$ transversally intersects~$\delta_1$.
Alternatively, the set
\[
\{ g \in \pi_1(M) \mid g \cdot l_2 \mbox{ transversally intersects } l_1 \}
\]
is a union of double cosets $\langle g_1 \rangle \, g \, \langle g_2 \rangle$ for some $g \in \pi_1(M)$.
The intersection form~$i(\gamma_1,\gamma_2)$ is equal to the number of these double cosets; see~\cite[\S3]{FHS}.
That is,
\[
i(\gamma_1,\gamma_2) = \card \{ \langle g_1 \rangle \, g \, \langle g_2 \rangle \mid g \cdot l_2 \mbox{ transversally intersects } l_1 \}.
\]

Let $c_1$ and~$c_2$ be two closed curves representing~$\alpha_1$ and~$\alpha_2$.
Since the curves~$c_i$ and $\gamma_i$ are freely homotopic, there exists a lift of~$c_i$ to~$\bar{M}$ whose image~$l'_i$ is at finite Hausdorff distance from the geodesic line~$l_i$.
In particular, $l_i$ and~$l'_i$ have the same ends.
Thus, if $g \cdot l_2$ transversally intersects~$l_1$, then $g \cdot l'_2$ intersects~$l'_1$ (at least once).
Hence,
\[
i(\gamma_1,\gamma_2) \leq \#(c_1 \cap c_2).
\]
That is,
\[
i(\gamma_1,\gamma_2) \leq I(\alpha_1, \alpha_2).
\]

Conversely, by~\eqref{eq:i-nearby}, we have
\[
i(\gamma_1,\gamma_2) = 
\#(\gamma_1^\varepsilon \cap \gamma_2) \geq I(\alpha_1,\alpha_2).
\]
Therefore, 
\[
I(\alpha_1,\alpha_2) = i(\gamma_1,\gamma_2) = \#(\gamma_1^\varepsilon \cap \gamma_2). \qedhere
\]
\forget
If $\gamma_1$ and~$\gamma_2$ are transverse, their number of transverse intersections is equal to their number of intersections.
That is,
\[
i(\gamma_1,\gamma_2) = 
\#(\gamma_1 \cap \gamma_2).
\]

If $\gamma_1=\gamma^{p_1}$ and $\gamma_2=\gamma^{p_2}$ with $p_1,p_2 \geq 1$, where $\gamma$ is a primitive closed geodesic, then
\[
i(\gamma_1,\gamma_2) = p_1 p_2 \, i(\gamma,\gamma) = 2 p_1 p_2 \, \#_t(\gamma)
\]
where the last equality comes from~\eqref{eq:it}.
\forgotten
\end{proof}

The second result asserts that a shortest closed curve in its free homotopy class has the least possible self-intersection number; see~\cite[Theorem~3.2]{FHS}.

\begin{proposition} \label{prop:ii}
Let $M$ be an orientable closed surface endowed with a Finsler metric.
Let $\gamma$ be a noncontractible closed geodesic of minimal length in its free homotopy class.
Then, for every closed curve~$c$ freely homotopic to~$\gamma$,
\[
i(\gamma,\gamma) \leq 2 \, \#(c).
\]
In particular, if $\gamma$ is primitive, then
\[
\#(\gamma) \leq \#(c).
\]
\end{proposition}

\begin{proof}
We argue as in the proof of the previous proposition.
We may assume that $\gamma$ is noncontractible.
Let $l$ be the geodesic line in the universal cover~$\bar{M}$ of~$M$ projecting to the image of~$\gamma$.
Since $\gamma$ is length-minimizing in its homotopy class, the geodesic line~$l$ is also length-minimizing; see~\cite{FHS}.
Thus, the geodesic lines~$g \cdot l$ and~$l$, where $g \in \pi_1(M)$, have at most one transverse intersection point.
Denote by~$h \in \pi_1(M)$ the homotopy class representing~$\gamma$ acting on~$l$.
The set
\[
\{ g \in \pi_1(M) \mid g \cdot l \mbox{ transversally intersects } l \}
\]
is a union of double cosets $\langle h \rangle \, g \, \langle h \rangle$ for some $g \in \pi_1(M)$.
The intersection form~$i(\gamma,\gamma)$ is equal to the number of these double cosets; see~\cite[\S3]{FHS}.
That is,
\[
i(\gamma,\gamma) = \card \{ \langle h \rangle \, g \, \langle h \rangle \mid g \cdot l \mbox{ transversally intersects } l \}.
\]

Let $c$ be a closed curve of~$M$ freely homotopic to~$\gamma$.
As previously, there exists a lift of~$c$ to~$\bar{M}$ whose image~$l'$ is at finite Hausdorff distance from the geodesic line~$l$.
In particular, $l$ and~$l'$ have the same ends.
Thus, if $g \cdot l$ transversally intersects~$l$, then $g \cdot l'$ intersects~$l'$.
Note that if $g \cdot l'$ intersects~$l'$, so does~$g^{-1} \cdot l'$.

Hence,
\[
i(\gamma,\gamma) \leq 2 \, \#(c). \qedhere
\]
\end{proof}

\section{Wall systems} \label{sec:wall}

In this section, we introduce the notion of wall systems, which will allow us to discretize Finsler metrics on surfaces in Section~\ref{sec:discretization}. 
We then show that the discrete marked length spectrum of a geodesic wall system, whose geodesic walls are length-minimizing in their free homotopy classes, can be expressed in terms of the intersection form with minimizing closed geodesics.

\medskip

We first need the notion of wall system introduced in~\cite{Cos} and further exploited in~\cite{CS}.
The following definitions are adapted to our setting.

\begin{definition} \label{def:wall}
A \emph{wall system}~$\W=(c_i,w_i)_{1\leq i \leq n}$ on a closed surface~$M$ is a finite collection of immersed closed curves~$(c_i)_{1 \leq i \leq n}$ with transverse intersections and self-intersections, together with positive \emph{weights}~$(w_i)_{1\leq i \leq n}$.
We will refer to the closed curves~$c_i$ as the \emph{walls} of~$\W$.
An \emph{unweighted wall system} is a wall system for which all the weights are equal to~$1$.

A wall system is \emph{geodesic} with respect to a Finsler metric on~$M$ if all its walls are geodesic.
In particular, since the walls are required to have only transverse intersections and self-intersections, the curves forming the geodesic walls are primitive and pairwise transverse.
A geodesic wall system is \emph{minimal} if each of its geodesic walls is length-minimizing in its unoriented free homotopy class.

The \emph{discrete area} of a wall system~$\W=(c_i,w_i)_{1 \leq i \leq n}$ is the total number of (transverse) intersections and self-intersections of its walls, counted with multiplicity and weights. 
That is,
\begin{align*}
\area(\W) & = \sum_{1 \leq i < j \leq n} w_i w_j \, \#(c_i \cap c_j) + \sum_{i=1}^n w_i^2 \, \#(c_i) \notag \\
 & = \frac{1}{2} \sum_{1 \leq i \neq j \leq n} w_i w_j \, \#(c_i \cap c_j) + \sum_{i=1}^n w_i^2 \, \#(c_i). \label{eq:area-discrete}
\end{align*}

The \emph{discrete length} of a curve~$c$ of~$M$ with respect to~$\W$ is given by
\[
\length_{\W}(c) = \sum_{i=1}^n w_i \, \#(c \cap c_i).
\]
The \emph{discrete distance} induced by~$\W$ is a pseudo-distance on~$M$ defined by
\[
d_\W(x,y) = \min_c \length_\W(c)
\]
where the minimum is taken over all paths~$c$ joining~$x$ to~$y$.

The \emph{discrete marked length spectrum} $\MLS_{\W}:\mathcal{C} \to [0,\infty)$ induced by~$\W$ is defined by
\[
\MLS_{\W}(\alpha) = \min_{c \in \alpha} \length_{\W}(c)
\]
where the minimum is taken over all closed curves~$c$ representing the unoriented free homotopy class~$\alpha$.
Compare with~\eqref{eq:MLS}.
Equivalently, the minimum may be taken over all representatives avoiding the vertices of the graph determined by~$\W$.
This observation is needed to translate the main theorem of~\cite{GS} into our setting; see Theorem~\ref{theo:GS}.

A geodesic wall system~$\W=(\sigma_i,w_i)_{1 \leq i \leq n}$ gives rise to a geodesic current, also denoted by~$\W$, defined as the sum of the geodesic currents induced by the (primitive) closed geodesics~$\sigma_i$ with weight~$w_i$.
Namely,
\[
\W = \sum_{i=1}^n w_i \, \sigma_i 
\]
Conversely, a positive linear combination of geodesic currents induced by primitive closed geodesics defines a geodesic wall system, up to wall parametrizations.
\end{definition}

We can express the discrete length, the discrete area and the discrete marked length spectrum of a geodesic wall system in terms of the intersection form as follows.

\begin{proposition} \label{prop:MLSi}
Let $M$ be an orientable closed surface endowed with a Finsler metric.
Let $\W$ be a geodesic wall system on~$M$.
Let~$\alpha$ be an unoriented free homotopy class.
Then the following properties hold:
\begin{enumerate}
\item For every closed geodesic~$\gamma$ transverse to~$\W$, \label{i1}
\begin{equation*}
\length_{\W}(\gamma) = i(\gamma,\W). \label{eq:r1}
\end{equation*}
\item \label{i2}
\begin{equation*}
\area(\W) = \frac{1}{2} \, i(\W,\W). \label{eq:r2}
\end{equation*}
\item Every closed geodesic~$\gamma$ representing~$\alpha$ satisfies \label{i22}
\[
\MLS_{\W}(\alpha) \leq i(\gamma,\W).
\]
\item There exits a closed geodesic~$\gamma$ representing~$\alpha$ such that \label{i3}
\[
i(\gamma,\W) =\MLS_{\W}(\alpha).
\]
\item Suppose that $\W$ is minimal. \label{i4}
Every closed geodesic~$\gamma$ of minimal length representing~$\alpha$ satisfies
\begin{equation*} \label{eq:MLSi}
i(\gamma,\W) = \MLS_{\W}(\alpha).
\end{equation*}
\end{enumerate}
%
\end{proposition}

\begin{proof}
Write $\W=(\sigma_i,w_i)_{1 \leq i \leq n}$.
As geodesic current, $\W = \sum_{i=1}^n w_i \, \sigma_i$.

\eqref{i1}. Let $\gamma$ be a closed geodesic transverse to~$\W$.
Since the geodesics~$\gamma$ and~$\sigma_i$ are transverse, the relation~\eqref{eq:itt} gives
\[
\#(\gamma \cap \sigma_i) = i(\gamma,\sigma_i).
\]
Hence, by bilinearity of the intersection form,
\[
\length_\W(\gamma) = \sum_{i=1}^n w_i \, \#(\gamma \cap \sigma_i) = \sum_{i=1}^n w_i \, i(\gamma, \sigma_i) = i(\gamma,\W).
\]

\eqref{i2}. Similarly, since the geodesics~$\sigma_i$ are pairwise transverse and primitive, the relations~\eqref{eq:itt} and~\eqref{eq:it} give
\begin{align*}
\#(\sigma_i,\sigma_j) & = i(\sigma_i,\sigma_j) \mbox{ for } i \neq j \\
\#(\sigma_i) & = \frac{1}{2} \, i(\sigma_i,\sigma_i).
\end{align*}
Hence,
\begin{align*}
\area(\W) & = \frac{1}{2} \sum_{1 \leq i \neq j \leq n} w_i w_j \, \#(\sigma_i \cap \sigma_j) + \sum_{i=1}^n w_i^2 \, \#(\sigma_i) \\
 & = \frac{1}{2} \sum_{1 \leq i, j \leq n} w_i w_j \, i(\sigma_i, \sigma_j) \\
 & =\frac{1}{2} \, i(\W,\W).
\end{align*}

\eqref{i22}. Let $\gamma$ be a closed geodesic representing~$\alpha$.
By~\eqref{eq:i-nearby}, every sufficiently small push-off~$\gamma^\varepsilon$ of~$\gamma$ satisfies
\[
\#(\gamma^\varepsilon \cap \sigma_i) = i(\gamma,\sigma_i).
\]
Thus, 
\[
\length_\W(\gamma^\varepsilon) = \sum_{i=1}^n w_i \, \#(\gamma^\varepsilon \cap \sigma_i) = \sum_{i=1}^n w_i \, i(\gamma,\sigma_i) = i(\gamma,\W).
\]
Since $\gamma^\varepsilon$ represents~$\alpha$, it follows that
\begin{equation} \label{eq:22}
\MLS_{\W}(\alpha) \leq \length_\W(\gamma^\varepsilon) = i(\gamma,\W).
\end{equation}

\medskip

\eqref{i3}. Let $c$ be closed curve of minimal discrete length in~$\alpha$, so that
\[
\MLS_{\W}(\alpha) = \length_\W(c).
\]
Apply the disk-flow~\cite{HS} simultaneously to~$c$ and the geodesics~$\sigma_i$ of~$\W$.
Since the curves~$\sigma_i$ are geodesic, the wall system~$\W$ remains fixed under the disk-flow.
On the other hand, the curve~$c$ evolves through a subsequence of piecewise geodesic curves~$c_n$ converging to a closed geodesic~$\gamma$ representing~$\alpha$.

Since the disk-flow does not increase the number of intersections, see~\cite{HS}, and the initial curve~$c$ has minimal discrete length among the curves representing~$\alpha$, we have
\[
\length_\W(c_n) = \length_\W(c) = \MLS_{\W}(\alpha).
\]

Write $\gamma=\gamma_0^k$, where $\gamma_0$ is primitive and $k \geq 1$.
Consider a sufficiently small immersed annular neighborhood~$N$ of~$\gamma_0$.
For $N$ sufficiently small, the pullback of~$\W$ to~$N$ consists of transverse geodesic arcs~$a_j$, with weights~$w(a_j)$, and possibly some multiple of~$\gamma_0$.
Denote by $m=\sum_j w(a_j)$ the total weight of the transverse arcs.
For $n$ sufficiently large, the curve~$c_n$ lies in~$N$ and winds $k$ times around it.
Consequently, this curve intersects each transverse arc~$a_j$ at least $k$ times.
Thus,
\[
\MLS_{\W}(\alpha) = \length_\W(c_n) \geq k m.
\]
On the other hand, by construction, 
\[
i(\gamma,\W) = km.
\]
Combining the two inequalities yields
\[
i(\gamma,\W) \leq \MLS_\W(\alpha).
\]
The reverse inequality follows from~\eqref{eq:22}.
Hence,
\[
i(\gamma,\W) = \MLS_{\W}(\alpha).
\]

\medskip


\eqref{i4}. Suppose that $\W$ is minimal.
Let $\gamma$ be a closed geodesic of minimal length representing~$\alpha$.
Since both~$\gamma$ and~$\sigma_i$ are length-minimizing in their unoriented free homotopy classes, Proposition~\ref{prop:i} implies that for every closed curve~$c$ representing~$\alpha$,
\[
i(\gamma,\sigma_i) = I(\alpha,[\sigma_i]) \leq \#(c \cap \sigma_i).
\]
Therefore, 
\[
i(\gamma,\W) = \sum_{i=1}^n w_i \, i(\gamma,\sigma_i) \leq \sum_{i=1}^n w_i \, \#(c \cap \sigma_i) = \length_\W(c).
\]
Taking the infimum over all closed curves~$c$ representing~$\alpha$, we obtain
\[
i(\gamma,\W) \leq \MLS_{\W}(\alpha).
\]
The reverse inequality follows from~\eqref{eq:22}.
Hence,
\[
i(\gamma,\W) = \MLS_{\W}(\alpha). \qedhere
\]
\end{proof}

\section{Uncrossing and straightening wall systems} \label{sec:uncrossing}

In this section, we show that a wall system can be uncrossed and straightened into a geodesic wall system without changing its discrete marked length spectrum or increasing its discrete area. 

\medskip

Behind the following theorem is a result of de Graaf and Schrijver~\cite{GS} (see also~\cite{schrijver}) who used Reidemeister-type moves to uncross the walls of a wall system.
Here, the notations are adapted to our setting.

\begin{theorem} \label{theo:GS}
Let $\W$ be a wall system with integral (resp. rational) weights on an orientable closed surface~$M$.
Fix a Finsler metric~$F$ on~$M$.
Then there exists a minimal geodesic wall system~$\overline{\W}$ with integral (resp. rational) weights
such that
\begin{align*}
\MLS_{\overline{\W}} & = \MLS_{\W} \\
\area(\overline{\W}) & \leq \area(\W).
\end{align*}
\end{theorem}

\begin{proof}
Since the discrete marked length spectrum and the discrete area of a wall system are homogeneous of degree one and two in the weights, we may assume that the weights of~$\W=(c_i,w_i)_{1 \leq i \leq n}$ are integral.

We may also assume that $\W$ is unweighted.
Otherwise, we replace each wall~$c_i$ by $w_i$ sufficiently close parallel copies.
The resulting unweighted wall system has the same discrete marked length spectrum and the same discrete area as~$\W$.

Now, using Reidemeister-type moves, de Graaf and Schrijver~\cite{GS} (see also~\cite{schrijver}) showed that the graph determined by~$\W$ can be decomposed into closed curves $(a_k)_{1 \leq k \leq q}$ where each edge of the graph is traversed exactly once by exactly one of the curves~$a_k$, such that for every unoriented free homotopy class~$\alpha \in \mathcal{C}$,
\begin{equation} \label{eq:GS}
\MLS_{\W}(\alpha) = \sum_{k=1}^q I(\alpha,\alpha_k),
\end{equation}
where $\alpha_k$ denotes the unoriented free homotopy class of~$a_k$.

Moreover, the discrete area of~$\W$ is equal to the total number of intersections and self-intersections of the curves~$(a_k)_{1 \leq k \leq q}$ .
That is, 
\begin{equation} \label{eq:GS2}
\area(\W) = \frac{1}{2} \sum_{1 \leq k \neq k' \leq q} \#(a_k \cap a_{k'}) + \sum_{k=1}^q \#(a_k).
\end{equation}

We may assume that all the curves~$a_k$ are noncontractible.
Indeed, if $a_k$ is contractible, then $I(\cdot,\alpha_k)=0$.
Thus, by~\eqref{eq:GS}, discarding the contractible curves does not affect the discrete marked length spectrum, while it can only decrease the discrete area.

Let $\beta_k$ be the primitive unoriented free homotopy classes underlying~$\alpha_k$.
After relabelling, we may assume that the distinct primitive classes occurring among $\beta_1,\dots,\beta_q$ are $\beta_1,\dots,\beta_m$.
Thus, for every $k \in \{1,\dots,q\}$, there is a unique $j_k \in \{1,\dots,m\}$ such that $\beta_k = \beta_{j_k}$.
For every~$j \in \{1,\dots,m\}$, choose a length-minimizing primitive closed geodesic~$\sigma_j$ representing~$\beta_j$.
There is a unique integer~$r_k \geq 1$ such that $\alpha_k$ is represented by
\[
\bar{a}_k = \sigma_{j_k}^{r_k}
\]
Since $M$ is an orientable surface, every positive iterate of a length-minimizing primitive closed geodesic is length-minimizing in the corresponding unoriented free homotopy class; see~Proposition~\ref{prop:classical}.
Therefore, $\bar{a}_k$ is length-minimizing in~$\alpha_k$.

For every $j \in \{1,\dots,m\}$, let 
\[
w_j = \sum_{k \mid j_k=j} r_k.
\]
Define the minimal geodesic wall system
\[
\overline{\W}=(\sigma_j,w_j)_{1 \leq j \leq m}
\]
with integral weights.
Equivalently, as geodesic currents,
\begin{equation} \label{eq:WW}
\overline{\W}= \sum_{j=1}^m w_j \, \sigma_j = \sum_{j=1}^m \sum_{k \mid j_k= j} r_k \, \sigma_{j_k} = \sum_{k=1}^q r_k \, \sigma_{j_k} = \sum_{k=1}^q \bar{a}_k.
\end{equation}

Let us show that $\W$ and~$\overline{\W}$ have the same discrete marked length spectrum.
Let $\alpha \in \mathcal{C}$ be a nontrivial unoriented free homotopy class.
Choose a length-minimizing closed geodesic~$\gamma$ representing~$\alpha$.
Since both~$\gamma$ and~$\bar{a}_k$ are length-minimizing in their unoriented free homotopy classes, Proposition~\ref{prop:i} gives
\[
I(\alpha,\alpha_k) = i(\gamma,\bar{a}_k).
\]
By the formula~\eqref{eq:GS} and the expression~\eqref{eq:WW} of~$\overline{\W}$,
\[
\MLS_{\W}(\alpha) = \sum_{k=1}^q I(\alpha,\alpha_k) = \sum_{k=1}^q i(\gamma,\bar{a}_k) = i(\gamma,\overline{\W}). 
\]
Since the walls of~$\overline{\W}$ are length-minimizing in their unoriented free homotopy classes, Proposition~\ref{prop:MLSi}.\eqref{i4} yields
\[
\MLS_{\overline{\W}}(\alpha) = i(\gamma,\overline{\W}).
\]
Hence,
\[
\MLS_{\overline{\W}}(\alpha) =\MLS_{\W}(\alpha).
\]

It remains to compare the discrete areas.
By the discrete area formula in Proposition~\ref{prop:MLSi}.\eqref{i2} and the expression~\eqref{eq:WW} of~$\overline{\W}$,
\[
\area(\overline{\W}) = \frac{1}{2} \, i(\overline{\W},\overline{\W}) = \frac{1}{2} \sum_{1 \leq k,k' \leq q} i(\bar{a}_k,\bar{a}_{k'}).
\]
Since the geodesics~$\bar{a}_k$ are length-minimizing in their unoriented free homotopy classes, Proposition~\ref{prop:i} and Proposition~\ref{prop:ii} yield
\begin{gather*}
i(\bar{a}_k,\bar{a}_{k'}) = I(\alpha_k,\alpha_{k'}) \leq \#(a_k \cap a_{k'})  \mbox{ for } k \neq k' \\
i(\bar{a}_k,\bar{a}_k) \leq 2 \, \#(a_k).
\end{gather*}
It follows that
\[
\area(\overline{\W}) \leq \frac{1}{2} \sum_{1 \leq k \neq k' \leq q} \#(a_k \cap a_{k'}) + \sum_{k=1}^q \#(a_k) = \area(\W)
\]
where the last equality comes from~\eqref{eq:GS2}.
\end{proof}

\forget

\begin{theorem} \label{theo:GS}
Let $\W$ be a wall system on a closed surface~$M$.
Then, the graph of~$M$ (whose vertices have even degree) formed by the images of the curves of~$\W$ decomposes as a union of closed curves $(a_k)_{1 \leq k \leq q}$ so that each edge of the graph is traversed by exactly one~$a_k$ and by that~$a_k$ exactly once, such that for every unoriented free homotopy class~$\alpha \in \mathcal{C}$,
\[
\MLS_{\W}(\alpha) = \sum_{k=1}^q I(\alpha,\alpha_k),
\]
where $\alpha_k$ is the unoriented free homotopy class of~$a_k$.

Furthermore, the discrete area of~$\W$ is equal to the total number of self-intersections of the collection of curves~$(a_k)_{1 \leq k \leq q}$ .
That is, 
\begin{equation} \label{eq:GS2}
\area(\W) = \frac{1}{2} \sum_{1 \leq k \neq k' \leq q} \#(a_k \cap a_{k'}) + \sum_{k=1}^q \#(a_k).
\end{equation}
\end{theorem}

The following construction shows how to turn a self-transverse wall system into a geodesic wall system.

\begin{definition} \label{def:new}
Let $M$ be a closed surface with a self-transverse wall system~$\W$.
Take a Finsler metric~$F$ without conjugate points on~$M$.
Consider the closed curves $a_k$ decomposing the wall system~$\W$ given by Theorem~\ref{theo:GS}.
We can assume that these curves are noncontractible (otherwise, we discard the contractible ones, which does not increase the discrete area and does not affect the marked length spectrum; see Theorem~\ref{theo:GS}).
Let~$\bar{a}_k$ be the privileged closed geodesic in the same unoriented free homotopy class as~$a_k$.
By construction, there exists a unique collection of distinct privileged primitive closed geodesics~$(\sigma_j)_{1 \leq j \leq m}$ such that every closed geodesics~$\bar{a}_k$ decomposes as
\[
\bar{a}_k = \sigma_{j_k}^{r_k}
\]
for some unique integers $j_k \in \{1,\dots,m\}$ and positive integers $r_k$.

Define a new (geodesic) wall system
\[
\overline{\W}=(\gamma_j)_{1 \leq j \leq m}
\]
as the collection of the pairwise transverse closed geodesics
\[
\gamma_j = \sigma_j^{p_j}
\]
where $p_j = \sum_{k \mid j_k=j} r_k$.
\end{definition}

By the relations~\eqref{eq:itt} and~\eqref{eq:it}, we have 
\begin{gather*}
\#(\gamma_j \cap \gamma_{j'}) = i(\gamma_j,\gamma_{j'}) \mbox{ for } j \neq j' \\
\#_t(\gamma_j) = \frac{1}{2} \, i(\gamma_j,\gamma_j).
\end{gather*}
Hence the following discrete area formula for~$\overline{\W}$
\begin{equation} \label{eq:AWbar}
\area(\overline{\W}) = \frac{1}{2} \sum_{1 \leq j \neq j' \leq m} \#(\gamma_j \cap \gamma_{j'}) + \sum_{j=1}^m \#_t(\gamma_j) = \frac{1}{2} \sum_{1 \leq j, j' \leq m} i(\gamma_j, \gamma_{j'}).
\end{equation}
Compare with the formula~\eqref{eq:area-discrete}. \\


The following result shows that uncrossing and straightening a wall system as in Definition~\ref{def:new} does not increase its discrete area.

\begin{proposition} \label{prop:new-area}
Let $M$ be an orientable closed surface with a self-transverse wall system~$\W$.
Fix a Finsler metric~$F$ without conjugate points on~$M$.
Let~$\overline{\W}$ be the geodesic wall system derived from~$\W$ defined in Definition~\ref{def:new}.
Then,
\[
\area(\overline{\W}) \leq \area(\W).
\]
\end{proposition}

\begin{proof}
Let us keep the notations of Definition~\ref{def:new}.
Denote by~$(c_i)_{1 \leq i \leq n}$ the collection of closed curves of~$\mathcal{W}$ and by~$(\gamma_j)_{1 \leq j \leq m}$ the collection of closed geodesics of~$\overline{\mathcal{W}}$.

From the discrete area formula~\eqref{eq:AWbar} and the relation $\gamma_j = \sigma_j^{p_j}$ where $p_j = \sum_{k \mid j_k=j} r_k$, we deduce that
\begin{align*}
\area(\overline{\W}) & = \frac{1}{2} \sum_{1 \leq j,j' \leq m} i(\sigma_j^{p_j},\sigma_{j'}^{p_{j'}}) \\
 & = \frac{1}{2} \sum_{1 \leq j,j' \leq m} \sum_{k \mid j_k=j} \sum_{k' \mid j_{k'}=j'} \, i(\sigma_{j_k}^{r_k},\sigma_{j_{k'}}^{r_{k'}}) \\
 & = \frac{1}{2} \sum_{1 \leq k,k' \leq q} i(\bar{a}_k,\bar{a}_{k'})
\end{align*}
after rearranging the terms.
Now, by Proposition~\ref{prop:i}, Eq.~\eqref{eq:i}, and Proposition~\ref{prop:ii}, we have
\begin{gather*}
i(\bar{a}_k,\bar{a}_{k'}) \leq \#(a_k \cap a_{k'})  \mbox{ for } k \neq k' \\
i(\bar{a}_k,\bar{a}_k) \leq 2 \, \#_t(a_k).
\end{gather*}
These inequalities along with Theorem~\ref{theo:GS}, Eq.~\eqref{eq:GS2}, yield
\[
\area(\overline{\W}) \leq \frac{1}{2} \sum_{1 \leq k \neq k' \leq q} \#(a_k \cap a_{k'}) + \sum_{k=1}^q \#_t(a_k) = \area(\W).
\]
\end{proof}

The following result shows that uncrossing and straightening a wall system as in Definition~\ref{def:new} does not affect its marked length spectrum.

\begin{proposition} \label{prop:new-mls}
Let $M$ be an orientable closed surface with a self-transverse wall system~$\W$.
Fix a Finsler metric~$F$ without conjugate points on~$M$.
Let~$\overline{\W}$ be the geodesic wall system derived from~$\W$ defined in Definition~\ref{def:new}.
Then,
\[
\MLS_{\W} = \MLS_{\overline{\W}}.
\]
\end{proposition}

\begin{proof}
Let us keep the notations of Definition~\ref{def:new}.
Denote by~$(c_i)_{1 \leq i \leq n}$ the collection of closed curves of~$\mathcal{W}$ and by~$(\gamma_j)_{1 \leq j \leq m}$ the collection of closed geodesics of~$\overline{\mathcal{W}}$.
Let $\alpha \in \mathcal{C}$ be an unoriented free homotopy class of~$M$.
Let $\gamma$ be a closed geodesic representing~$\alpha$.
By Theorem~\ref{theo:GS} and Proposition~\ref{prop:i}, we derive
\[
\MLS_{\W}(\alpha) = \sum_{k=1}^q I(\alpha,\alpha_k) = \sum_{k=1}^q i(\gamma,\sigma_{j_k}^{r_k}) 
\]
after observing that $\bar{a}_k=\sigma_{j_k}^{r_k}$ is a closed geodesic representing~$\alpha_k$.
By rearranging the terms and using the relation $\gamma_j = \sigma_j^{p_j}$ where $p_j = \sum_{k \mid j_k=j} r_k$, we obtain
\[
\MLS_{\W}(\alpha) =\sum_{j=1}^m \sum_{k \mid j_k = j} i(\gamma,\sigma_{j_k}^{r_k}) = \sum_{j=1}^m i(\gamma,\gamma_j) = \MLS_{\overline{\W}}(\alpha)
\]
where the last equality follows from Proposition~\ref{prop:geod}.
\end{proof}

\forgotten

\section{Discrete marked length spectrum and discrete area} \label{sec:discrete}

In this section, we establish a discrete version of the main theorem, where both length structures are induced by wall systems.
Unlike the results involving Anosov metrics, this result also applies to the two-torus.

\begin{theorem} \label{theo:discrete}
Let $M$ be an orientable closed surface endowed with a Finsler metric.
Let $\W_0$ and~$\W$ be two wall systems on~$M$, with $\W_0$ minimal.
If $\MLS_{\W_0} \leq \MLS_{\W}$ then 
\[
\area(\W_0) \leq \area(\W).
\]
\end{theorem}

\begin{proof}
Write $\W_0=(\gamma_i^0,w_i^0)_{1 \leq i \leq n}$.
Denote by $\overline{\W}=(\bar{\gamma}_j,\bar{w}_j)_{1 \leq j \leq m}$ the minimal geodesic wall system obtained from~$\W$ by Theorem~\ref{theo:GS}.
Viewing these geodesic wall systems as geodesic currents, we have
\[
\W_0 = \sum_{i=1}^n w_i^0 \gamma_i^0 \qquad \textrm{ and } \qquad \overline{\W} = \sum_{j=1}^m \bar{w}_j \bar{\gamma}_j.
\]
By Theorem~\ref{theo:GS}, we have
\[
\MLS_{\W_0} \leq  \MLS_{\W} = \MLS_{\overline{\W}}.
\]
By Proposition~\ref{prop:MLSi}.\eqref{i4}, we obtain
\[
i(\gamma,\W_0) \leq i(\gamma,\overline{\W})
\]
for every closed geodesic~$\gamma$ of minimal length in its unoriented free homotopy class.

Now, since the walls~$\gamma_i^0$ of~$\W_0$ are length-minimizing in their unoriented free homotopy classes, the previous inequality and the discrete area formula~\eqref{eq:r2} give
\[
2  \area(\W_0) = \sum_{i=1}^n w_i^0 \, i(\gamma_i^0,\W_0) \leq \sum_{i=1}^n w_i^0 \, i(\gamma_i^0,\overline{\W}) = i(\W_0,\overline{\W}).
\]
Similarly, since the walls~$\bar{\gamma}_j$ of~$\overline{\W}$ are length-minimizing in their unoriented free homotopy classes, we obtain
\[
 i(\overline{\W},\W_0) = \sum_{j=1}^m \bar{w}_j \, i(\bar{\gamma}_j,\W_0) \leq \sum_{j=1}^m \bar{w}_j \, i(\bar{\gamma}_j,\overline{\W}) =  2 \area(\overline{\W}).
\]
By symmetry of the intersection form, 
\[
2  \area(\W_0) \leq i(\W_0,\overline{\W}) =  i(\overline{\W},\W_0) \leq 2 \area(\overline{\W}).
\]
Therefore,
\[
\area(\W_0) \leq \area(\overline{\W}).
\]
Finally, Theorem~\ref{theo:GS} yields
\[
\area(\W_0) \leq \area(\overline{\W}) \leq \area(\W). \qedhere
\]
\end{proof}

\section{Liouville geodesic current approximation for Anosov geodesic flow} \label{sec:approximation_anosov}

We present two approximation results for the Liouville geodesic current of a Finsler metric with Anosov geodesic flow.
This is the only part where the Anosov assumption on the geodesic flow is required.
As an application of the second approximation result, we obtain an extension of the Croke--Dairbekov theorem in which the marked length spectra need only be compared on free homotopy classes represented by null-homologous closed curves.

\medskip

The first result, based on Sigmund's density theorem~\cite[Theorem~1]{Sig} (see also~\cite[Proposition~2]{Bon88}), approximates the Liouville geodesic current by positive multiples of geodesic currents induced by closed geodesics.

\begin{theorem} \label{theo:dense}
Let $M$ be a closed surface endowed with a Finsler metric~$F$ whose geodesic flow is Anosov.
Then the positive multiples of the geodesic currents induced by primitive closed geodesics are dense in the space of geodesic currents of~$M$.

In particular, the Liouville geodesic current of~$F$ is a limit of positive multiples of geodesic currents induced by primitive closed geodesics.
\end{theorem}

The second, based on an equidistribution theorem of Coles and Sharp~\cite[Theorem~1.4]{coles-sharp}, allows the approximating currents to the Liouville geodesic current to be chosen as positive linear combinations of geodesic currents induced by null-homologous closed geodesics.


\begin{theorem} \label{theo:dense+}
Let $M$ be a closed surface endowed with a Finsler metric~$F$ whose geodesic flow is Anosov.
Then the Liouville geodesic current of~$F$ is a limit of positive linear combinations of geodesic currents induced by primitive null-homologous closed geodesics.
\end{theorem}

Since the results of~\cite{Sig} and~\cite{coles-sharp} are formulated in terms of invariant measures for the geodesic flow on~$U^*M$, we first explain how such measures give rise to geodesic currents and how weak$^*$ convergence of invariant measures translates into convergence of geodesic currents. 
We then deduce Theorems~\ref{theo:dense} and~\ref{theo:dense+}. 
This correspondence will also be used in Section~\ref{sec:generic}.


\subsection*{Correspondence between flow-invariant measures and geodesic currents}

Let $M$ be a closed surface endowed with a Finsler metric~$F$.
Let $m$ be a Radon measure on~$U^*M$ invariant under the geodesic flow~$\phi_t$.
We construct a current~$\mu$ as follows.
First, we symmetrize~$m$ by defining
\[
\bar{m} = \tfrac{1}{2} (m+A_* m)
\]
where $A:U^*M \to U^*M$ is the antipodal involution given by~\mbox{$A(\xi)=-\xi$}.
Since $F$ is reversible, we have $A \circ \phi_t = \phi_{-t} \circ A$.
It follows that $\bar{m}$ is invariant under both the geodesic flow and the antipodal involution.

Let $\Sigma$ be a surface transverse to the geodesic foliation of~$PT^*M$.
On this surface~$\Sigma$, the measure~$\bar{m}$ determines a measure~$\mu_\Sigma$ as follows.
Choose a lift~$\bar{\Sigma}$ of~$\Sigma$ in~$U^*M$.
For $\varepsilon >0$ sufficiently small, the map $\Phi:\bar{\Sigma} \times (-\varepsilon,\varepsilon) \to U^*M$, where $\Phi(\xi,t) = \phi_t(\xi)$, is a flow box.
The invariance of~$\bar m$ under the geodesic flow gives a unique product decomposition
\begin{equation} \label{eq:mudt}
\Phi^* \bar{m} = \bar{\mu}_{\bar{\Sigma}} \otimes dt
\end{equation}
where $\bar{\mu}_{\bar{\Sigma}}$ is a Radon measure on~$\bar{\Sigma}$.
The antipodal invariance of~$\bar{m}$ implies that, under the natural identification with~$\Sigma$, this measure is independent of the choice of the lift~$\bar{\Sigma}$.

We define the transverse measure on~$\Sigma$ by
\begin{equation} \label{eq:muu}
\mu_\Sigma = \tfrac{1}{2} \, \bar{\mu}_{\bar{\Sigma}}.
\end{equation}
The factor~$\frac{1}{2}$ accounts for the two lifts of~$\Sigma$ under the double cover $U^*M \to PT^*M$.
It is chosen so that the Liouville measure induces the Liouville geodesic current; see Example~\ref{ex:measure-current2}.
The measures~$\mu_\Sigma$ are invariant under holonomy and therefore define a geodesic current~$\mu$ on~$M$.

Conversely, a geodesic current~$\mu$ on~$M$ determines a Radon measure on~$U^*M$ invariant under both the geodesic flow and the antipodal involution. 
More precisely, in a flow box associated with a lift~$\bar{\Sigma}$ of~$\Sigma$, this measure is given by $2 \, \mu_\Sigma \otimes  dt$.

\medskip

We present two fundamental examples.

\begin{example} \label{ex:measure-current1}
Let~$\gamma$ be a closed geodesic on~$M$.
Its lift to~$U^*M$, obtained via the Legendre transform, is a periodic orbit of the geodesic flow.
Integration along this orbit defines an invariant Radon measure~$m_\gamma$ by
\[
\int_{U^*M} f \, dm_\gamma = \int_0^{\length(\gamma)} f(\phi_t(\xi_\gamma)) \, dt
\]
for every continuous function $f:U^*M \to \R$, where $\xi_\gamma$ is the momentum of~$\gamma$ at~$0$.
After symmetrization with respect to the antipodal involution, the previous construction gives $\frac{1}{2}$ times the geodesic current induced by~$\gamma$.
\end{example}

\begin{example} \label{ex:measure-current2}
The Liouville measure~$m_L$, induced by the volume form \mbox{$\alpha \wedge d\alpha$}, is invariant under the geodesic flow and the antipodal involution.
The geodesic current~$\mu_L$ induced by~$m_L$ through the previous construction is precisely the Liouville geodesic current~$\lambda$.
Indeed, in a flow box, the transverse measure induced by an invariant volume form is induced by the contraction of that volume form with the geodesic vector field~$Z$.
Since $\iota_Z(\alpha \wedge d\alpha) = d\alpha$, the transverse measure induced by~$m_L$ on~$\bar{\Sigma}$ is $| d\alpha|_{\bar{\Sigma}} |$.
It follows from~\eqref{eq:muu} that
\[
\mu_{L,\Sigma} =  \tfrac{1}{2} \, | d\alpha|_{\bar{\Sigma}} |
\]
which is exactly the definition of the Liouville geodesic current~$\lambda$.
\end{example}

We will also need the following convergence result.

\begin{lemma} \label{lem:measure-current}
Let $M$ be a closed surface endowed with a Finsler metric.
Let $m_n$ be a sequence of Radon measures on~$U^*M$ invariant under the geodesic flow converging to~$m$ in the weak$^*$ topology.
Denote by $\mu_n$ and~$\mu$ the geodesic currents associated to~$m_n$ and~$m$.
Then $\mu_n$ converges to~$\mu$ as geodesic currents.
\end{lemma}

\begin{proof}
%
Fix a surface~$\Sigma$ transverse to the geodesic foliation of~$PT^*M$, and choose a lift~$\bar\Sigma$ to~$U^*M$.
Let $h \in C_c(\Sigma)$ and, by a slight abuse of notation, denote its lift to~$\bar\Sigma$ also by~$h$.
Choose $\chi \in C_c((-\varepsilon,\varepsilon))$ such that $\chi \geq 0$ and $\int_{-\varepsilon}^\varepsilon \chi(t) \, dt =1$.
Define the compactly supported continuous function~$H$ on the flow box by $H(\Phi(\xi,t)) = h(\xi) \chi(t)$, and extend it by zero outside the flow box.
The product decomposition~\eqref{eq:mudt} and the formula~\eqref{eq:muu} give
\begin{equation} \label{eq:Hh}
\frac{1}{2} \int_{U^*M} H \, d\bar{m}  = \frac{1}{2} \int_{\bar{\Sigma} \times (-\varepsilon,\varepsilon)} h(\xi) \chi(t) \, d\bar{\mu}_{\bar{\Sigma}}(\xi) dt  = \int_\Sigma h \, d\mu_\Sigma.
\end{equation}
Similarly with~$m_n$.
Now, since $\bar{m}_n$ converges to~$\bar{m}$ in the weak$^*$ topology, it follows from~\eqref{eq:Hh} that
\[
\int_\Sigma h \, d\mu_{n,\Sigma} \to \int_\Sigma h \, d\mu_{\Sigma}.
\]
Thus, $\mu_{n,\Sigma}$ converges vaguely to~$\mu_{\Sigma}$ on every transverse surface~$\Sigma$, which is precisely the convergence of geodesic
currents.
\end{proof}

Theorem~\ref{theo:dense} follows from theorem~\cite[Theorem~1]{Sig}.
Indeed, after normalization, Sigmund's theorem provides positive multiples of periodic-orbit measures converging weakly to any invariant measure of the geodesic flow.
Applying Lemma~\ref{lem:measure-current} shows that the corresponding positive multiples of geodesic currents induced by closed geodesics converge to the associated geodesic current.

\medskip

Similarly, Theorem~\ref{theo:dense+} follows from~\cite[Theorem~1.4]{coles-sharp}.
Indeed, since $\iota_Z(\alpha\wedge d\alpha)=d\alpha$, the geodesic flow is null-homologous in the sense of~\cite{coles-sharp}.
Their theorem provides positive linear combinations of periodic-orbit measures supported on primitive null-homologous closed geodesics converging weakly to the Liouville measure.
Applying Lemma~\ref{lem:measure-current} (and Example~\ref{ex:measure-current2}) gives the desired convergence to the Liouville geodesic current.

\medskip


Applying Theorem~\ref{theo:dense+} in place of Sigmund's density theorem in the Croke--Dairbekov argument described in Section~\ref{sec:outline}, and using the intersection formulas for geodesic currents on Finsler surfaces (see Propositions~\ref{prop:bonahon_length} and~\ref{prop:bonahon_area}), we obtain the following extension of Theorem~\ref{theo:CD}.

\begin{theorem} \label{theo:extension}
Let $M$ be a closed surface with two Finsler metrics~$F_0$ and~$F$ whose geodesic flows are Anosov.
If $\MLS_{F_0}(\alpha) \leq \MLS_F(\alpha)$ for every unoriented free homotopy class~$\alpha$ represented by a null-homologous closed curve, then 
\[
\area(F_0) \leq \area(F).
\]
\end{theorem}

\section{Liouville geodesic current approximation for generic Finsler metrics} \label{sec:generic}

We now turn to the approximation of the Liouville geodesic current for Finsler metrics whose geodesic flow is not necessarily Anosov.
The main result of this section follows from a generic approximation theorem of Irie~\cite{irie} for Reeb flows (more precisely, a refined version of Liu and Liu~\cite{liuliu}).
Unlike the approximation results of the previous section, Irie's theorem is formulated in terms of $1$-dimensional currents on~$U^*M$.
We therefore explain how its conclusion gives rise to an approximation by geodesic currents.

\begin{theorem} \label{theo:irie+}
Let $M$ be a closed surface endowed with a generic reversible Finsler metric~$F$.
Then the Liouville geodesic current of~$F$ is a limit of positive linear combinations of geodesic currents induced by primitive closed geodesics.
\end{theorem}

Here, the genericity assumption means that the conclusion of the theorem holds for a residual subset of the space of reversible Finsler metrics endowed with the $C^\infty$-topology, that is, for a subset containing a countable intersection of open dense sets.

\medskip

We first recall some definitions concerning currents in the sense of geometric measure theory.

\begin{definition}
A \emph{$k$-dimensional current}~$C$ (or simply a \emph{$k$-current}) on an $n$-manifold~$N$ is a continuous linear form on the space~$\Omega_c^k(N)$ of smooth $k$-forms with compact support on~$N$.
For instance, an oriented smooth closed curve of~$N$ induces a $1$-current by integration along this curve.
The space of $k$-currents on~$N$ is endowed with the weak$^*$ topology.

\end{definition}

We will also need the notion of positive Reeb currents and Liouville \mbox{$1$-currents} in contact dynamics, adapted here to our setting.

\begin{definition}
Let $M$ be a closed surface endowed with a Finsler metric.
Every oriented closed geodesic~$\gamma$ on~$M$ lifts, via the Legendre transform, to a periodic orbit~$\hat{\gamma}$ of the geodesic flow on~$U^*M$.
Integration of $1$-forms along this orbit defines a $1$-current~$C_\gamma$ on~$U^*M$.
That is,
\begin{equation} \label{eq:reeb}
C_\gamma(\eta) = \int_{\hat{\gamma}} \eta
\end{equation}
for every smooth $1$-form~$\eta$ on~$U^*M$.
For simplicity, this $1$-current is also denoted by~$\gamma$.


A \emph{positive Reeb current} is a $1$-current~$C$ on~$U^*M$ which decomposes as a positive linear combination of $1$-currents induced by distinct closed geodesics~$\gamma_i$.
That is,
\[
C(\eta) = \sum_i a_i \int_{\hat{\gamma}_i} \eta
\]
for every smooth $1$-form~$\eta$ on~$U^*M$, where $a_i>0$.

The \emph{Liouville $1$-current} is the $1$-current~$C_L$ on~$U^*M$ induced by the \mbox{$2$-form}~$d\alpha$, where $\alpha$ is the natural contact form on~$U^*M$.
That is,
\[
C_L(\eta) = \int_{U^*M} \eta \wedge d\alpha
\]
for every smooth one-form~$\eta$ on~$U^*M$.
\end{definition}

\begin{remark}
For a Reeb current~$C_\gamma$, the integral formula~\eqref{eq:reeb} applies to continuous $1$-forms~$\eta$ on~$U^*M$.
In particular, $C_\gamma$ defines a continuous linear form on the space of continuous $1$-forms on~$U^*M$, with
\[
|C_\gamma(\eta)| \leq \length(\gamma) \, ||\eta||_\infty.
\]
The same holds for positive Reeb currents.
\end{remark}

The following approximation result is due to Irie~\cite{irie} for generic Finsler metrics, not necessarily reversible.
The argument has been adapted to generic reversible Finsler metrics by Liu and Liu~\cite{liuliu}; see in particular Theorem~1.1 and the comments following Theorem~1.2 in~\cite{liuliu}.

\begin{theorem} \label{theo:irie}
Let $M$ be a closed surface.
For a generic reversible Finsler metric~$F$ on~$M$, there exists a sequence of positive Reeb currents~$C_n$ on~$U^*M$ converging weakly to the Liouville $1$-current~$C_L$.
That is,
\[
\lim_{n \to \infty} C_n(\eta) = \int_{U^*M} \eta \wedge d\alpha
\]
for every smooth $1$-form~$\eta$ on~$U^*M$.
\end{theorem}

We can now proceed to the proof of Theorem~\ref{theo:irie+}.

\begin{proof}[Proof of Theorem~\ref{theo:irie+}]
By the Riesz representation theorem~\cite[Theorem~7.2]
{folland}, every positive Reeb current~$C$ determines a unique Radon measure~$m_C$ on~$U^*M$ characterized by
\[
\int_{U^*M} f \, dm_C = C(f \alpha)
\]
for every continuous function~$f$ on~$U^*M$.
More concretely, in the case of a Reeb current $C=C_\gamma$, the corresponding measure~$m_{C_\gamma}$ is characterized by
\[
\int_{U^*M} f \, dm_{C_\gamma} = \int_{\hat{\gamma}} f \alpha = \int_0^{\length(\gamma)} f(\phi_t(\xi_\gamma)) \, dt
\]
where $\xi_\gamma$ is the momentum of~$\gamma$ at~$0$.
Thus, $m_{C_\gamma}$ coincides with the invariant measure~$m_\gamma$ defined in Example~\ref{ex:measure-current1}.
More generally, if $C= \sum_i a_i \, \gamma_i$ is a positive Reeb current, then $m_C = \sum_i a_i \, m_{\gamma_i}$.
In particular, $m_C$ is invariant under the geodesic flow.

Given a reversible Finsler metric~$F$ in the generic set provided by Theorem~\ref{theo:irie}, there exists a sequence of positive Reeb currents~$C_n$ on~$U^*M$ converging weakly to the Liouville $1$-current~$C_L$.
Denote $m_n = m_{C_n}$.
For every smooth function~$f$ on~$U^*M$,
\[
\int_{U^*M} f\, dm_n = C_n(f \alpha) \xrightarrow[n \to \infty]{} C_L(f \alpha) =  \int_{U^*M} f \alpha \wedge d\alpha = \int_{U^*M} f \, dm_L
\]
where $m_L$ is the Liouville measure induced by~$\alpha \wedge d\alpha$.
Taking $f=1$, we see that the total masses of~$m_n$ converge to the total mass of~$m_L$, and are therefore uniformly bounded.
Since smooth functions are dense among continuous functions, it follows that the invariant measure~$m_n$ converges to~$m_L$ in the weak$^*$ topology.

By Lemma~\ref{lem:measure-current} (and Example~\ref{ex:measure-current2}), the geodesic current~$\mu_n$ associated with~$m_n$ converges to the Liouville geodesic current.
\end{proof}

\forget

\begin{proof}[Proof of Theorem~\ref{theo:irie+}]
Let $F$ be a reversible Finsler metrics in the generic set provided by Theorem~\ref{theo:irie}.
There exists a sequence of positive Reeb currents~$C_n$ on~$U^*M$ converging weakly to the Liouville $1$-current~$C_L$.

Let $C$ be a $1$-current on~$U^*M$.
Define
\[
\bar{C} = \tfrac{1}{2} ( C - A_* C).
\]
where $A:U^*M \to U^*M$ is the antipodal involution.

Since $F$ is reversible, we have
$A^* \alpha = - \alpha$ and $A \circ \phi_t = \phi_{-t} \circ A$.
Thus, if $\gamma$ is the positive Reeb current associated with an oriented closed geodesic, then $-A_* \gamma$ is the positive Reeb current associated with the same closed geodesic with the opposite orientation.
It follows that $\bar{C}_n$ is a positive Reeb current satisfying \mbox{$A_* \bar{C}_n = - \bar{C}_n$}.

The Liouville $1$-current has the same anti-invariance property.
Indeed, since $A^*d\alpha=-d\alpha$, we have
\[
A_* C_L(\eta) =  \int_{U^*M} A^* \eta \wedge d\alpha = - \int_{U^*M} A^*(\eta \wedge d\alpha) = - \int_{U^*M} \eta \wedge d\alpha = - C_L(\eta)
\]
for every smooth $1$-form~$\eta$.
Thus, $A_* C_L = - C_L$ and therefore $\bar{C}_L = C_L$.

Since pushforward by~$A$ is continuous for the weak$^*$ topology of currents, the positive Reeb currents~$\bar{C}_n$ also converge to~$C_L$.


We now associate with each positive Reeb current an invariant measure for the geodesic flow.
By the Riesz representation theorem~\cite[Corollary~7.18]{folland}, every positive Reeb current~$C$ determines a unique Radon measure~$m_C$ on~$U^*M$ characterized by
\[
\int_{U^*M} f \, dm_C = C(f \alpha)
\]
for every continuous function~$f$ on~$U^*M$.

If $C$ is induced by a closed geodesic~$\gamma$, then
\[
\int_{U^*M} f \, dm_\gamma = \int_\gamma f \alpha = \int_0^{\length(\gamma)} f(\phi_t(\gamma'(0))) \, dt.
\]
Thus, $m_\gamma$ is the arclength measure along the periodic Reeb orbit corresponding to~$\gamma$.
More generally, if $C= \sum_j a_j \, \gamma_j$, then $m_C = \sum_j a_j \, m_{\gamma_j}$.
In particular, $m_C$ is invariant under the geodesic flow.

If $C$ is anti-invariant under the antipodal involution, \ie, $A_* C = -C$, then $m_C$ is invariant under the antipodal involution.
Indeed, since $A^* \alpha = - \alpha$,
\[
C((f \circ A) \alpha) = - C(A^*(f \alpha)) = - (A_* C)(f \alpha) = C(f \alpha).
\]
That is,
\[
\int_{U^*M} f \circ A \, dm_C =  \int_{U^*M} f \, dm_C.
\]
Hence,
\[
A_* m_C = m_C.
\]

Let $m_n = m_{\bar{C}_n}$.
Recall that $\bar{C}_n \to C_L$.
For every smooth function~$f$ on~$U^*M$,
\[
\bar{C}_n(f \alpha) = \int_{U^*M} f\, dm_n \xrightarrow[n \to \infty]{} \int_{U^*M} f \alpha \wedge d\alpha.
\]
Moreover, the masses of~$m_n$ are uniformly bounded since $m_n(U^*M)=\bar{C}_n(\alpha) \to m_L(U^*M)$.
Since smooth functions are dense in~$C(U^*M)$, it follows that $m_n$ converges to the Liouville volume measure $m_L = \alpha \wedge d\alpha$ in the weak$^*$ topology.
That is,
\[
m_n \to m_L.
\]

Each~$m_n$ is invariant under both the geodesic flow and the antipodal involution.
Let~$\mu_n$ be the geodesic current associated with~$m_n$ by the construction of the previous section.
Since~$\bar C_n$ is a positive Reeb current, $\mu_n$ is a positive linear combination of geodesic currents induced by closed geodesics.
By Example~\ref{ex:measure-current2}, the measure~$m_L$ corresponds to the Liouville geodesic current~$\lambda$.
Therefore, Lemma~\ref{lem:measure-current} gives that $\mu_n$ converges to~$\lambda$ as geodesic currents.
\end{proof}

\forgotten

\section{Metric discretization by wall systems} \label{sec:discretization}

The purpose of this section is to show how the approximation of the Liouville geodesic current by geodesic wall systems leads to a
discretization of both the marked length spectrum and the area of a Finsler metric.

\medskip

It will be convenient to introduce the following notations.

\begin{definition}
Let $a,b,c \geq 0$ and $\varepsilon >0$.
We write $a \simeq_\varepsilon b$ if 
\begin{equation} \label{eq:mult}
(1+\varepsilon)^{-1} b \leq a \leq (1+\varepsilon) b.
\end{equation}
Similarly, we write $a \gtrsim_\varepsilon b$ or $a \lesssim_\varepsilon b$ if only the first or the second inequality holds.
We also write $a \simeq b \pm c$ if
\begin{equation} \label{eq:add}
b-c \leq a \leq b+c.
\end{equation}
\end{definition}

We can now state the main discretization result of this section.
It applies whenever the Liouville geodesic current can be approximated by geodesic wall systems.

\begin{theorem} \label{theo:MLSA}
Let $M$ be a closed surface endowed with a Finsler metric~$F$.
Suppose that the Liouville geodesic current~$\lambda$ of~$F$ can be approximated, in the space of geodesic currents, by geodesic wall systems (as in Theorem~\ref{theo:irie+}).
Then, for every~$\varepsilon >0$, there exists a geodesic wall system~$\W$ such that 
\begin{align*}
\MLS_{\W} & \simeq_\varepsilon \MLS_F \\
\area(\W) & \simeq_\varepsilon \frac{\pi}{4} \, \area(F).
\end{align*}
\end{theorem}

We will need the following uniform convergence result.

\begin{proposition} \label{prop:i_WL}
Under the assumptions of Theorem~\ref{theo:MLSA}, for every~\mbox{$\varepsilon >0$}, there exists a geodesic wall system~$\W$ such that for every geodesic current~$\mu$,
\[
i(\mu,\W) \simeq_\varepsilon i(\mu,\lambda).
\]
\end{proposition}

\begin{proof}
By assumption, there exists a sequence of geodesic wall systems~$\W_n$ converging to~$\lambda$ as geodesic currents.


Since~$\lambda$ is filling and without atoms, the space $\mathcal K = \{ \nu \in \Curr(M) \mid i(\nu,\lambda) = 1 \}$ is compact by Proposition~\ref{prop:compact}.

Fix $\varepsilon \in (0,1)$.
We claim that, for all sufficiently large~$n$,
\begin{equation} \label{eq:uniform_WL}
i(\mu,\W_n) \simeq_\varepsilon i(\mu,\lambda)
\end{equation}
for every geodesic current~$\mu$.
By homogeneity of the intersection form, it is enough to prove that for all sufficiently large~$n$,
\[ 
(1+\varepsilon)^{-1} \leq i(\mathcal W_n,\nu) \leq 1+\varepsilon.
\] 
for every $\nu \in \mathcal{K}$.

Suppose otherwise.
Then there exist a subsequence~$(\W_{n_k})$ and geodesic currents~$\nu_k \in \mathcal{K}$ such that
\begin{equation} \label{eq:contradiction0}
i(\nu_k,\W_{n_k}) \notin [(1+\varepsilon)^{-1},1+\varepsilon].
\end{equation}
for every~$k$.
Since $\mathcal{K}$ is compact and metrizable, after passing to a subsequence, we may assume that $\nu_k$ converges to~$\nu \in \mathcal{K}$.
Recall also that $\W_{n_k}$ converges to~$\lambda$.
By continuity of the intersection form at~$(\nu,\lambda)$, see Proposition~\ref{prop:continuity}, we obtain 
\[
i(\nu_k,\W_{n_k}) \to i(\nu,\lambda)=1,
\]
which contradicts~\eqref{eq:contradiction0}.
Thus~\eqref{eq:uniform_WL} holds for all sufficiently large~$n$.
\end{proof}

We can now prove Theorem~\ref{theo:MLSA}.

\begin{proof}[Proof of Theorem~\ref{theo:MLSA}]
Fix~$\varepsilon \in (0,1)$.
Choose $\delta>0$ such that $(1+\delta)^2 \leq 1+\varepsilon$.
Note that $\delta \leq \varepsilon$.
By Proposition~\ref{prop:i_WL}, there exists a geodesic wall system such that, for every geodesic current~$\mu$,
\begin{equation} \label{eq:i_WL}
i(\mu,\W) \simeq_\delta i(\mu,\lambda).
\end{equation}

Let $\alpha$ be an unoriented free homotopy class.
By Proposition~\ref{prop:MLSi}.\eqref{i3}, there exists a closed geodesic~$\gamma$ representing~$\alpha$ such that
\[
\MLS_{\W}(\alpha) = i(\gamma,\W).
\]
By the relation~\eqref{eq:i_WL} applied to~$\mu=\gamma$ and the length formula of Proposition~\ref{prop:bonahon_length}, we obtain
\[
\MLS_{\W}(\alpha) = i(\gamma,\W) \simeq_\delta i(\gamma,\lambda)  = \length_F(\gamma) \geq \MLS_F(\alpha).
\]

For the reverse inequality, let $\bar{\gamma}$ be a closed geodesic of minimal length representing~$\alpha$.
By Proposition~\ref{prop:MLSi}.\eqref{i22}, the relation~\eqref{eq:i_WL} applied to~$\mu=\bar{\gamma}$ and the length formula of Proposition~\ref{prop:bonahon_length}, we obtain
\[
\MLS_{\W}(\alpha) \leq i(\bar{\gamma},\W) \simeq_\delta i(\bar{\gamma},\lambda) = \length_F(\bar{\gamma}) = \MLS_F(\alpha).
\]

Therefore,
\[
\MLS_{\W}(\alpha) \simeq_\varepsilon \MLS_F(\alpha).
\]

It remains to compare the areas.
By the relation~\eqref{eq:i_WL} applied first to~$\mu=\W$ then to $\mu=\lambda$, and the symmetry of the intersection form, we obtain
\[
i(\W,\W) \simeq_\delta i(\W,\lambda) = i(\lambda,\W) \simeq_\delta i(\lambda,\lambda).
\]
By the choice of~$\delta$, it follows that
\[
i(\W,\W) \simeq_\varepsilon i(\lambda,\lambda).
\]
By the discrete area formula of Proposition~\ref{prop:MLSi}.\eqref{i2} and the area formula of Proposition~\ref{prop:bonahon_area}, we obtain
\[
\area(\W) = \frac{1}{2} \, i(\W,\W) \simeq_{\varepsilon} \frac{1}{2} \, i(\lambda,\lambda) = \frac{\pi}{4} \, \area(F).
\]
This proves the result.
\end{proof}

\forget

\begin{proof}[Proof of Theorem~\ref{theo:MLSA}]
Fix~$\varepsilon \in (0,1)$.
By Proposition~\ref{prop:i_WL}, there exists a geodesic wall system such that for every geodesic current~$\mu$,
\begin{equation} \label{eq:i_WL}
i(\mu,\W) \simeq_\varepsilon i(\mu,\lambda).
\end{equation}
Applying this relation first to~$\mu=\W$ then to $\mu=\lambda$, and using the symmetry of the intersection form, we obtain
\[
i(\W,\W) \simeq_\varepsilon i(\W,\lambda) = i(\lambda,\W) \simeq_\varepsilon i(\lambda,\lambda).
\]
By the discrete area formula of Proposition~\ref{prop:MLSi}.\eqref{i2} and the area formula of Proposition~\ref{prop:bonahon_area}, this yields
\[
\area(\W) \simeq_{3 \varepsilon} \frac{\pi}{4} \, \area(F).
\]

Write~$\W=(\sigma_i,w_i)_{1 \leq i \leq n}$.
We want to show that 
\[
\MLS_{\W}(\alpha) \simeq_\varepsilon \MLS_F(\alpha)
\]
for every unoriented free homotopy class~$\alpha$.

Let $a=[\alpha] \in H_1(M;\Z)$ be the homology class associated to~$\alpha$, up to sign.
Let also $s_i=[\sigma_i] \in H_1(M;\Z)$. 
We will consider two cases.

\medskip

\textit{Case 1.}
Suppose first that $a$ does not lie in any line~$\R s_i$ in~$H_1(M;\R)$, where the lines~$\R s_i$ may be reduced to~$\{0\}$.
In particular, every closed geodesic in~$\alpha$ is transverse to~$\W$.

Let $\gamma$ be a closed geodesic of minimal length in~$\alpha$.
By the length formula of Proposition~\ref{prop:bonahon_length} and the relation~\eqref{eq:i_WL} applied to~$\mu=\gamma$, we have
\[
\MLS_F(\alpha) = \length_F(\gamma) = i(\gamma,\lambda) \simeq_\varepsilon i(\gamma,\W).
\]
Now, since $\gamma$ is transverse to~$\W$, the discrete length formula of Proposition~\ref{prop:MLSi}.\eqref{i1} gives
\[
i(\gamma,\W) = \length_\W(\gamma) \geq \MLS_{\W}(\alpha).
\]
Hence, $\MLS_F(\alpha) \gtrsim_\varepsilon \MLS_{\W}(\alpha)$.

Let $c$ be a closed curve of minimal discrete length in~$\alpha$.
Then
\begin{equation} \label{eq:alpha_W}
\MLS_{\W}(\alpha) = \length_\W(c) = \sum_{i=1}^n w_i \, \#(c \cap \sigma_i).
\end{equation}
Apply the disk-flow~\cite{HS} simultaneously to the curves~$c$ and~$\sigma_i$.
(Technically, the disk-flow has been defined only for Riemannian metrics, but the definition carries over to reversible Finsler metrics.)
Since the curves~$\sigma_i$ are geodesic, they remain unchanged under the disk-flow.
The curve~$c$ sub-converges to a closed geodesic~$\bar{c}$ in the unoriented free homotopy class~$\alpha$ under the disk-flow.
By assumption, the geodesic~$\bar{c}$ is transverse to~$\sigma_i$.
Recall also that the disk-flow does not increase the number of intersections; see~\cite{HS}.
Thus,
\[
\#(c \cap \sigma_i) \geq \#(\bar{c} \cap \sigma_i) = i(\bar{c},\sigma_i)
\]
where the last equality comes from~\eqref{eq:itt}.
Combined with~\eqref{eq:alpha_W}, this yields
\[
\MLS_{\W}(\alpha) \geq i(\bar{c},\W).
\]
Applying~\eqref{eq:i_WL} to~$\mu = \bar{c}$, and using the length formula of Proposition~\ref{prop:bonahon_length}, we obtain
\[
i(\bar{c},\W) \simeq_\varepsilon i(\bar{c},\lambda) = \length_F(\bar{c}) \geq \MLS_F(\alpha).
\]
Hence, $\MLS_{\W}(\alpha) \gtrsim_\varepsilon \MLS_F(\alpha)$.

Together with the reverse inequality obtained above, this gives
\[
\MLS_{\W}(\alpha) \simeq_\varepsilon \MLS_F(\alpha).
\]

\textit{Case 2.}
Assume now that $a$ lies in some line~$\R s_j$.
Choose $b \in H_1(M;\Z)$ which does not lie in any line~$\R s_i$.
The line $b + ta$ with $t \in \R$, possibly reduced to~$b$, does not pass through the origin of~$H_1(M;\R)$ and intersects each line~$\R s_i$ in at most one point at~$t=t_i$.
Thus, for every integer~$k$ sufficiently large, $b+ka$ does not lie in any line~$\R s_i$.

Let $\gamma$ be a closed geodesic of minimal length in~$\alpha$.
Let $\gamma_0$ be a closed curve representing~$b$ with the same basepoint as~$\gamma$.
Denote by~$\alpha_k$ the unoriented free homotopy class induced by the concatenation~$\gamma_0 \gamma^k$.
By construction, the homology class~$[\alpha_k] = b + ka$, defined up to sign, does not lie in any line~$\R s_i$.
Thus, by Case~1, we have,
\begin{equation} \label{eq:MLS_WF}
\MLS_{\W}(\alpha_k) \simeq_\varepsilon \MLS_F(\alpha_k).
\end{equation}

To derive a similar relation for~$\alpha$, we need the following elementary asymptotic property of the marked length spectrum.

\begin{lemma} \label{lem:sim}
If $d$ is a Finsler metric or the pseudo-distance induced by a wall system, then
\[
\MLS_d(\alpha_k) \sim k \, \MLS_d(\alpha)
\]
as $k \to \infty$.
\end{lemma}

\begin{proof}
Identify~$\gamma$ and~$\gamma_0$ with their homotopy classes in~$\pi_1(M)$.
By definition,
\begin{align*}
\MLS_d(\alpha) & = \inf_{x \in \bar{M}} d(x,\gamma . x) \\
\MLS_d(\alpha_k) & = \inf_{x \in \bar{M}} d(x, \gamma_0 \gamma^k . x) 
\end{align*}
where $d$ also denotes the pseudo-distance on the universal cover~$\bar{M}$ of~$M$.
Now, by the triangle inequality,
\begin{align*}
d(x,\gamma_0 \gamma^k.x) & \simeq d(\gamma_0.x, \gamma_0 \gamma^k.x) \pm d(x,\gamma_0.x) \\
 & \simeq d(x, \gamma^k .x) \pm d(x,\gamma_0.x)
\end{align*}
It follows that 
\begin{align*}
\MLS_d(\alpha_k) & \simeq \MLS_d([\gamma^k]) \pm C(\gamma_0 \\
 & \simeq k \, \MLS_d(\alpha) \pm C(\gamma_0)
 \end{align*}
 for some constant~$C(\gamma_0) >0$ independent of~$k$,  where the relation $\MLS_d([\gamma^k]) = k \, \MLS_d([\gamma])$ comes from Proposition~\ref{prop:classical}.
 Thus,
 \[
\MLS_d(\alpha_k) \sim k \, \MLS_d(\alpha)
\]
when $k \to \infty$, as claimed.
\end{proof}

We can now complete the proof of the second case.
Dividing~\eqref{eq:MLS_WF} by~$k$, applying Lemma~\ref{lem:sim} to $d=d_\W$ and $d=d_F$, and letting $k \to \infty$, we obtain
\[
\MLS_{\W}(\alpha) \simeq_\varepsilon \MLS_F(\alpha)
\]
in the second case too.
\end{proof}

\forgotten

We present two consequences of Theorem~\ref{theo:MLSA}.

\begin{corollary} \label{coro:W0}
Let $M$ be a closed surface endowed with a Finsler metric~$F$ whose geodesic flow is Anosov.
Then, for every~$\varepsilon >0$, there exists a geodesic wall system~$\W$ such that
\begin{align*}
\MLS_{\W} & \simeq_\varepsilon \MLS_F \\
\area(\W) & \simeq_\varepsilon \frac{\pi}{4} \, \area(F).
\end{align*}
\end{corollary}

\begin{proof}
Since the geodesic flow of~$F$ is Anosov, Theorem~\ref{theo:dense} implies that its Liouville geodesic current~$\lambda$ can be approximated by geodesic wall systems.
The conclusion then follows from Theorem~\ref{theo:MLSA}.
\end{proof}

The discrete approximations of the marked length spectrum and area in the following corollary were previously obtained by the first author using different techniques~\cite{Cos}.
The version below provides the additional conclusion that the wall system can be chosen almost geodesic.

\begin{corollary} \label{coro:W}
Let $M$ be a closed surface endowed with a Finsler metric~$F$.
Then, for every~$\varepsilon >0$, there exists an almost geodesic wall system~$\W$ such that
\begin{align*}
\MLS_{\W} & \simeq_\varepsilon \MLS_F \\
\area(\W) & \simeq_\varepsilon \frac{\pi}{4} \, \area(F).
\end{align*}
\end{corollary}

\begin{proof}
After a small perturbation of~$F$, which changes its length and area functionals by arbitrarily small multiplicative factors, we may assume that the Liouville geodesic current of the perturbed metric can be approximated by geodesic wall systems; see Theorem~\ref{theo:irie+}.
The conclusion then follows from Theorem~\ref{theo:MLSA}, after choosing the perturbation sufficiently small.
The resulting wall system is geodesic with respect to the perturbed Finsler metric, and hence almost geodesic with respect to the original metric~$F$.
\end{proof}

\section{Proof of the main theorem} \label{sec:proof}

In this section, we prove the main theorem, namely Theorem~\ref{theo:main}, using its discrete counterpart, Theorem~\ref{theo:discrete}, together with the approximation of Finsler metrics by wall systems established in the previous section.
We also establish two mixed comparison theorems between Finsler metrics and wall systems.

\begin{theorem}  \label{theo:mainn}
Let $M$ be a closed surface endowed with a Finsler metric~$F_0$ whose geodesic flow is Anosov.
Let $F$ be another Finsler metric on~$M$.
If $\MLS_{F_0} \leq \MLS_F$ then 
\[
\area(F_0) \leq \area(F).
\]
\end{theorem}

\begin{proof}
Let $\varepsilon>0$.
By Theorem~\ref{theo:irie+}, there exist reversible Finsler metrics~$F'_0$ and~$F'$, arbitrarily close to~$F_0$ and~$F$, whose Liouville geodesic currents can be approximated by $F'_0$-geodesic and $F'$-geodesic wall systems~$\W_0^\varepsilon$ and~$\W^\varepsilon$.
Choosing the perturbations and the approximating wall systems sufficiently close, Theorem~\ref{theo:MLSA} gives
\[
\begin{aligned}
\MLS_{\W_0^\varepsilon} & \simeq_\varepsilon \MLS_{F'_0} \simeq_\varepsilon \MLS_{F_0} \\
\area(\W_0^\varepsilon) & \simeq_\varepsilon \frac{\pi}{4} \,\area(F'_0) \simeq_\varepsilon \frac{\pi}{4} \, \area(F_0).
\end{aligned}
\]
and
\[
\begin{aligned}
\MLS_{\W^\varepsilon} & \simeq_\varepsilon \MLS_{F'} \simeq_\varepsilon \MLS_{F} \\
\area(\W^\varepsilon) & \simeq_\varepsilon \frac{\pi}{4} \,\area(F') \simeq_\varepsilon \frac{\pi}{4} \, \area(F).
\end{aligned}
\]

By assumption, we have
\[
(1+\varepsilon)^{-2} \, \MLS_{\W_0^\varepsilon} \leq \MLS_{F_0} \leq \MLS_F \leq (1+\varepsilon)^2 \, \MLS_{\W^\varepsilon}.
\]
That is,
\[
\MLS_{(1+\varepsilon)^{-2} \W_0^\varepsilon} \leq \MLS_{(1+\varepsilon)^2 \W^\varepsilon}.
\]

By the structural stability of Anosov flows, see Theorem~\ref{theo:Anosov}.\eqref{Anosov4}, we may choose $F'_0$ sufficiently close to~$F_0$ so that its geodesic flow is still Anosov.
It follows that the $F'_0$-geodesic wall system~$\W_0^\varepsilon$ is minimal, and so is~$(1+\varepsilon)^{-2} \W_0^\varepsilon$.

We may therefore apply Theorem~\ref{theo:discrete} to~$(1+\varepsilon)^{-2} \W_0^\varepsilon$ and~$(1+\varepsilon)^2 \W^\varepsilon$.
Since the area of a wall system is homogeneous of degree two in its weights, we obtain
\[
(1+\varepsilon)^{-4} \, \area(\W_0^\varepsilon) \leq (1+\varepsilon)^4 \, \area(\W^\varepsilon).
\]
Letting~$\varepsilon \to 0$, we derive that
\[
\area(F_0) \leq \area(F). \qedhere
\]
\end{proof}

\begin{remark} \label{rem:alternative}
There is a more direct way to construct the wall systems on the $F_0$-side.
Instead of perturbing~$F_0$ to a generic metric~$F'_0$ and then using structural stability to ensure that the geodesic flow of~$F'_0$ remains Anosov, one may apply Sigmund's density theorem directly to~$F_0$.
Through Corollary~\ref{coro:W0}, this immediately yields approximating minimal $F_0$-geodesic wall systems.
This is the approach used in~\cite{CD}.
Our approach provides a unified treatment of the two metrics~$F_0$ and~$F$.
The Anosov hypothesis is then used only to guarantee, by structural stability, the minimality of the wall system for the perturbed metric~$F'_0$.
\end{remark}

\forget

\begin{proof}
Let $\varepsilon>0$.
By Corollary~\ref{coro:W0}, since the geodesic flow of~$F_0$ is Anosov, there exists a geodesic wall system~$\W_0^\varepsilon$ such that
\begin{align*}
\MLS_{\W_0^\varepsilon} & \simeq_\varepsilon \MLS_{F_0} \\
\area(\W_0^\varepsilon) & \simeq_\varepsilon \frac{\pi}{4} \, \area(F_0).
\end{align*}
By Theorem~\ref{theo:Anosov}.\eqref{Anosov1}, the closed $F_0$-geodesics are length-minimizing in their unoriented free homotopy classes.
It follows that the geodesic wall system~$\W_0^\varepsilon$ is minimal.

Similarly, by Corollary~\ref{coro:W}, there exists a wall system~$\W^\varepsilon$ such that
\begin{align*}
\MLS_{\W^\varepsilon} & \simeq_\varepsilon \MLS_F \\
\area(\W^\varepsilon) & \simeq_\varepsilon \frac{\pi}{4} \, \area(F).
\end{align*}

By assumption, we have
\[
(1+\varepsilon)^{-1} \, \MLS_{\W_0^\varepsilon} \leq \MLS_{F_0} \leq \MLS_F \leq (1+\varepsilon) \, \MLS_{\W^\varepsilon}.
\]
That is,
\[
\MLS_{(1+\varepsilon)^{-1} \W_0^\varepsilon} \leq \MLS_{(1+\varepsilon) \W^\varepsilon}.
\]
Applying Theorem~\ref{theo:discrete} to~$(1+\varepsilon)^{-1} \W_0^\varepsilon$ and~$(1+\varepsilon) \W^\varepsilon$, we obtain
\[
(1+\varepsilon)^{-2} \, \area(\W_0^\varepsilon) \leq (1+\varepsilon)^2 \, \area(\W^\varepsilon).
\]
Letting~$\varepsilon \to 0$, we derive that
\[
\area(F_0) \leq \area(F).
\]
\end{proof}

\begin{remark} \label{rem:alternative}
We do not need to apply Sigmund's density theorem through Corollary~\ref{coro:W0} to prove Theorem~\ref{theo:mainn}.
Instead, we can rely directly on the approximation of Liouville geodesic currents for generic Finsler metrics given by Theorem~\ref{theo:irie+}.
Indeed, by Theorem~\ref{theo:irie+}, there exists a reversible Finsler metric~$F'_0$, arbitrarily close to~$F_0$, whose Liouville geodesic current can be approximated by $F'_0$-geodesic wall systems~$\W_0^\varepsilon$.
By the structural stability of Anosov flows, see Theorem~\ref{theo:Anosov}.\eqref{Anosov4}, we may choose $F'_0$ sufficiently close to~$F_0$ so that its geodesic flow is still Anosov.
It follows that the $F'_0$-geodesic wall system~$\W_0^\varepsilon$ is minimal.
By Theorem~\ref{theo:MLSA}, we have
\[
\begin{aligned}
\MLS_{\W_0^\varepsilon} & \simeq_\varepsilon \MLS_{F'_0} \simeq_\varepsilon \MLS_{F_0} \\
\area(\W_0^\varepsilon) & \simeq_\varepsilon \frac{\pi}{4} \,\area(F'_0) \simeq_\varepsilon \frac{\pi}{4} \, \area(F_0).
\end{aligned}
\]
We can then conclude as in the proof of Theorem~\ref{theo:mainn}: approximate the metric~$F$ by a wall system~$\W^\varepsilon$ using Corollary~\ref{coro:W}, and apply Theorem~\ref{theo:discrete}.
\end{remark}

\forgotten

In the following version, the initial length structure comes from a Finsler metric with Anosov geodesic flow, while the competing length structure is induced by a wall system.

\begin{theorem} \label{theo:F0W}
Let $M$ be an orientable closed surface endowed with a Finsler metric~$F_0$ whose geodesic flow is Anosov.
Let $\W$ be a wall system on~$M$.
If $\MLS_{F_0} \leq \MLS_{\W}$, then 
\[
\area(F_0) \leq \frac{4}{\pi} \, \area(\W).
\]
\end{theorem}

\begin{proof}
As in the proof of Theorem~\ref{theo:mainn}, for every $\varepsilon>0$, we approximate the marked length spectrum and the area of~$F_0$ by those of a minimal geodesic wall system~$\W_0^\varepsilon$.
By assumption, 
\[
(1+\varepsilon)^{-1} \, \MLS_{\W_0^\varepsilon} \leq \MLS_{F_0} \leq \MLS_{\W}.
\]
Applying Theorem~\ref{theo:discrete} to~$(1+\varepsilon)^{-1} \W_0^\varepsilon$ and~$\W$, we obtain
\[
\area(F_0) \leq \frac{4}{\pi} \, \area(\W).
\]
after letting~$\varepsilon \to 0$.
\end{proof}

\forget

\begin{proof}
Since the geodesic flow of~$F_0$ is Anosov, the metric~$F_0$ has no conjugate points; see Theorem~\ref{theo:Anosov}.\eqref{Anosov1}.
Let~$\overline{\W}$ be the minimal geodesic wall system obtained from~$\W$ by Theorem~\ref{theo:GS}, also viewed as a geodesic current.
By Theorem~\ref{theo:GS}, we have 
\[
\MLS_{F_0} \leq \MLS_{\W} = \MLS_{\overline{\W}}.
\]
By Proposition~\ref{prop:bonahon_length}, applied to~$F_0$, and Proposition~\ref{prop:MLSi}.\eqref{i4}, applied to~$\overline{\W}$, it follows that
\[
i(\gamma,\lambda) \leq i(\gamma,\overline{\W})
\]
for every closed $F_0$-geodesic~$\gamma$ (that is length-minimizing in its unoriented free homotopy class), where $\lambda$ denotes the Liouville geodesic current of~$F_0$.

By the density of the positive multiples of the geodesic currents induced by closed geodesics (see Theorem~\ref{theo:dense}) and the continuity of the intersection form, we deduce that
\[
i(\mu,\lambda) \leq i(\mu,\overline{\W})
\]
for every geodesic current~$\mu$ of~$M$.
Applying this inequality first to $\mu=\lambda$ then to $\mu=\overline{\W}$, and using the symmetry of the intersection form, we obtain
\[
i(\lambda,\lambda) \leq i(\lambda,\overline{\W}) = i(\overline{\W},\lambda) \leq i(\overline{\W},\overline{\W}).
\]
By the discrete area formula of Proposition~\ref{prop:MLSi}.\eqref{i2} and the area inequality of Theorem~\ref{theo:GS},
\[
i(\overline{\W},\overline{\W}) = 2 \area(\overline{\W}) \leq 2 \area(\W).
\]
Combining these inequalities yields
\[
i(\lambda,\lambda) \leq 2 \area(\W).
\]
Now, Proposition~\ref{prop:bonahon_area} gives
\[
i(\lambda,\lambda) = \frac{\pi}{2} \, \area(F_0).
\]
Hence the desired inequality
\[
\area(F_0) \leq \frac{4}{\pi} \, \area(\W).
\]
\end{proof}

\forgotten

In the last version, the initial length structure comes from a minimal geodesic wall system, while the competing length structure is Finsler.

\begin{theorem}
Let $M$ be an orientable closed surface.
Let $\W_0$ be a minimal geodesic wall system on~$M$ with respect to some Finsler metric.
Let $F$ be another Finsler metric on~$M$.
If $\MLS_{\W_0} \leq \MLS_F$ then 
\[
\area(\W_0) \leq \frac{\pi}{4} \, \area(F).
\]
\end{theorem}

\begin{proof}
For every $\varepsilon>0$, we approximate the marked length spectrum and the area of~$F$ by those of a wall system~$\W^\varepsilon$ as in Corollary~\ref{coro:W}.
By assumption, 
\[
\MLS_{\W_0} \leq \MLS_F \leq (1+\varepsilon) \, \MLS_{\W^\varepsilon}.
\]
Applying Theorem~\ref{theo:discrete} to~$\W_0$ and~$(1+\varepsilon) \W^\varepsilon$, we obtain
\[
\area(\W_0) \leq \frac{\pi}{4} \, \area(F)
\]
after letting~$\varepsilon \to 0$.
\end{proof}

\forget

\section{Proof of the main theorem} \label{sec:proof}

We can prove the main theorem.

\begin{proof}[Proof of Theorem~\ref{theo:main}]
Let $M$ be an orientable closed surface endowed with a Finsler metric~$F_0$ whose geodesic flow is Anosov.
Let $F$ be another Finsler metric on~$M$ with
\[
\MLS_{F_0} \leq \MLS_F.
\]
Both this inequality and the area inequality $\area(F_0) \leq \area(F)$ we want to prove are invariant under a simultaneous homothety of the metrics~$F_0$ and~$F$ (\ie, when both metrics are multiplied/divided by the same positive constant).

Fix $\varepsilon >0$.
By Theorem~\ref{theo:Cos}, there exist a self-transverse wall system~$\W$ on~$M$ and a constant $\lambda>0$ such that
\begin{alignat*}{2}
(1+\varepsilon)^{-1} \, \MLS_F & \leq \lambda \, \MLS_{\W} && \leq (1+\varepsilon) \, \MLS_F \\
(1+\varepsilon)^{-1} \frac{\pi}{4} \, \area(F) & \leq \, \lambda^2 \, \area(\W) && \leq (1+\varepsilon) \frac{\pi}{4} \, \area(F).
\end{alignat*}
Dividing both $F_0$ and~$F$ by~$(1+\varepsilon) \lambda$, we can assume that
\[
\MLS_F \leq \MLS_{\W}
\]
and
\[
\area(\W) \leq (1+\varepsilon)^3 \, \frac{\pi}{4} \, \area(F).
\]
Now, since $\MLS_{F_0} \leq \MLS_F \leq \MLS_{\W}$, it follows from Theorem~\ref{theo:F0W} that 
\[
\area(F_0) \leq \frac{4}{\pi} \, \area(\W) \leq (1+\varepsilon)^3 \, \area(F).
\]
Letting $\varepsilon \to 0$, we derive that 
\[
\area(F_0) \leq \area(F).
\]
\end{proof}

\forgotten

\section{Marked length spectrum of the two-torus} \label{sec:torus}

We prove a version of the main theorem for the two-torus, namely Theorem~\ref{theo:torus}.
The main tool is the stable norm, which allows us to reduce the comparison of the marked length spectra and areas of Finsler metrics on the two-torus to an area comparison theorem of Burago and Ivanov~\cite{BI}.

\begin{definition} 
Let $\T^2$ be the two-torus with a Finsler metric~$F$.
The \emph{stable norm} of~$F$ is the norm~$||.||$ on~$H_1(\T^2;\R)$ given by
\[
|| h || = \lim_{n \to \infty} \frac{1}{n} \, \MLS_F([nh])
\]
for every $h \in H_1(\T^2;\Z)$, where $[n h]$ represents the unoriented free homotopy class induced by the homology class $nh \in H_1(\T^2;\Z) \simeq \pi_1(\T^2)$; see~\cite[\S4]{gro99}.
By Proposition~\ref{prop:classical}, the stable norm on~$\T^2$ satisfies $|| h || = \MLS(h)$.

The stable norm need not be smooth.
However, it defines a translation-invariant (possibly nonsmooth) Finsler metric on~$H_1(\T^2;\R)$, and thus, on the \emph{Jacobi torus} $J=H_1(\T^2;\R)/H_1(\T^2;\Z) \simeq \T^2$.
The resulting flat Finsler metric on~$J$ is also denoted by~$||.||$.
%
\end{definition}

We can proceed with the proof of Theorem~\ref{theo:torus}, which is an immediate consequence of~\cite[\S5]{BI}.

\begin{proof}[Proof of Theorem~\ref{theo:torus}]
Let $F_0$ be a Finsler metric on~$\T^2$ without conjugate points.
Let $F$ be another Finsler metric on~$\T^2$ with $\MLS_{F_0} \leq \MLS_F$.
This inequality between the marked length spectra of~$F_0$ and~$F$ implies the inequality $||.||_0 \leq ||.||$ between the stable norms of~$F_0$ and~$F$.
Therefore, 
\[
\area(\T^2,||.||_0) \leq \area(\T^2,||.||).
\]
Now, Burago and Ivanov~\cite[\S5]{BI} showed that 
\[
\area(\T^2,||.||) \leq \area(\T^2,F)
\]
with equality if $F$ has no conjugate points.
Hence,
\[
\area(\T^2,F_0) = \area(\T^2,||.||_0) \leq \area(\T^2,||.||) \leq \area(\T^2,F). \qedhere
\]
\end{proof}

\section{Finsler deformations preserving the marked length spectrum} \label{sec:deformation}

In this section, we recall a classical construction for deforming a reversible Finsler metric while preserving its marked length spectrum and Holmes--Thompson volume. 
The construction relies on small exact symplectic deformations of its unit cotangent sphere bundle.
We then show how to ensure that deformed metrics are not isometric to the original one.

\begin{theorem}
Let $M$ be a closed $n$-manifold with $n \geq 2$.
Every reversible Finsler metric on~$M$ admits arbitrarily small nonisometric reversible Finsler deformations preserving its marked length spectrum and volume.
\end{theorem}

\begin{proof}
Let $F$ be a reversible Finsler metric on~$M$.
Denote by $U_F^*M$ its unit cotangent sphere bundle and by~$\alpha$ the tautological one-form on~$T^*M$.

\medskip

We start with the following construction.

\begin{proposition} \label{prop:deformation}
Let $\Phi:T^*M \to T^*M$ be an exact symplectomorphism  isotopic to the identity and sufficiently $C^\infty$-close to the identity.
Then, $\Phi(U_F^*M)$ is the unit cotangent sphere bundle of a Finsler metric~$F'$ on~$M$ with the same oriented marked length spectrum and volume as~$F$.

Moreover, if $F$ is reversible and $\Phi$ commutes with the antipodal involution $A:T^*M\to T^*M$, then $F'$ is reversible as well.
\end{proposition}

\begin{proof}
For $\Phi$ sufficiently close to the identity, the hypersurface~$\Phi(U_F^*M)$ intersects each cotangent fiber in a smooth, strictly convex hypersurface enclosing the origin.
It is therefore the unit cotangent sphere bundle of a Finsler metric~$F'$.
If $F$ is reversible and $\Phi$ commutes with~$A$, then $\Phi(U_F^*M)$ is invariant under~$A$, so $F'$ is reversible as well.

Let us prove that the metrics~$F$ and~$F'$ have the same oriented marked length spectrum.
Since $\Phi$ is symplectic, it sends the characteristic foliation of~$U_F^*M$ to that of~$U_{F'}^*M$.
Consequently, it induces a bijection between the periodic geodesic orbits of~$F$ and those of~$F'$.
For every closed $F$-geodesic~$\gamma$, denote by~$\hat{\gamma}$ its lift to~$U_F^*M$. 
The periodic characteristic~$\Phi(\hat{\gamma})$ projects to a closed $F'$-geodesic, which we denote by~$\phi(\gamma)$. 
Since $\Phi$ is isotopic to the identity, the curves~$\gamma$ and~$\phi(\gamma)$ are freely homotopic.
Moreover, since $\Phi$ is exact, there exists a smooth function $G:T^*M \to \R$ such that $\Phi^* \alpha = \alpha + dG$.
Therefore,
\[
\length_{F'}(\phi(\gamma)) = \int_{\Phi(\hat{\gamma})} \alpha = \int_{\hat{\gamma}} \Phi^* \alpha = \int_{\hat{\gamma}} \alpha = \length_F(\gamma).
\]
Thus, the map~$\Phi$ induces a bijection between the closed geodesics of~$F$ and~$F'$ which preserves both their lengths and their free homotopy classes.
That is,
\[
\MLS_{F'} = \MLS_F.
\]

Finally, the volumes of~$F$ and~$F'$ are equal because the unit co-ball bundles~$B_F^*M$ and~$B_{F'}^*M$ are symplectomorphic by~$\Phi$.
\end{proof}

We now explain how to construct symplectomorphisms satisfying the hypotheses of Proposition~\ref{prop:deformation}.

\medskip

Let  $H:T^*M \to \R$ be a smooth, compactly supported odd Hamiltonian, that is, such that $H \circ A = -H$, where $A$ is the antipodal involution on~$T^*M$.
Denote by~$X=X_H$ its Hamiltonian vector field, defined by 
\begin{equation}
\label{eq:hamiltonian_vector_field}
\iota_{X_H} \omega = -dH.
\end{equation}

We claim that $X$ is $A$-invariant.
Indeed, note first that $A$ is antisymplectic, that is, $A^* \omega = - \omega$ because $A^* \alpha = -\alpha$.
Applying the pullback operator~$A^*$ to each side of~\eqref{eq:hamiltonian_vector_field}, we obtain, on one side,
\[ A^*(\iota_X\omega) = \iota_{A^*X}(A^*\omega) = -\iota_{A^*X}\omega, \]
and on the other side,
\[ A^*(-dH) = d(-H\circ A) = dH. \]
Equating these two expressions and using~\eqref{eq:hamiltonian_vector_field} yields
\[
\iota_{A^*X}\,\omega=-dH=\iota_X\omega.
\]
Since $\omega$ is nondegenerate, it follows that $A^*X=X$, as claimed.

Therefore, the Hamiltonian flow of~$H$ satisfies
\[
A \circ \Phi_H^t = \Phi_H^t \circ A.
\]
Moreover, every Hamiltonian diffeomorphism is exact symplectic.
Hence, for sufficiently small~$t$, the map~$\Phi=\Phi_H^t$ has all the required properties.

\medskip

We finally explain how the Hamiltonian can be chosen so that the resulting Finsler metric is not isometric to~$F$.

\medskip
\
We will need the following extension property.

\begin{lemma} \label{lem:extension}
Let $x_0 \in M$.
Every smooth, compactly supported vector field~$V$ on~$T^*_{x_0}M$ can be extended to a compactly supported Hamiltonian vector field~$X_H$ on~$T^*M$.
Moreover, if $V$ is odd, the corresponding Hamiltonian~$H$ can be chosen odd as well.
\end{lemma}

\begin{proof}
Consider canonical coordinates~$(x,\xi)$ on~$T^*M$ induced from a coordinate system on an open set~$U \subseteq M$ centered at~$x_0$.
We identify each coordinate vector~$x$ with an element of~$T_{x_0}M$.
Choose a smooth cutoff function $\chi:M \to [0,1]$ with support in~$U$ and equal to~$1$ near~$x_0$.
Using the canonical identification~$T_\xi(T_{x_0}^*M) \simeq T_{x_0}^*M$, define
\[
H(x,\xi) = - \chi(x) V(\xi)(x)
\]
on~$T^*U$, and extend it by zero to the rest of~$T^*M$.
Note that if $V$ is odd, the Hamiltonian~$H$ is also odd.

In the canonical coordinates, the Hamiltonian vector field is given by
\[
X_H = \frac{\partial H}{\partial \xi} \frac{\partial}{\partial x} - \frac{\partial H}{\partial x} \frac{\partial}{\partial \xi}.
\]
Since the coordinates are centered at~$x_0$ and $\chi=1$ near~$x_0$, we have
\[
\frac{\partial H}{\partial \xi}(x_0,\xi) = 0 \qquad \textrm{ and  } \qquad \frac{\partial H}{\partial x}(x_0,\xi) = -V(\xi)
\]
for every~$\xi \in T^*_{x_0}M$.
Hence,
\[
X_H(x_0,\xi) = V(\xi). \qedhere
\]
\end{proof}

To first order, the Hamiltonian flow induced by~$X_H$ deforms the unit co-sphere~$S_{x_0} = U_{F,x_0}^*M$ in the prescribed vertical direction~$V$.

\medskip

\forget

Fix~$x_0 \in M$.
Denote by~$S_{x_0}=U_{F,x_0}^*M$ the unit co-sphere in~$T_{x_0}^*M$.
Let~$V$ be a smooth vertical vector field along~$S_{x_0}$, that is,
\[
V(\xi) \in T_\xi(T_{x_0}^*M) \simeq T_{x_0}^*M
\]
for every~$\xi \in S_{x_0}$.
Assume also that $V(-\xi)=-V(\xi)$.
Extend~$V$ smoothly to an odd vertical vector field with compact support on a neighborhood of~$S_{x_0}$.
Choose local coordinates centered at~$x_0$, so that $x_0$ corresponds to~$x=0$, and use them to identify each coordinate vector~$x$ with an element of~$T_{x_0}M$.
Using the canonical identification~$T_\xi(T_{x_0}^*M) \simeq T_{x_0}^*M$, define locally 
\[
H(x,\xi) = - \chi(x,\xi) V(\xi)(x)
\]
where $\chi$ is an even cutoff function equal to~$1$ in a neighborhood of~$\{ x_0 \} \times S_{x_0}$.
Since $V$ is odd and~$\chi$ is even, the Hamiltonian~$H$ is odd.
In the induced canonical coordinates~$(x,\xi)$ on~$T^*M$, the Hamiltonian vector field is given by
\[
X_H = \frac{\partial H}{\partial \xi} \frac{\partial}{\partial x} - \frac{\partial H}{\partial x} \frac{\partial}{\partial \xi}.
\]
Since the coordinates are centered at~$x_0$ and $\chi=1$ near~$\{ x_0 \} \times S_{x_0}$, for every~$\xi \in S_{x_0}$, we have
\[
\frac{\partial H}{\partial \xi}(x_0,\xi) = 0 \qquad \textrm{ and  } \quad \frac{\partial H}{\partial x}(x_0,\xi) = -V(\xi)
\]
for every~$\xi \in S_{x_0}$.
Hence,
\[
X_H(x_0,\xi) = V(\xi) \in T_\xi(T_{x_0}^*M).
\]

Thus, to first order, the Hamiltonian flow deforms the unit co-sphere~$S_{x_0}$ in the prescribed vertical direction~$V$.

\forgotten

It remains to choose the vector field~$V$ appropriately.
The family of linear equivalence classes of the unit co-spheres~$S_x = U^*_{F,x}M$, with $x \in M$, is finite-dimensional.
By contrast, the space of infinitesimal deformations of~$S_{x_0}$, preserving its central symmetry, modulo infinitesimal linear transformations, is infinite-dimensional.
We may therefore choose an odd vector field~$V$ on~$T^*_{x_0}M$ so that the deformed unit co-spheres~$\Phi_t(U_{F,x_0}^*M)$, corresponding to the Hamiltonian flow~$\Phi_t$ induced by the Hamiltonian given by Lemma~\ref{lem:extension}, are not linearly equivalent to any~$S_x$, for $t$ small enough.
Therefore, the Finsler metrics~$F_t'$ with unit cotangent sphere bundle~$\Phi_t(U_{F}^*M)$, given by Proposition~\ref{prop:deformation}, are not isometric to~$F$.
\end{proof}


\forget
We finally explain how the Hamiltonian can be chosen so that the resulting Finsler metric is not isometric to~$F$.
Fix~$x_0 \in M$.
Denote by~$S_{x_0}=U_{F,x_0}^*M$ the unit co-sphere in~$T_{x_0}^*M$.
Let~$V$ be a smooth vertical vector field along~$S_{x_0}$, that is,
\[
V(\xi) \in T_\xi(T_{x_0}^*M) \simeq T_{x_0}^*M
\]
for every~$\xi \in S_{x_0}$.
Assume also that $V(-\xi)=-V(\xi)$.
Extend~$V$ smoothly to an odd vertical vector field with compact support on a neighborhood of~$S_{x_0}$.
Choose local coordinates centered at~$x_0$, so that $x_0$ corresponds to~$x=0$, and use them to identify each coordinate vector~$x$ with an element of~$T_{x_0}M$.
Using the canonical identification~$T_\xi(T_{x_0}^*M) \simeq T_{x_0}^*M$, define locally 
\[
H(x,\xi) = - \chi(x,\xi) V(\xi)(x)
\]
where $\chi$ is an even cutoff function equal to~$1$ in a neighborhood of~$\{ x_0 \} \times S_{x_0}$.
Since $V$ is odd and~$\chi$ is even, the Hamiltonian~$H$ is odd.
In the induced canonical coordinates~$(x,\xi)$ on~$T^*M$, the Hamiltonian vector field is given by
\[
X_H = \frac{\partial H}{\partial \xi} \frac{\partial}{\partial x} - \frac{\partial H}{\partial x} \frac{\partial}{\partial \xi}.
\]
Since the coordinates are centered at~$x_0$ and $\chi=1$ near~$\{ x_0 \} \times S_{x_0}$, for every~$\xi \in S_{x_0}$, we have
\[
\frac{\partial H}{\partial \xi}(x_0,\xi) = 0 \qquad \textrm{ and  } \quad \frac{\partial H}{\partial x}(x_0,\xi) = -V(\xi)
\]
for every~$\xi \in S_{x_0}$.
Hence,
\[
X_H(x_0,\xi) = V(\xi) \in T_\xi(T_{x_0}^*M).
\]
Thus, to first order, the Hamiltonian flow deforms the unit co-sphere~$S_{x_0}$ in the prescribed vertical direction~$V$.

It remains to choose~$V$ appropriately.
The family of linear equivalence classes of the unit co-spheres~$S_x = U_{F,x}^*M$, with $x \in M$, is finite-dimensional.
By contrast, the space of infinitesimal deformations of~$S_{x_0}$, preserving its central symmetry, modulo infinitesimal linear transformations, is infinite-dimensional.
We may therefore choose~$V$ so that the deformed unit co-sphere~$S'_{x_0}=U_{F',x_0}^*M$, corresponding to the symplectomorphism~$\Phi_t$, is not linearly equivalent to any~$S_x$, for $t$ small enough.
Therefore, the Finsler metrics~$F$ and~$F'$ are not isometric.
\forgotten

\begin{remark}
An alternative approach, based on the notion of enveloping functions introduced in~\cite{Iva13} and~\cite{BI16}, was used in~\cite{chen} to establish the flexibility of Finsler tori without conjugate points and in~\cite[Appendix]{Sab19} for Zoll Finsler metrics.
A similar approach can be used in the present setting to construct Finsler deformations preserving the marked length spectrum and the Holmes--Thompson volume.
\end{remark}

\section{Open question} \label{sec:open}

We conclude with an open question related to this paper.

\medskip

In Theorem~\ref{theo:main}, the assumption that the Finsler metric~$F_0$ has Anosov geodesic flow is used in two ways:
\begin{itemize}
\item First, to guarantee that there are no conjugate points, or equivalently, that every closed geodesic minimizes length in its homotopy class.
This ensures that every geodesic wall system is minimal, allowing us to apply Theorem~\ref{theo:discrete}.
\item Second, to ensure that the absence of conjugate points persists under small perturbations of the metric.
Such a perturbation is needed to approximate the Liouville current by geodesic wall systems using the theorem of Irie and Liu--Liu.
\end{itemize}

However, as far as we know, the perturbation step may not be necessary. 
More precisely, it is conceivable that, on every Finsler surface without conjugate points, the Liouville current can be approximated by geodesic wall systems.
If so, the Anosov assumption on~$F_0$ in our main theorem could be relaxed to the property of having no conjugate points.

\medskip

This leads to the following question.

\begin{question} \label{q:1}
Let $M$ be a closed surface (more generally, a closed manifold) endowed with a Finsler metric~$F_0$ without conjugate points.
Let $F$ be another Finsler metric on~$M$ such that $\MLS_{F_0} \leq \MLS_F$.
Does the following inequality hold:
\[
\area(F_0) \leq \area(F)?
\]
\end{question}

The answer is affirmative on the two-torus by Theorem~\ref{theo:torus}. 

\medskip

As a final remark, we recall that the $C^2$-interior of the space of Finsler metrics without conjugate points coincides with the space of Finsler metrics with Anosov geodesic flow.
See~\cite{rug}, \cite{CISM}, and~\cite[Proposition~6.1.2]{schulz} for results in the Riemannian, Hamiltonian, and Finsler settings.

\forget

In Theorem~\ref{theo:main}, the Anosov assumption on the geodesic flow of the Finsler metric~$F_0$ is used only through its structural stability along with the minimality of the closed geodesics, or alternatively, through the density theorem for closed-geodesic currents; see Remark~\ref{rem:alternative}.
This allows us to extend an inequality initially known for closed geodesics to all geodesic currents and, in particular, to the Liouville current of~$F_0$.

Recall that every Anosov Finsler metric has no conjugate points; see Theorem~\ref{theo:Anosov}.\eqref{Anosov1}. 
It is therefore natural to ask whether the Anosov assumption can be replaced by the weaker assumption that the metric has no conjugate points.

\begin{question} \label{q:1}
Let $M$ be a closed surface (more generally, a closed manifold) endowed with a Finsler metric~$F_0$ without conjugate points.
Let $F$ be another Finsler metric on~$M$ such that $\MLS_{F_0} \leq \MLS_F$.
Does the following inequality hold:
\[
\area(F_0) \leq \area(F)?
\]
\end{question}

The answer is affirmative on the two-torus by Theorem~\ref{theo:torus}.
For surfaces of negative Euler characteristic, two possible approaches are the following:

\begin{enumerate}
\item Show that a metric without conjugate points can be approximated by an Anosov metric (and apply the main theorem to conclude).
A result in this direction identifies the $C^2$-interior of the space of Riemannian metrics without conjugate points with the space of Anosov Riemannian metrics; see~\cite{rug}.
\item Show that the Liouville current of a metric without conjugate points can be approximated by a geodesic wall system.
In contrast to the case of Anosov metrics, it may not be possible to approximate the Liouville current by geodesic wall systems composed of a single closed geodesic (after suitable normalization) as in Theorem~\ref{theo:dense}, but it may be possible to approximate it by  more general geodesic wall systems.
With such an approximation result, an affirmative answer to Question~\ref{q:1} would follow from Theorem~\ref{theo:discrete} (the discrete analogue of the main theorem), since in the absence of conjugate points, every geodesic wall system is minimal.
\end{enumerate}

\forgotten

\forget
\subsection{Marked length spectrum of possibly irreversible metrics.}
A second natural question is whether the reversible assumption on the Finsler metric can be removed (\ie, allowing Finsler metrics where $F(-v) \neq F(v)$).
Of course, in this case, one needs to consider the oriented marked length spectrum.

\begin{question} \label{q:2}
Does Theorem~\ref{theo:main} hold true for irreversible Finsler metrics?
\end{question}

The first author~\cite{Cos} constructed discretizations of possibly irreversible Finsler metrics using \emph{fine complexes}.
However, some questions remain unanswered.
For example, it is not clear yet that:

\begin{enumerate}
\item Every Anosov metric on a surface can be discretized by a fine structure that is pseudostraight (\ie, whose Postnikov strands on the universal cover form neither closed arcs nor directed bigons).
\item Every fine structure on a surface can be tightened, without changing its oriented marked length spectrum and without decreasing its area, until it has no conjugate points.
\item The area functional is monotonic with respect to the oriented marked length spectrum among fine structures without conjugate points.
\end{enumerate}

Proving the last two points would yield an affirmative answer to Question~\ref{q:2}.
\forgotten

\forget

\section{Appendix} \label{sec:appendix}

In this appendix, we give a direct proof of the discrete approximations of the marked length spectrum and area (see~\cite{Cos} and Corollary~\ref{coro:W}) for Riemannian metrics on closed surfaces.

\begin{theorem} \label{theo:appendix}
Let $M$ be an orientable closed surface endowed with a Riemannian metric~$g$.
Then, for every $\varepsilon >0$, there exist a wall system~$\W$ on~$M$ such that 
\begin{align*}
\MLS_{\W} & \simeq_\varepsilon \MLS_g \\
\area(\W) & \simeq_\varepsilon \frac{\pi}{4} \cdot \area(g).
\end{align*}
\end{theorem}


The proof proceeds as follows.
We first approximate the Riemannian surface by piecewise linear surfaces embedded in~$\R^3$.
We then use integral geometry formulas expressing the length of a curve on the surface in terms of the expected number of its intersections with a random plane of~$\R^3$, and the area of the surface in terms of the expected number of intersection points on the surface between two random planes.
The wall system~$\W$ is obtained by intersecting the surface with finitely many independent random planes.
The desired approximations of the marked length spectrum and area then follow from these integral formulas by applying suitable versions of the law of large numbers.

\medskip

We start by introducing some notations.

\begin{definition}
The Grassmannian of (affine) planes in~$\R^3$, denoted by~$\HH = \Gr(2,3)$, is the set of all (affine) planes in~$\R^3$.
It is a homogeneous space on which the isometry group~$E(3)$ of~$\R^3$ acts transitively.
It identifies with
\[
\HH = (S^2 \times \R) / \pm
\]
where every plane~$P$ is given by the equation
$
\langle u,\cdot \rangle = t
$
with $(u,t) \in S^2 \times \R$.
The measure~$\mu$ on~$\HH$ locally defined by
\[
\mu = du \, dt
\]
is invariant by the action of~$E(3)$.
Denote by~$\HH_B \subseteq \HH$ the subset of planes in~$\R^3$ intersecting the unit ball~$B$ of~$\R^3$.
Note that 
\[
\mu(\HH_B) = 4\pi.
\]
Define the following probability measure~$\Prob$ on~$\HH$
\[
\Prob=\frac{1}{4 \pi} \, \mu \mres \HH_B.
\] 
\end{definition}

We will need the following integral geometry formulas established by Santal\'o~\cite[III.14.5, Example 2]{San} (where the equation~\eqref{eq:S3} below is obtained as a combination of~\eqref{eq:S1} and~\eqref{eq:S2}.)

\begin{theorem} \label{theo:San}
Let $B$ be the unit ball of~$\R^3$.
\begin{enumerate}
\item The length of every piecewise immersed curve~$c$ in~$B$ is
\begin{equation} \label{eq:S1} 
\length(c) = 4 \, \int_{P \in \HH} \#(c \cap P) \, d\Prob. 
\end{equation}
\item The area of every closed surface~$M$ piecewise smoothly embedded in~$\R^3$ is
\begin{align}
\area(M) & = \frac{8}{\pi} \, \int_{P \in \HH} \length(M \cap P) \, d\Prob \label{eq:S2} \\
 & = \frac{32}{\pi} \, \int_{P \in \HH} \int_{P' \in \HH} \#(M \cap P \cap P') \, d\Prob \, d\Prob. \label{eq:S3}
\end{align}
\end{enumerate}
\end{theorem}

Note that for almost every plane~$P \in \HH$, the intersection~$c \cap P$ is formed of finitely many points and the intersection~$M \cap P$ is a finite union of piecewise smoothly embedded closed curves in~$\R^3$. \\

We can now proceed to the proof of the discretization theorem for Riemannian metrics on surfaces.

\begin{proof}[Proof of Theorem~\ref{theo:appendix} for Riemannian metrics]
%
%
Let $g$ be a Riemannian metric on the orientable closed surface~$M$.
By~\cite[Lemma~6]{Bur}, the metric~$g$ can be approximated, with bilipschitz distortion arbitrarily close to~$1$, by a piecewise flat metric with conical singularities.
After rescaling, every orientable closed piecewise flat surface with conical singularities admits a piecewise linear isometric embedding into the unit ball~$B \subseteq \R^3$, after possibly subdividing its triangulation; see~\cite{BZ}.
As a consequence, by a metric approximation, we can assume that $M$ is a piecewise linear embedded surface in~$B \subseteq \R^3$ endowed with the induced piecewise flat metric.
Denote by~$\Theta$ a triangulation of~$M$ compatible with this piecewise linear structure.
After some possible simplicial subdivisions, we may assume that the diameter of every simplex~$\Delta \in \Theta$ satisfies
$\diam(\Delta) < \frac{1}{4} \sys(g)$. 
Here and below, a simplex of~$M$ will always designate a closed $2$-simplex of~$\Theta$. 

\medskip

We argue as in the proof of~\cite[Theorem~4.9]{CS}.
The wall system~$\W$ on~$M$ is formed by the intersections~$M \cap P$ of random planes with respect to the probability measure~$\Prob$.
More precisely, take $n$ independent identically distributed (i.i.d.) random planes $P_1,\dots,P_n$ of~$\HH$ with probability distribution~$\Prob$.
Almost surely, the intersections of these planes with~$M$ are finite unions of simple closed piecewise geodesics, transverse to the $1$-skeleton of~$\Theta$, and form an unweighted wall system~$\W$.
We also denote by~$\W$ the graph in~$M$ given by the union of its walls.

\medskip

The following lemma shows that a segment contained in a simplex of~$M$ is minimizing both for the piecewise flat metric and the discrete metric.

\begin{lemma} \label{lem:segment}
Let $\Delta$ be a simplex of~$M$.
\begin{enumerate}
\item If $x,y \in \Delta$, then
\[
d_g(x,y) = \length_g([x,y]).
\]
Moreover, the segment~$[x,y] \subseteq \Delta$ is the unique $g$-minimizing arc between~$x$ and~$y$. \label{segment1}
\item If $x,y \in \Delta \setminus \W$, then
\[
d_\W(x,y) = \length_\W([x,y]).
\]
\label{segment2}
\end{enumerate}
\end{lemma}

\begin{proof}
Let us first prove the first assertion.
Since $\Delta$ is a Euclidean simplex, the $g$-length of the segment~$[x,y] \subseteq \Delta$ is equal to the Euclidean distance between its endpoints. 
Therefore,
\[
\length_g([x,y]) = d_{\R^3}(x,y).
\]
On the other hand, every arc~$c$ of~$M$ joining~$x$ to~$y$ satisfies $\length_g(c) = \length_{\R^3}(c) \geq d_{\R^3}(x,y)$.
Applying this to~$[x,y]$, it follows that
\[
d_g(x,y)=\length_g([x,y]).
\]
Moreover, equality in the Euclidean length-distance inequality occurs only for the segment~$[x,y]$. 
Hence~$[x,y]$ is the unique $g$-minimizing arc joining~$x$ to~$y$.

For the second assertion, consider one of the planes~$P_i$ defining the wall system~$\W$.
Since $x,y \notin \W$, the plane~$P_i$ intersects the segment~$[x,y]$ if and only if it separates~$x$ and~$y$ in~$\R^3$.
In this case, every arc of~$M$ joining~$x$ to~$y$ intersects~$P_i$, and hence the corresponding wall~$M \cap P_i$.
Thus every such arc crosses at least all the walls crossed by~$[x,y]$.
Therefore,
\[
\length_\W([x,y]) \leq d_\W(x,y).
\]
The reverse inequality is obvious. 
Hence, the equality.
\end{proof}

On every simplex of~$M$, the distance for the piecewise flat metric can be approximated, up to rescaling, by a discrete metric.
More precisely, we have the following approximation result, which relies on the first integral formula~\eqref{eq:S1} of Theorem~\ref{theo:San}.

\begin{lemma} \label{lem:dWdM}
Let~$\varepsilon >0$.
With probability converging to~$1$ as $n \to \infty$, every pair of points $x,y \in \Delta$, where $\Delta$ is a simplex of~$M$, satisfies
\begin{equation} \label{eq:dWdM}
\frac{1}{n} \, d_\W(x,y) \simeq \frac{1}{4} \, d_g(x,y) \pm \varepsilon.
\end{equation}
\end{lemma}

\begin{proof}
Fix a simplex~$\Delta$.
For every $x,y \in \Delta$, consider the function $f_{x,y}:\HH \to \{0,1\}$ defined by
\[
f_{x,y}(P) = \min \{ \#([x,y] \cap P),1 \}.
\]
This is a measurable bounded function.
Moreover, $f_{x,y}(P) = \#([x,y] \cap P)$ for almost every~$P\in\HH$. 

By the formula~\eqref{eq:S1} of Theorem~\ref{theo:San}, the random variables~$X_i=f_{x,y}(P_i)$ with $1 \leq i \leq n$ are i.i.d. with finite expected value 
\[
\E(X_i) = \int_{P \in \HH} \#([x,y] \cap P) \, d\Prob = \frac{1}{4} \, d_g(x,y).
\]
Note that $\E(|X_i|) = \E(X_i) < \infty$.
By the weak law of large numbers applied to~$\{X_i\}$, 
we derive that
\[
\left| \frac{1}{n} \, \sum_{i=1}^n \#([x,y] \cap P_i) - \frac{1}{4} \, d_g(x,y) \right| < \varepsilon
\]
with probability converging to~$1$ as $n \to \infty$.


We need this convergence to be uniform in~$(x,y) \in \Delta \times \Delta$.
For this, we apply a uniform version of the law of large numbers; see~\cite[Lemma~1]{Tau}. 
We simply need to verify the following conditions:
\begin{enumerate}
\item The parameter space~$\Delta \times \Delta$ is compact and $|f_{x,y}(P)| \leq 1$.
\item For every $(x_0,y_0) \in \Delta \times \Delta$ and almost every~$P \in \HH$, the map $(x,y) \mapsto f_{x,y}(P)$ is continuous at~$(x_0,y_0)$.
Indeed, if $x_0,y_0 \notin P$, then for $(x,y)$ sufficiently close to~$(x_0,y_0)$, the points $x$ and~$y$ remain on the same respective sides of~$P$ as $x_0$ and~$y_0$.
Hence $(x,y)\mapsto f_{x,y}(P)$ is locally constant near~$(x_0,y_0)$.
Moreover, the set of planes containing~$x_0$ or~$y_0$ has zero measure.
\item The function $(P,x,y) \mapsto f_{x,y}(P)$ is separable in the sense of Doob.
Indeed, fix a dense countable subset $\Delta'  \subseteq \Delta$.
Let $U \subseteq \Delta \times \Delta$ be open.
%
If $P \cap [x,y] \neq \emptyset$ for some $(x,y) \in U$, then there are arbitrarily close pairs $(x',y') \in \Delta' \times \Delta'$ for which $P \cap [x',y'] \neq \emptyset$.
Hence such a pair can be chosen in $U \cap (\Delta'\times\Delta')$.
Similarly, if $P \cap [x,y] = \emptyset$ for some $(x,y) \in U$, then $x$ and~$y$ lie strictly on the same side of~$P$, and the same holds for all pairs sufficiently close to~$(x,y)$.
Thus the values of~$f_{x,y}(P)$ on~$U$ are determined by its values on $U \cap (\Delta' \times \Delta')$, which gives the required separability.
\end{enumerate}

It follows that
\[
\sup_{x,y \in \Delta} \, \left| \frac{1}{n} \, \sum_{i=1}^n f_{x,y}(P_i) - \frac{1}{4} \, d_g(x,y) \right| < \varepsilon
\]
with probability converging to~$1$ as $n \to \infty$. 
In particular, for every $x,y \in \Delta \setminus \W$, we obtain
\[
\left| \frac{1}{n} \, \sum_{i=1}^n \#([x,y] \cap P_i) - \frac{1}{4} \, d_g(x,y) \right| < \varepsilon.
\]

On the other hand, by Lemma~\ref{lem:segment}.\eqref{segment2}, we have
\[
d_\W(x,y) = \length_\W([x,y]) = \sum_{i=1}^n \#([x,y] \cap P_i)
\]
whenever $x,y \in \Delta \setminus \W$.

Therefore, the equation~\eqref{eq:dWdM} holds for $x,y \in \Delta \setminus \W$.
The general case follows by slightly perturbing the points.
\end{proof}

\begin{lemma} \label{lem:dWdM3}
Let~$\varepsilon >0$.
With probability converging to~$1$ as $n \to \infty$, the following properties hold:
\begin{enumerate}
\item For every geodesic arc~$\gamma$ of minimal $g$-length in~$M$,
\[
\frac{1}{n} \, \length_\W(\gamma) \leq \frac{1}{4} \, \length_g(\gamma) + \varepsilon.
\]
\item For every arc~$c$ of minimal $\W$-length in~$M$, there exists a piecewise geodesic arc~$\gamma$ with the same endpoints as~$c$, homotopic to~$c$ relative to its endpoints, such that
\[
\frac{1}{n} \, \length_\W(c) \geq \frac{1}{4} \, \length_g(\gamma) - \varepsilon.
\]
\end{enumerate}

In particular, for every pair of points $x,y \in M$,
\begin{equation} \label{eq:dWdM3}
\frac{1}{n} \, d_\W(x,y) \simeq \frac{1}{4} \, d_g(x,y) \pm \varepsilon.
\end{equation}
\end{lemma}

\begin{proof}
Let $N=|\Theta|$ be the number of simplices of~$M$.
Fix~$\eta>0$.
Since the triangulation is finite, we may assume, with probability converging to~$1$ as $n \to \infty$, that the conclusion of
Lemma~\ref{lem:dWdM} holds simultaneously on every simplex with error~$\eta$.

Let us prove the first assertion.
Let $\gamma$ be an $g$-length-minimizing geodesic arc in~$M$.
By Lemma~\ref{lem:segment}.\eqref{segment1}, the arc~$\gamma$ cannot leave a simplex and later return to it.
Thus, it visits every simplex at most once.
Therefore, $\gamma$ decomposes into at most~$N$ geodesic segments
\[
\gamma = [x_0,x_1] \cup \cdots \cup [x_{m-1},x_m]
\]
each contained in a simplex, where $m \leq N$.
By Lemma~\ref{lem:segment} and~\eqref{eq:dWdM},
\[
\frac{1}{n} \length_\W([x_j,x_{j+1}]) = \frac{1}{n} d_\W(x_j,x_{j+1}) \leq \frac{1}{4} d_g(x_j,x_{j+1}) + \eta.
\]
Summing over~$j$, we obtain
\[
\frac{1}{n} \length_\W(\gamma) \leq \frac{1}{4} \length_g(\gamma) + N \eta.
\]
In particular, by taking~$\gamma$ to be an $g$-minimizing geodesic between arbitrary points $x,y \in M$,
\begin{equation} \label{eq:<N}
\frac{1}{n} \, d_\W(x,y) \leq \frac{1}{4} \, d_g(x,y) + N \eta.
\end{equation}

Now, instead of proving directly the second assertion, we establish the reverse inequality for the distances, namely~\eqref{eq:dWdM3}.
Let $c$ be a $\W$-minimizing arc between~$x$ and~$y$.
After a slight perturbation, we may assume that its intersection points with~$\W$ do not lie in the $1$-skeleton of~$\Theta$.

Suppose that $c$ leaves a simplex~$\Delta$ at a point~$x'$ and later returns to it at a point~$y'$.
Replace the corresponding subarc~$c_{x' \to y'}$ of~$c$ by the segment~$[x',y'] \subseteq \Delta$.
By Lemma~\ref{lem:segment}.\eqref{segment2},
\begin{equation} \label{eq:equality}
\length_\W([x',y']) = d_\W(x',y') \leq  \length_\W(c_{x' \to y'}).
\end{equation}
Since $c$ is $\W$-length minimizing, equality must hold.
Thus, the replacement does not change the $\W$-length.

Repeating this operation, we obtain a $\W$-length-minimizing arc~$c'$ which visits each simplex at most once.
At this stage, we do not yet claim that $c'$ is homotopic to~$c$ relative to its endpoints.
By Lemma~\ref{lem:segment}.\eqref{segment2}, straightening each portion of~$c'$ contained in a simplex gives a
piecewise geodesic arc
\[
\gamma = [x_0,x_1] \cup \cdots \cup [x_{m-1},x_m]
\]
with the same $\W$-length and the same endpoints as~$c'$, homotopic to~$c'$, where $m \leq N$.
By~\eqref{eq:dWdM} and Lemma~\ref{lem:segment}.\eqref{segment1},
\[
\frac{1}{n} \, \length_\W([x_j,x_{j+1}]) \geq \frac{1}{4} \, d_g(x_j,x_{j+1}) - \eta = \frac{1}{4} \length_g([x_j,x_{j+1}]) - \eta.
\]
Summing over~$j$, we obtain
\[
\frac{1}{n} \, \length_\W(c) = \frac{1}{n} \, \length_\W(\gamma) \geq \frac{1}{4} \, \length_g(\gamma) -N \eta.
\]
Since $\length_g(\gamma) \geq d_g(x,y)$, it follows that
\begin{equation} \label{eq:>N}
\frac{1}{n} \, d_\W(x,y) \geq \frac{1}{4} \, d_g(x,y) - N \eta.
\end{equation}

Combining~\eqref{eq:<N} and~\eqref{eq:>N}, we obtain
\begin{equation} \label{eq:pmN}
\frac{1}{n} \, d_\W(x,y) \simeq \frac{1}{4} \, d_g(x,y) \pm N \eta.
\end{equation}

We can now prove the second assertion.
It remains to show that~$c'$ is homotopic to~$c$ relative to its endpoints.
Suppose that one of the subarcs~$c_{x' \to y'}$ is not homotopic to the segment~$[x',y']$ relative to its endpoints.
Then a point~$z'$ of~$c_{x' \to y'}$ is at ~$g$-distance at least~$\frac{1}{4} \sys(g)$ from~$x'$ and~$y'$. 
(Otherwise, the loop~$[x',y'] \cup c_{x' \to y'}$ would lie in the disk~$B_g(x',\frac{1}{2} \sys(g))$, which is contractible in~$M$.)
By the triangle inequality and~\eqref{eq:pmN}, we obtain
\begin{align*}
\length_\W(c_{x' \to y'}) & \geq d_W(x',z') + d_\W(z',y') \\
 & \geq \frac{n}{4} \, d_g(x',z') + \frac{n}{4} \, d_g(z',y') - 2nN \eta \\
 & \geq \frac{n}{8} \, \sys(g) - 2nN \eta.
\end{align*}
On the other hand, the $g$-diameter of the simplices is less than~$\frac{1}{4} \sys(g)$.
Thus, $d_g(x',y') \leq \frac{1}{4} \sys(g)$.
By Lemma~\ref{lem:segment}.\eqref{segment2} and~\eqref{eq:dWdM}, we obtain
\[
\length_\W([x',y']) = d_\W(x',y') \leq \frac{n}{4} \, d_g(x',y') + n \eta \leq \frac{n}{16} \, \sys(g) + n \eta.
\]
For $\eta>0$ small enough, we deduce that 
\[
\length_\W([x',y']) <  \length_\W(c_{x' \to y'})
\]
which contradicts the equality in~\eqref{eq:equality}.

The two assertions follow by taking~$\eta < \varepsilon/N$. 
\end{proof}

\begin{lemma} \label{lem:syst}
Let $\varepsilon>0$.
With probability converging to~$1$ as $n \to \infty$, we have
\[
\frac{1}{n} \, \sys(\W) \gtrsim_\varepsilon \frac{1}{4} \, \sys(g).
\]

In particular, for $\varepsilon>0$ sufficiently small, $\frac{1}{n} \, \sys(\W) \geq \frac{1}{8} \, \sys(g)$.
\end{lemma}

\begin{proof}
Fix $\eta >0$.
We may assume, with probability converging to~$1$ as $n \to \infty$, that the conclusion of Lemma~\ref{lem:dWdM3} holds with error~$\eta$.

Let $c$ be a noncontractible closed curve of~$M$.
There exist two points $x,y \in c$ such that $d_g(x,y) \geq \frac{1}{2} \sys(g)$.
Otherwise, the curve~$c$ would be contained in a disk of radius less than~$\frac{1}{2} \sys(g)$ and would be contractible in~$M$.
The two subarcs of~$c$ joining~$x$ to~$y$ both have $\W$-length at least~$d_\W(x,y)$.
Therefore, by the last assertion of Lemma~\ref{lem:dWdM3},
\[
\frac{1}{n} \, \length_\W(c) \geq \frac{2}{n} \, d_W(x,y) \geq \frac{1}{2} \, d_g(x,y) - 2 \eta \geq \frac{1}{4} \, \sys(g) - 2 \eta.
\]
Hence, 
$
\frac{1}{n} \, \sys(\W) \geq \frac{1}{4} \, \sys(g) - \varepsilon
$
for $\eta>0$ small enough.
\end{proof}

\begin{proposition} \label{prop:MLS}
Let $\varepsilon>0$.
With probability converging to~$1$ as $n \to \infty$, we have
\[
\frac{1}{n} \, \MLS_{\W} \simeq_\varepsilon \frac{1}{4} \, \MLS_g.
\]
\end{proposition}

\begin{proof}
Fix~$\eta>0$.
We may assume, with probability converging to~$1$ as $n \to \infty$, that the conclusions of Lemma~\ref{lem:dWdM3} and Lemma~\ref{lem:syst} hold with error~$\eta$.
Let $\alpha$ be a nontrivial unoriented free homotopy class of~$M$.

We first prove the upper bound.
Let $\gamma$ be an $g$-minimizing geodesic representing~$\alpha$.
Decompose $\gamma$ into arcs $\gamma=\gamma_1 \cup \cdots \cup \gamma_m$ such that
\[
\frac{1}{4} \, \sys(g) < \length_g(\gamma_j) < \frac{1}{2} \, \sys(g).
\]
Note that 
\[
m \leq \frac{4}{\sys(g)} \length_g(\gamma).
\]
Each arc~$\gamma_j$ is $g$-minimizing between its endpoints.
Thus, by the first assertion of Lemma~\ref{lem:dWdM3},
\[
\frac{1}{n} \, \length_\W(\gamma_j) \leq \frac{1}{4} \, \length_g(\gamma_j) + \eta.
\]
Summing over $j$, we obtain
\begin{align*}
\frac{1}{n} \, \length_\W(\gamma) & \leq \frac{1}{4} \, \length_g(\gamma) + m \eta \\
 & \leq \left( \frac{1}{4} + \frac{4}{\sys(g)} \eta \right) \, \length_g(\gamma).
\end{align*}
Since $\MLS_{\W}(\alpha) \leq \length_\W(\gamma)$ and $\MLS_g(\alpha) = \length_g(\gamma)$, choosing $\eta>0$ sufficiently small in terms of~$\varepsilon$ gives
\[
\frac{1}{n} \, \MLS_{\W}(\alpha) \lesssim_\varepsilon \frac{1}{4} \, \MLS_g(\alpha).
\]

We now prove the reverse inequality following the same argument.
Let $c$ be a closed curve of minimal $\W$-length representing~$\alpha$.
Decompose~$c$ into arcs $c=c_1 \cup \cdots \cup c_m$ such that
\[
\frac{1}{4} \, \sys(\W) < \length_\W(c_j) < \frac{1}{2} \, \sys(\W).
\]
Note that 
\[
m \leq \frac{4}{\sys(\W)} \, \length_\W(c) \leq \frac{1}{n} \, \frac{32}{\sys(g)} \, \length_\W(c)
\]
by the last part of Lemma~\ref{lem:syst} provided that $\eta$ is sufficiently small.
As above, each arc~$c_j$ is $\W$-minimizing between its endpoints.
Thus, by the second assertion of Lemma~\ref{lem:dWdM3}, there exists a piecewise geodesic arc~$\gamma_j$ with the same endpoints as~$c_j$, homotopic to~$c_j$ relative to its endpoints, such that
\[
\frac{1}{4} \, \length_g(\gamma_j) \leq \frac{1}{n} \, \length_\W(c_j)  + \eta.
\]
Summing over $j$, we obtain
\begin{align*}
\frac{1}{4} \, \length_g(\gamma) & \leq \frac{1}{n} \, \length_\W(c) + m \eta \\
 & \leq \frac{1}{n} \, \left( 1+  \frac{32}{\sys(g)} \eta \right) \, \length_\W(c)
\end{align*}
where $\gamma=\gamma_1 \cup \cdots \cup \gamma_m$.
By construction, $\gamma$ is freely homotopic to~$c$ and represents~$\alpha$.
Since $\MLS_g(\alpha) \leq \length_g(\gamma)$ and $\MLS_{\W}(\alpha) = \length_\W(c)$, choosing $\eta>0$ sufficiently small in terms of~$\varepsilon$ gives
\[
\frac{1}{4} \, \MLS_g(\alpha) \lesssim_\varepsilon \frac{1}{n} \, \MLS_{\W}(\alpha).
\]

Combining the two inequalities proves the proposition.
\end{proof}

The following generalization of the weak law of large numbers, for identically distributed random variables with a small proportion of dependent pairs, will be used in the proof of Proposition~\ref{prop:area}. 
Its proof can be found in~\cite[Theorem~4.13]{CS}.

\begin{theorem} \label{theo:wlln}
Fix a real valued random variable~$X$ with finite expected absolute value~$\E(|X|)<\infty$ and an integer~$n>0$.
Then the average $\overline X = \frac 1n \sum_i X_i$ of~$n$ random variables~$X_i$, each with the same distribution as~$X$, is near the expected value~$\E(X)$ with probability arbitrarily close to~$1$
if the proportion of nonindependent pairs
\[ 
p = \frac{ \# \{(i,j)\mid X_i\text{ and }X_j\text{ are not independent} \} }{n^2} 
\]
is small.
More precisely, for every $\eps,\delta>0$, there exists $p_0=p_0(X,\delta,\eps) >0$ such that if $p \leq p_0$, then
$\Prob \left( \left| \overline X - \E(X) \right| \geq \eps \right) \leq \delta$.
\end{theorem}

The second estimate needed in the proof of Theorem~\ref{theo:appendix} relates the discrete area of~$\W$ to the area of the piecewise flat metric~$g$.

\begin{proposition} \label{prop:area}
Let $\varepsilon>0$.
With probability converging to~$1$ as~$n \to \infty$, we have
\[
\frac{2}{n^2-n} \, \area(\W) \simeq \frac{\pi}{32} \, \area(g) \pm \varepsilon.
\]
\end{proposition}

\begin{proof}
The intersection counting function $f:\HH \times \HH \to \N \cup \{ \infty \}$ defined by
\[
f(P,P') = \#(M \cap P \cap P')
\]
is a nonnegative measurable function with finite expected value.
The $N_n=n(n-1)/2$ random variables~$X_{i,j}=f(P_i,P_j)$ with $i <j$ are identically distributed but not completely independent.
However, $X_{i,j}$ and~$X_{k,l}$ are independent whenever $\{i,j\} \cap \{k,l\} = \emptyset$.

We claim that the proportion of nonindependent pairs among the variables~$X_{i,j}$ tends to zero as~$n \to \infty$.
Indeed, for each~$X_{i,j}$, there are at most~$2n$ variables~$X_{k,l}$ whose index set intersects~$\{i,j\}$. 
Hence the number of pairs which may fail to be independent is at most $2nN_n=O(n^3)$.
Since the total number of ordered pairs of variables is~$N_n^2=O(n^4)$, the proportion of nonindependent pairs is~$O(n^{-1})$, and therefore tends to zero.

We may thus apply Theorem~\ref{theo:wlln}. 
Using the definition of the discrete area, it follows that the average value of the variables~$X_{i,j}$,
\[
\frac{\sum_{i<j} X_{i,j}}{n(n-1)/2} = \frac{\sum_{i<j} \#(M \cap P_i \cap P_j)}{n(n-1)/2} = \frac{\area(\W)}{n(n-1)/2}
\]
converges in probability to the expected value, which is equal to
\[
\E(X_{i,j}) = \int \! \int_{\HH \times \HH} \#(M \cap P \cap P') \, d\Prob(P) \, d\Prob(P') = \frac{\pi}{32} \, \area(M)
\]
by the equation~\eqref{eq:S3} of Theorem~\ref{theo:San}.
Hence the proposition.
\end{proof}

We conclude by observing that Theorem~\ref{theo:appendix} immediately follows from Proposition~\ref{prop:MLS} and Proposition~\ref{prop:area} by assigning the weight~$\frac{4}{n}$ to each wall of~$\W$.
\end{proof}

\forgotten

\end{document}